\documentclass{amsart}

\usepackage{amsthm}
\usepackage{url}
\usepackage{amsfonts}
\usepackage{amssymb}
\usepackage{bbm}
\usepackage{mathrsfs}
\usepackage{mathtools}
\numberwithin{equation}{section}
\usepackage{color}
\usepackage{hyperref}
\usepackage{graphicx}
\usepackage{caption}
\usepackage{subcaption}
\usepackage{inputenc}
\usepackage[linesnumbered,commentsnumbered,ruled,vlined]{algorithm2e}
\usepackage[perpage,symbol]{footmisc}

\newtheorem{theorem}{Theorem}
\newtheorem{lemma}{Lemma}
\newtheorem{corollary}{Corollary}
\theoremstyle{definition}
\newtheorem{definition}{Definition}
\theoremstyle{remark}
\newtheorem*{remark}{Remark}

\newcommand{\ee}{{\rm e}}

\newcommand{\dd}{{\rm d}}

\newcommand{\abs}[1]{\ensuremath{|#1|}}
\newcommand{\bx}{\ensuremath{\mathbf{x}}}
\newcommand{\bm}{\ensuremath{\mathbf{m}}}
\newcommand{\ba}{\ensuremath{\mathbf{a}}}
\newcommand{\bv}{\ensuremath{\mathbf{v}}}

\newcommand{\bw}{\ensuremath{\mathbf{w}}}
\newcommand{\bu}{\ensuremath{\mathbf{u}}}
\newcommand{\bM}{\ensuremath{\mathbf{M}}}

\newcommand{\bQ}{\ensuremath{\mathbf{Q}}}

\newcommand{\cJ}{\ensuremath{\mathcal{J}}}

\newcommand{\cK}{\ensuremath{\mathcal{K}}}
\newcommand{\uM}{\ensuremath{\underline{M}}}

\newcommand{\bR}{{\mathbb R}}
\newcommand{\norm}[2]{\ensuremath{\|#2\|_{#1}}}
\newcommand{\inner}[2]{\ensuremath{\left\langle #1,#2 \right\rangle}}

\newcommand{\pd}{ \partial}
\newcommand{\ud}{\, \mathrm{d} }

\newcommand{\Diff}{\mathrm{Diff}}
\newcommand{\Xcal}{\mathfrak{X}}

\title[Application of the DVDM to the EPDiff equation]
{Discrete Variational Derivative Methods for the EPDiff equation}
\author{Stig Larsson, Takayasu Matsuo, Klas Modin, Matteo Molteni}

\begin{document}

\begin{abstract}
The aim of this paper is the derivation of structure preserving schemes for the
solution of the EPDiff equation, with particular emphasis on the two dimensional
case. We develop three different schemes based on the Discrete Variational 
Derivative Method (DVDM) on a rectangular domain discretized with a 
regular, structured, orthogonal grid.

We present numerical experiments to support our claims: we investigate the 
preservation of energy and linear momenta, the reversibility, and 
the empirical convergence of the schemes. 
The quality of our schemes is finally tested by simulating the interaction of 
singular wave fronts. 
\end{abstract}

\maketitle

\tableofcontents

\section{Introduction}\label{section:introduction}

In this paper we develop and analyze numerical schemes for the \emph{EPDiff 
equation}, i.e., the nonlinear partial differential equation (PDE) given by
\begin{align}\label{eq:EPDiff}
&\frac{\partial}{\partial t} \bm  + \nabla \bm\cdot\bu + (\nabla \bu)^\top \cdot
\bm + \bm (\nabla\cdot \bu) = 0,
\end{align}
where $\bm=(m_1,\ldots,m_n)$ and $\bu=(u_1,\ldots,u_n)$ are vector-valued
functions of time $t$ and space $\bx=(x_1,\ldots,x_n)$, and 
$\bu$ is related to $\bm$ through the Helmholtz equation
\begin{equation}
	\bm = (1-\alpha^2\Delta) \bu, \quad \alpha>0.
\end{equation} 
We refer to $\bm$ as \emph{momentum} and to $\bu$ as \emph{velocity}.
Throughout the paper, the spatial domain is denoted~$\Omega$.
For simplicity, we take
\begin{equation}\label{eq:domain_rectangular}
	\Omega = [-1,1]^n = [-1,1]\times\cdots\times[-1,1]	
\end{equation}
with periodic boundary conditions.
That is, $\Omega$ is the flat $n$--torus.


The EPDiff equation arises in several contexts:
\begin{enumerate}
	\item As a geodesic equation in infinite-dimensional Riemannian geometry; it is the Poisson reduced form of a geodesic equation on the infinite-dimensional space of diffeomorphisms equipped with the right-invariant $H^1_\alpha$-metric.
	For details, see~\cite{HoMa2005,HolmBook} and references therein.

	\item As a model of shallow water wave dynamics. 
	In particular, to model ocean wave-fronts, 100-200~km in length, created by tides and currents, and regularly observed in satellite images of the earth~\cite{Holm}.
	In this context, $n=2$ and the EPDiff equation is a two-dimensional generalization of the Camassa--Holm equation~\cite{CaHo1993} for one-dimensional shallow water waves.

	\item As an inviscid, compressible version of the 
Navier--Stokes-$\alpha$ model used in fluid turbulence~\cite{Sh1998,HoMa2005}.
	In this context, $n=3$.

	\item As the governing equations in diffeomorphic shape analysis, particularly \emph{computational anatomy}, where continuous warps between medical images and shapes are computed using geodesics on the space of diffeomorphisms.
	Typically, $n=3$ in this context, but $n=2$ and $n=4$ are also of interest.
	For details, see~\cite{Shapes} and references therein.
\end{enumerate}

A consequence of the geodesic nature of the EPDiff equation~\eqref{eq:EPDiff} is a set of salient structural properties.
In particular, the equation is a Hamiltonian system with respect to a \emph{Lie--Poisson structure}~\cite[ch.13]{MaRa1999}, which implies that
\begin{enumerate}
 \item the solution flow preserve a Poisson structure, and,
 \item there are first integrals (conservation laws) given by the linear momenta and the total energy.
\end{enumerate}

In spite of its many applications, little attention has been given to numerical solution of the EPDiff equation, addressed in this paper.
Following the strategy of \emph{geometric integration} (see the monographs~\cite{SaCa1994,LeRe2004,HaLuWa2006,FeQi2010}), numerical discretizations
should preserve as much as possible of the geometric structure in order to 
produce a qualitatively correct behaviour.
It is, however, not possible to preserve both the Poisson structure and the total energy, as follows from a general result by Ge and Marsden~\cite{ZhMa1988} that applies to numerical integration of any non-integrable Hamiltonian system.
Geometric integrators for Hamiltonian systems are therefore naturally divided into two groups: those that preserve the Poisson structure and those that conserve the total energy.

For Poisson structure preserving discretization of the EPDiff equation, the only known approach is to use \emph{particle methods}~\cite{McMa2007,Particle}.
Here, one utilizes that the EPDiff equation has weak soliton solutions, where the momentum $\bm(\bx,t)$ is a finite sum of weighted Dirac delta functions (the ``particles'').
The motion of the solitons is governed by a finite dimensional, non-separable canonical Hamiltonian system, for which symplectic Runge--Kutta methods can be used.
Due to the non-smooth character of the solitons, the convergence of particle methods is slow, if at all.
It is an open problem to construct Poisson preserving discretizations of the EPDiff equation using classical numerical PDE methodology.

Another approach to the discretization, used in \cite{Holm}, is to deploy the 
compatible differencing algorithm (CDA), presented in 
\cite{CDA2,CDA1,CDA3}. 
The choice of CDA is feasible whenever the governing equations can be expressed in terms of the divergence, gradient and curl 
operators.
It is, however, not clear to what extent such methods are structure preserving.



In this paper we develop the first energy conserving geometric integrators for the EPDiff equation.
Our methods are based on the \emph{discrete variational derivative method} (DVDM), which provides a systematic approach to energy preserving discretizations of Hamiltonian PDE.
For the Camassa--Holm equation (corresponding to 1D EPDiff), DVDM is developed in~\cite{Yuto}, showing good numerical results accompanied by rigorous analysis.
The main motivation for this paper is to extend the results in~\cite{Yuto} to the higher dimensional case, thereby providing reliable numerical algorithms for exploring EPDiff and its emerging applications.



The paper is organized as follows.
In Section~\ref{section:EPDiff} we briefly recall how the explicit form of the EPDiff
equation is obtained. 
In Section~\ref{section:tools} we state some results about DVDM that shall be 
useful later on.
In Section~\ref{section:DVDM} we present the discrete variational derivative 
method and recall its application in the case of the Camassa--Holm equation.
In Sections~\ref{section:first_scheme} we present 
the first scheme, obtained by discretizing the 
energy at time $t_n$ by simply evaluating it on the grid point. 
This scheme is implicit and non-linear, so we implement it by suitable fixed point iterations. 
We prove that the scheme preserves energy and linear momenta, and give a result of 
solvability and uniqueness. 
In Section~\ref{section:second_scheme} we derive an 
explicit scheme where the non-linearity is no longer present; this is done by discretizing the energy by means of a suitable average. 
We prove conservation of linear momenta and energy, but we are no longer 
able to prove that the solution is bounded. In 
Section~\ref{section:third_scheme} we present a modification of the second 
scheme, which is now implicit and performs better in terms of stability. However 
the price we pay is that the scheme does not preserve the linear momenta 
anymore, although we prove that energy is conserved.
In Section~\ref{section:predictor-corrector} we present a predictor-corrector method based on Scheme~1 and Scheme~2. 
Finally, in Section~\ref{section:numerics}, we report our numerical tests, together with an empirical convergence analysis, a comparative performance analysis and a study of time-reversibility. 

For simplicity, the schemes are presented in a two-dimensional setting, 
but they extend naturally to higher dimensions.

\section{Background on EPDiff}\label{section:EPDiff}
The EPDiff equation~\eqref{eq:EPDiff} is an instance of a large class of non-linear PDE called \emph{Euler--Arnold equations}.
Such an equation describes geodesics on a Lie group equipped with a left or right invariant Riemannian metric.
The first example, given by Poincar\'e~\cite{Po1901}, is the equation of a free rigid body;
here the Lie group is given by $\mathrm{SO}(3)$, the group of rotation matrices, and the metric is provided by the moments of inertia.
The first infinite-dimensional example is Arnold's remarkable discovery that the Euler equations of an incompressible perfect fluid is a geodesic equation~\cite{Ar1966}; here the group is given by the volume preserving diffeomorphism of the domain occupied by the fluid.
Since Arnold's discovery, many PDE in mathematical physics are found to by Euler--Arnold equations.
For example, the KdV, Camassa--Holm, Hunter--Saxton, Landau--Lifshitz, and compressible Euler equations (see~\cite{ArKh1998,KhWe2009} for details).

In this section we briefly discuss the origin, derivation, and properties of the EPDiff equation.
For details on the derivation, see~\cite{HoMa2005,HolmBook}.
For results on well-posedness, see~\cite{Ga2009,MiPr2010,MiMu2013,Mo2015}.

\subsection{EPDiff is a Geodesic Equation}
Let $\Omega$ be the rectangular domain~\eqref{eq:domain_rectangular}, $\Diff(\Omega)$ be the group of diffeomorphisms of $\Omega$, and $\Xcal(\Omega)$ be the space of smooth vector fields on $\Omega$.
A (weak) inner product on $\Xcal(\Omega)$ is given by
\begin{equation}\label{eq:H1inner}
	\inner{\bu}{\bv}_{H^1_{\alpha}} \coloneqq \int_\Omega \sum_{i=1}^{n} (u_iv_i + \alpha^2\nabla u_i\cdot\nabla v_i) \ud \bx
\end{equation}
Integration by parts in combination with periodic boundary conditions give
\begin{equation*}
	\inner{\bu}{\bv}_{H^1_{\alpha}} = \int_\Omega \sum_{i=1}^{n} (u_i - \alpha^2\Delta u_i) v_i \ud \bx 
	= \int_\Omega (\underbrace{\bu-\alpha^2\Delta \bu}_{Q\bu})\cdot \bv \ud \bx = \inner{Q\bu}{\bv}_{L^2}.
\end{equation*}
The self-adjoint differential operator $Q\colon \Xcal(\Omega)\to\Xcal(\Omega)$ is often called \emph{inertia operator}, reflecting the finite-dimensional case of the free rigid body, where $Q$ is the moments of inertia matrix.

As we shall now see, the inner product $\inner{\cdot}{\cdot}_{H^1_\alpha}$ (or equivalently the operator $Q$) induces a Riemannian metric on the infinite-dimensional manifold $\Diff(\Omega)$ (for details about the manifold structure of $\Diff(\Omega)$, see for example~\cite{EbMa1970,Ha1982}).
Indeed, recall that a Riemannian metric on a manifold $M$ consists of a smooth field of inner products, one on each tangent space $T_pM$ for $p\in M$.
So, in our case $M=\Diff(\Omega)$ and the tangent space $T_\eta\Diff(\Omega)$ at $\eta\in\Diff(\Omega)$ consists of smooth functions $\Omega\to\mathbb{R}^{n}$.
The Riemannian metric induced by $\inner{\cdot}{\cdot}_{H^1_\alpha}$ is then given by
\begin{equation}\label{eq:H1metric}
	G_\eta(\dot\eta,\dot\eta) = \inner{\dot\eta\circ\eta^{-1}}{\dot\eta\circ\eta^{-1}}_{H^1_\alpha}.
\end{equation}
By construction, it is \emph{right-invariant}.
That is, for each $\varphi\in\Diff(\Omega)$ we have
\begin{equation*}
	G_{\eta\circ\varphi}(\dot\eta\circ\varphi,\dot\eta\circ\varphi) = G_{\eta}(\dot\eta,\dot\eta).
\end{equation*}

We are interested in deriving the equations for geodesics on $\Diff(\Omega)$ with respect to the Riemannian metric~\eqref{eq:H1metric}.
The definition of a geodesic is a curve $\gamma:[a,b]\to\Diff(\Omega)$ that extremizes the action functional
\begin{equation}
	S(\gamma) = \frac{1}{2}\int_{a}^{b} G_{\gamma(t)}(\dot\gamma(t),\dot\gamma(t))\ud t.
\end{equation}
Thus, if $\gamma(t)$ is an extremal curve, then for any variation $\gamma_\epsilon(t)\coloneqq \gamma(t) + \epsilon\delta\gamma(t)$ such that $\delta\gamma(a) = \delta\gamma(b) = 0$ we have that
\begin{equation}
	\frac{\ud}{\ud\epsilon}\Big|_{\epsilon=0} S(\gamma_\epsilon) = 0.
\end{equation}
If we carry out the differentiation we get
\begin{equation}\label{eq:var_S}
	\begin{split}
		\frac{\ud}{\ud\epsilon}\Big|_{\epsilon=0} S(\gamma_\epsilon)  = \int_a^b \inner{\frac{\ud}{\ud\epsilon}\Big|_{\epsilon=0}\dot{\gamma}_\epsilon(t)\circ\gamma_\epsilon^{-1}}{\bu(t)}_{H^1_{\alpha}} \ud t
	\end{split}
\end{equation}
where $\bu(t) = \dot\gamma(t)\circ\gamma(t)^{-1}$.
To proceed from here we need the following result, stated without a proof.

\begin{definition}\label{def:vf_commutator}
	The \emph{vector field commutator} is the skew-symmetric bi-linear form $[\cdot,\cdot]\colon\Xcal(\Omega)\times\Xcal(\Omega)\to\Xcal(\Omega)$ given by
	\begin{equation*}
		[\bv,\bu] = \nabla \bu \cdot \bv - \nabla \bv\cdot\bu,
	\end{equation*}
	or in coordinates
	\begin{equation*}
		[\bv,\bu]_i = \sum_{j=1}^{n}\left(v_j\frac{\partial u_i}{\partial x_j} - u_j\frac{\partial v_i}{\partial x_j}\right).
	\end{equation*}
\end{definition}

\begin{lemma}[Arnold~\cite{Ar1966}]\label{lem:lem_arnold}
	Let
	\begin{equation*}
		\bu_{\epsilon}(t) \coloneqq \dot\gamma_{\epsilon}(t)\circ\gamma_\epsilon(t)^{-1} .
	\end{equation*}
	Then
	\begin{equation*}
		\frac{\ud}{\ud\epsilon}\Big|_{\epsilon=0} \bu_{\epsilon}(t) = \dot\bv(t) - [\bv(t),\bu(t)].
	\end{equation*}
	where
	\begin{equation*}
		\bv(t) = \delta\gamma(t)\circ\gamma(t)^{-1}.
	\end{equation*}
\end{lemma}

For simplicity, let us from now on omit the~$t$ argument and use the momentum variable $\bm=Q\bu$ where suitable.
Continuing from \eqref{eq:var_S}, Lemma~\ref{lem:lem_arnold} in combination with 
integration by parts in both $t$ and $\bx$ then give
\begin{equation}\label{eq:var_S2}
	\begin{split}
		\frac{\ud}{\ud\epsilon}\Big|_{\epsilon=0} S(\gamma_\epsilon)
		&= 
		\int_a^b \inner{\bu}{\dot\bv + [\bv,\bu]}_{H^1_{\alpha}} \ud t  \\
		&= 
		\int_a^b \inner{\bm}{\dot\bv - [\bv,\bu]}_{L^2} \ud t \\
		&= 
		\int_a^b \int_\Omega \bm \cdot (\dot\bv - [\bv,\bu]) \ud\bx \ud t \\
		&= 
		\int_a^b \int_\Omega \left( -\dot\bm \cdot \bv - \bm\cdot [\bv,\bu]\right) \ud\bx \ud t \\
		&= 
		\int_a^b \int_\Omega \left( -\dot\bm \cdot \bv - \bm\cdot (\nabla\bu\cdot\bv) + \bm\cdot(\nabla\bv\cdot\bu) \right) \ud\bx \ud t \\
		&= 
		\int_a^b \int_\Omega \left( -\dot\bm \cdot \bv - (\nabla\bu)^\top\bm\cdot \bv - \bm(\nabla\cdot\bu)\cdot\bv - (\nabla\bm\cdot\bu)\cdot\bv \right) \ud\bx \ud t \\
		&= 
		\int_a^b \inner{-\dot\bm - (\nabla\bu)^\top\bm - \bm(\nabla\cdot\bu) - \nabla\bm\cdot\bu}{\bv}_{L^2} \ud t .
	\end{split}
\end{equation}
Since this should be valid for any path $\bv(t)$, it follows from calculus of variations that the governing equations expressed in the variables $\bu$ and $\bm$ are given by~\eqref{eq:EPDiff}.

To reconstruct the geodesic path from a solution $\bu=\bu(t)$ of~\eqref{eq:EPDiff}, one needs to solve the spatially point-wise non-autonomous ODE
\begin{equation*}
	\frac{\ud}{\ud t}\gamma = \bu\circ\gamma.
\end{equation*}
It is in this sense that the EPDiff equation is a geodesic equation.
It is often called a \emph{reduced} geodesic equation.

\subsection{EPDiff for $n=1$ is Camassa--Holm}

In the case $n=1$ the EPDiff equation~\eqref{eq:EPDiff} becomes
\begin{equation}\label{eq:EPDiff1d}
\begin{split}
&{\partial_t} m  + u \partial_x m + m \partial_x u + m
\partial_x u = 0,\\
&m = (1-\alpha^2 \partial_{xx}^2) u.
\end{split}	
\end{equation}
Substituting the second in the first equation, we get
\begin{align*}
{\partial_t} (u-\alpha^2 \partial_{xx}^2u)  + u \partial_x
(u-\alpha^2 \partial_{xx}^2u) + (u-\alpha^2 \partial_{xx}^2u) \partial_x u +
 (u-\alpha^2 \partial_{xx}^2u) \partial_x u = 0,
\end{align*}
which can be rewritten as
\begin{align*}
\partial_t u - \alpha^2 \partial_{xxt}^3 u =  u \alpha^2
\partial_{xxx}^3u  - 3 u \partial_x u + 2 \alpha^2
\partial_{xx}^2 u \partial_x u.
\end{align*}
If we take $\alpha=1$, this is the Camassa--Holm equation for shallow water waves~\cite{CaHo1993}.

\subsection{EPDiff is a Hamiltonian PDE}

The numerical methods that we use in this paper are developed for Hamiltonian PDE.
In this section we give the Hamiltonian form of the EPDiff equation.

In the language of mechanics, the inner product~\eqref{eq:H1inner} defines \emph{kinetic energy}.
Thus, a geodesic equation can be thought of as a mechanical system where only kinetic energy is present (there is no potential energy).
In terms of the velocity $\bu$ and the momentum $\bm$, the energy is given by
\begin{equation}
	H = \frac{1}{2}\inner{\bu}{\bm}_{L^2}.
\end{equation}
We think of the energy $H=H(\bm)$ as the \emph{Hamiltonian function} for our mechanical system.
It is now straightforward to check that the EPDiff equation can be written
\begin{align}\label{eq:CH_hamiltonian2}
	\frac{\partial \bm}{\partial t} &= -\Gamma_{\bm} \, \frac{\delta H}{\delta \bm},
\end{align}
where the first order differential operator $\Gamma_{\bm}$ is given by
\begin{equation}
	\Gamma_{\bm}\bv = \nabla \bm\cdot\bv + (\nabla \bv)^\top \cdot \bm + \bm (\nabla\cdot \bv).
\end{equation}
Notice that 
\begin{equation*}
	\frac{\delta H}{\delta \bm} = \bu = Q^{-1}\bm.
\end{equation*}
Associated with $H$ is the \emph{Hamiltonian density}, i.e., the scalar field $\mathcal H$ defined so that
\begin{equation}
	H(\bm) = \int_{\Omega} \mathcal H \ud\bx .
\end{equation}
Thus, the Hamiltonian density associated with~\eqref{eq:CH_hamiltonian2} is given by $\mathcal H = \frac{\bm\cdot\bu}{2}$.
Throughout the paper we use a slight abuse of notation in that $H$ denotes both the Hamiltonian function and the Hamiltonian density; which shall be clear from the context.

The operator $\Gamma_{\bm}$ defines a \emph{Lie--Poisson structure} 
and the EPDiff equation is Hamiltonian with respect to this Poisson structure.
For more information on Lie--Poisson structures we refer to~\cite{MaRa1999}.
In this paper, the following result suffices.

\begin{lemma}\label{lem:skew_symmetry}
	$\Gamma_{\bm}$ is a \emph{skew-symmetric} operator. 
	That is, for all $\bv$ and $\bu$
	\begin{equation}
		\inner{\bv}{\Gamma_{\bm}\bu}_{L^2} + \inner{\bu}{\Gamma_{\bm}\bv}_{L^2} = 0.
	\end{equation}
\end{lemma}

\begin{proof}
	From the calculation~\eqref{eq:var_S2} we see that
	\begin{equation*}
		\inner{\bv}{\Gamma_{\bm}\bu}_{L^2} = \inner{\bm}{[\bv,\bu]}_{L^2}.
	\end{equation*}
	The result now follows from skew-symmetry of the commutator $[\cdot,\cdot]$.
\end{proof}

A direct consequence of Lemma~\ref{lem:skew_symmetry} is that the energy is conserved.
Indeed, if $\bm=\bm(t)$ is a solution to equation~\eqref{eq:CH_hamiltonian2}, then
\begin{equation}
	\frac{\ud}{\ud t}H(\bm)
	= \inner{\frac{\delta H}{\delta \bm}}{\frac{\pd\bm}{\pd t}}_{L^2}
	= \inner{\frac{\delta H}{\delta \bm}}{\Gamma_{\bm}\frac{\delta H}{\delta \bm}}_{L^2} = 0,
\end{equation}
where the last equality follows from the skew-symmetry of $\Gamma_{\bm}$.
In addition, the linear momenta $\int_\Omega u_k \ud\bx$ for $k=1,\ldots,n$ are also conserved.

The key property of the numerical discretization schemes considered in this paper is that they \emph{conserve the energy}, and therefore preserve a qualitative property of the exact solution.
In most cases, the schemes also \emph{conserve the linear momenta}.
In the next section we shall briefly describe the basic notion of the discretization methods, before going deeper in detail and present the
schemes that we implement.


\begin{remark}
	The Camassa--Holm equation~\eqref{eq:EPDiff1d} has a \emph{bi-Hamiltonian structure}.
	That is, it is Hamiltonian with respect to two different Poisson structures.
	A closely related property is that the Camassa--Holm equation is \emph{integrable}---its solutions are determined by an infinite number of first integrals.
	However, the bi-Hamiltonian property of the Camassa--Holm equation is false for the EPDiff equation when $n>1$.
	In particular, the EPDiff equation~\eqref{eq:EPDiff} with $n>1$ is not integrable.
\end{remark}


\section{Background on DVDM} \label{section:preliminaries}
%
In this section we prepare our notation and some discrete tools.
We also give a short survey of the discrete variational derivative method,
based on \cite{Yuto}.

\subsection{Discrete Definitions, Identities, and Estimates}\label{section:tools}

As mentioned above, we consider here the two-dimensional case $(n=2)$ for simplicity, but we remind the reader that it is straightforward to extend to arbitrary dimensions.
Let $\cK\geq 1$ and $\cJ\geq 1$ denote the number of grid points in the $x$ and $y$ direction respectively, and let $\Delta x= 2/\cK$ and $\Delta 
y = 2/\cJ$ be the distance between the grid points.
An inner product on $\bR^{\cK \times \cJ}$ is given by
\begin{align*}
\inner{\bv}{\bw}:= \sum_{k=0}^{\cK-1}\sum_{j=0}^{\cJ-1} v_{k,j} w_{k,j} \Delta 
x 
\Delta y.
\end{align*}
The corresponding norm is 
\begin{align*}
\norm{}{\bw}^2:= \sum_{k=0}^{\cK-1}\sum_{j=0}^{\cJ-1} w_{k,j}^2 \Delta x 
\Delta y.
\end{align*}
The Hadamard product is defined as 
\begin{align*}
(\bv \cdot \bw)_{k,j} :=  v_{k,j} w_{k,j},
\end{align*}
and satisfies the following inequality
\begin{align*}
\norm{}{\bv \cdot \bw }^2 & =  \sum_{k=0}^{\cK-1}\sum_{j=0}^{\cJ-1} v_{k,j}^2 
w_{k,j}^2 \Delta x \Delta y \\
&\leq \frac{1}{\Delta x \Delta y}\Big( \sum_{k=0}^{\cK-1}\sum_{j=0}^{\cJ-1} 
v_{k,j}^2 \Delta x \Delta y \Big)\Big( \sum_{k=0}^{\cK-1}\sum_{j=0}^{\cJ-1} 
w_{k,j}^2 \Delta x \Delta y \Big),
\end{align*}
that is to say 
\begin{align}\label{eq:Hadamard_inequality}
\norm{}{\bv \cdot \bw } \leq \frac{1}{\sqrt{\Delta x \Delta y}} 
\norm{}{\bv}\norm{}{\bw}.
\end{align}

To define the schemes, we make use of the following discrete operators
\begin{align*}
\delta^{<1>}_k f_{k,j} := \frac{f_{k+1 , j} -f_{k-1 , j}   }{2\Delta x},
\end{align*}
\begin{align*}
\delta^{<1>}_j f_{k,j} := \frac{f_{k , j+1} -f_{k , j-1}   }{2\Delta y},
\end{align*}
\begin{align*}
\delta^{<2>}_{kk} f_{k,j} := \frac{f_{k+1 , j} - 2f_{k , j} + f_{k-1 , j}   
}{\Delta x^2},
\end{align*}
\begin{align*}
\delta^{<2>}_{jj} f_{k,j} := \frac{f_{k, j+1} - 2f_{k , j} + f_{k , j-1}   
}{\Delta y^2}.
\end{align*}
Sometimes we shall also use
\begin{align*}
\delta^{+}_k f_{k,j} := \frac{f_{k+1 , j} -f_{k, j}   }{\Delta x},
\end{align*}
\begin{align*}
\delta^{+}_j f_{k,j} := \frac{f_{k , j+1} -f_{k , j}   }{\Delta y},
\end{align*}
and
\begin{align*}
\delta^{-}_k f_{k,j} := \frac{f_{k , j} -f_{k-1, j}   }{\Delta x},
\end{align*}
\begin{align*}
\delta^{-}_j f_{k,j} := \frac{f_{k , j} -f_{k , j-1}   }{\Delta y},
\end{align*}
We have the following results (see \cite{Yuto}).
\begin{lemma}\label{lemma:symmetry_and_skew}
The operators defined above are such that
\begin{tiny}
\begin{align*}
&\sum_{j=0}^{\cJ-1} \sum_{k=0}^{\cK-1} f_{k,j} (\delta^{<2>}_{kk} + 
\delta^{<2>}_{jj} ) g_{k,j} \Delta x \Delta y = \sum_{j=0}^{\cJ-1} 
\sum_{k=0}^{\cK-1} g_{k,j} (\delta^{<2>}_{kk} + 
\delta^{<2>}_{jj} ) f_{k,j} \Delta x \Delta y,\\
&\sum_{j=0}^{\cJ-1} \sum_{k=0}^{\cK-1} f_{k,j} \delta^{<1>}_{k} g_{k,j} \Delta 
x \Delta y = -\sum_{j=0}^{\cJ-1} \sum_{k=0}^{\cK-1} g_{k,j} \delta^{<1>}_{k} 
f_{k,j} \Delta x \Delta y,\\
&\sum_{j=0}^{\cJ-1} \sum_{k=0}^{\cK-1} f_{k,j} \delta^{<1>}_{j} g_{k,j} \Delta 
x \Delta y = -\sum_{j=0}^{\cJ-1} \sum_{k=0}^{\cK-1} g_{k,j} \delta^{<1>}_{j} 
f_{k,j} \Delta x \Delta y,\\
\end{align*}
\end{tiny}
\end{lemma}
\begin{corollary}\label{corollary:identities_diff_op}
The following identities hold:
\begin{align*}
&\sum_{j=0}^{\cJ-1} \sum_{k=0}^{\cK-1}(\delta_{kk}^{<2>} + \delta_{jj}^{<2>}) 
f_{k,j}^{(n)}  \Delta x 
\Delta y = 0,\\
&\sum_{j=0}^{\cJ-1} \sum_{k=0}^{\cK-1} \delta_{k}^{<1>} f_{k,j}^{(n)}  \Delta x 
\Delta y = 0,\\
&\sum_{j=0}^{\cJ-1} \sum_{k=0}^{\cK-1}  \delta_{j}^{<1>} f_{k,j}^{(n)}  \Delta 
x \Delta y = 0.
\end{align*}
\end{corollary}

We can think of the operators introduced above as of matrices, acting on 
vectors in $\bR^{\cK \times \cJ}$. 
We use the notation $ D^{<1>}_{x} $ and $ D^{<1>}_{y} $ to denote 
the matrices associated to $\delta^{<1>}_k $ and $ \delta^{<1>}_j $, 
and we denote by $ D^{<2>}$ the 
matrix associated to $\delta_{kk}^{<2>} + \delta_{jj}^{<2>}$. With an abuse 
of notation we use $\bQ$ to denote both $1 -\alpha^2 D^{<2>}$ and $1 - \alpha^2
\delta_{kk}^{<2>} - \alpha^2
\delta_{jj}^{<2>}$. The following estimates on their norms hold:
\begin{lemma}\label{lemma:norms_discrete_operators}
The discrete operators fulfil the following inequalities.
\begin{align*}
&\norm{}{D^{<1>}_{x}} \leq \frac{1}{\Delta x}, &&\norm{}{D^{<1>}_{y}} \leq 
\frac{1}{\Delta y},\\
&\norm{}{D^{<2>}} \leq 4\Big( \frac{1}{\Delta y^2} + \frac{1}{\Delta 
y^2}\Big),&&\norm{}{(1 - \alpha^2 D^{<2>})^{-1}} \leq 1.
\end{align*}
\end{lemma}
\begin{proof}
We notice that
\begin{align*}
\norm{}{D^{<1>}_{x} \bv }^2 &= \sum_{k=0}^{\cK-1}\sum_{j=0}^{\cJ-1} 
(\delta^{<1>}_k v_{k,j})^2  \Delta x \Delta y \\
&= \sum_{k=0}^{\cK-1}\sum_{j=0}^{\cJ-1} \Big( \frac{ v_{k+1,j} - 
v_{k-1,j}}{2\Delta x} \Big)^2  \Delta x \Delta y \\
&\leq \frac{1}{\Delta x ^2} \sum_{k=0}^{\cK-1}\sum_{j=0}^{\cJ-1} \frac{ 
v_{k+1,j}^2 + v_{k-1,j}^2}{2}  \Delta x \Delta y \\
&= \frac{1}{\Delta x ^2} \norm{}{ \bv }^2.
\end{align*}
The same holds trivially for $D^{<1>}_{y}$.

The argument to proof the final two bounds is rather standard, and follows from 
the fact that the discrete Laplace operator has eigenvalues of the form
\begin{align*}
\lambda_{k,j} = -\frac{4}{\Delta x^2}\sin^2(\cdot) -\frac{4}{\Delta 
y^2}\sin^2(\cdot) 
\end{align*}
\end{proof}

\subsection{DVDM for Camassa--Holm}\label{section:DVDM}

Here we briefly review discrete variational derivative methods (DVDM) as developed for the Camassa--Holm equation in~\cite{Yuto}.
The interested reader, however, is invited to read~\cite{matsuo} for a general overview of DVDM.

We start from the Hamiltonian form~\eqref{eq:CH_hamiltonian2} in the case $n=1$.
The velocity $u$ and momentum $m$ are then scalar functions on the spatial domain~$[-1,1]$ (with periodic boundary conditions).
The Hamiltonian density is given by
\[
H = \frac{um}{2},
\]
and the Poisson operator $\Gamma_m$ is given by
\[
	\Gamma_m u = m\partial_x u + \partial_x (m u).
\]
Let us recall conservation of energy in this case:
\begin{align} \label{eq:Hpres}
\frac{\dd}{\dd t}\int_{\Omega} H \dd x
= \int_{\Omega} \frac{\delta H}{\delta m}m_t \dd x
= - \int_{\Omega} \frac{\delta H}{\delta m} (m\partial_x + \partial_x m) \frac{\delta H}{\delta m} \dd x
=0,
\end{align}
where, again, the last equality follows from skew-symmetry of $\Gamma_m$.

In the discrete variational derivative method,
we try to copy this structure; namely, we try to find
a discrete version of the variational derivative $\frac{\delta H}{\delta m}$ so that
it replicates the first equality of~\eqref{eq:Hpres}.
Then we define a scheme with it analogously to the Hamiltonian form~\eqref{eq:CH_hamiltonian2}.

Let us denote numerical solutions by
$M_k^{(n)}$ and $U_k^{(n)}$, where $k$ and $n$ denotes the indexes in $x$ and 
$t$ directions,
respectively.
(This definition will be overridden later for the multi-dimensional case.)
In view of the continuous definition, we define $M_k^{(n)} = (1 - \alpha^2 
\delta^{<2>}_{kk}) U_k^{(n)},$ that is,
$
M^{(n)} = \bQ U^{(n)}.
$
Let us then define a discrete version of the Hamiltonian by
\[
H_k^{(n)} = \frac{U_k^{(n)}M_k^{(n)}}{2}.
\]
To find a discrete version of the variational derivative,
we consider the difference
\begin{align}
&\frac{1}{\Delta t}
\left(
 \sum_{k=0}^{\cK-1} H_k^{(n+1)} \Delta x
 -
 \sum_{k=0}^{\cK-1} H_k^{(n)} \Delta x
\right)  \nonumber\\
&= \sum_{k=0}^{\cK-1} \left(
  \frac{M_k^{(n+\frac12)}}{2} \frac{U_k^{(n+1)} - U_k^{(n)}}{\Delta t} 
  +
  \frac{U_k^{(n+\frac12)}}{2} \frac{M_k^{(n+1)} - M_k^{(n)}}{\Delta t} 
\right) \Delta x \\ \nonumber \\
&=  \sum_{k=0}^{\cK-1} U_k^{(n+\frac12)}  \frac{M_k^{(n+1)} - M_k^{(n)}}{\Delta 
t} \Delta x. \nonumber
\end{align}
Here, we introduced an abbreviation
\[
M_k^{(n+\frac12)} := \frac{ M_k^{(n+1)} + M_k^{(n)} }{2}.
\]
We will use similar abbreviations throughout this paper.
This reveals a candidate for the discrete variational derivative, namely,
\[
\frac{\delta H}{\delta(\bM^{(n+1)} , \bM^{(n)})}_{k} := 
U_{k}^{(n+\frac12)}.
\]
Finally, we define a discrete scheme as follows.
\begin{align}\label{eq:1d_scheme}
\frac{M_k^{(n+1)} - M_k^{(n)}}{\Delta t}
=
\left( M_k^{(n+\frac12)} \delta_k^{<1>} + \delta_k^{<1>} M_k^{(n+\frac12)} 
\right)
\frac{\delta H}{\delta(\bM^{(n+1)} , \bM^{(n)})}_{k}.
\end{align}

This scheme keeps the discrete Hamiltonian conservation law:
\[
  \sum_{k=0}^{\cK-1} H_k^{(n)} \Delta x
  =
  \sum_{k=0}^{\cK-1} H_k^{(0)} \Delta x ,\qquad n=1,2,\ldots
\]
The proof proceeds in the same way as in the continuous case~\eqref{eq:Hpres}:
the first equality is guaranteed by the derivation of the discrete variational derivative;
the second is the definition of the scheme itself;
and finally, the third is from the skew-symmetry of the discrete operator
$M_k^{(n+\frac12)} \delta_k^{<1>} + \delta_k^{<1>} M_k^{(n+\frac12)}$.

In what follows, we consider extensions of~\eqref{eq:1d_scheme} to the multi-dimensional case.


\section{First Scheme: Implicit, Energy-Momentum Conserving}\label{section:first_scheme}
All the schemes presented in this and in the forthcoming sections are 
generalizations of the ones in~\cite{Yuto}. We start by 
introducing the scheme naturally obtained by taking the discrete energy to be the real 
energy evaluated at time~$t_n$.
\subsection{Derivation of the Scheme}
We define discrete quantities:
\begin{align*}
M_{i;k,j}^{(n)}\quad \mbox{and} \quad U_{i;k,j}^{(n)},
\end{align*}
where the index $i$ can either be $1$ or $2$ and refers to the component of the 
solution, the indexes $j$ and $k$ refers respectively to the $x$ and $y$ 
direction, and the index $(n)$ refers to the time instant we consider. We 
recall that $ M^{(n)}_{i} = \bQ U^{(n)}_{i},\,i=1,2 $, which 
component-wise means:
\begin{align*}
M^{(n)}_{i;k,j} = (1 - \alpha^2\delta^{<2>}_{kk} -  \alpha^2\delta^{<2>}_{jj} 
)U^{(n)}_{i;k,j}, \quad i=1,2.
\end{align*}
We define a discrete energy function given by:
\begin{align}\label{eq:disc_energy_scheme1}
H_{k,j}^{(n)} = \frac{ M_{1;k,j}^{(n)} U_{1;k,j}^{(n)} + M_{2;k,j}^{(n)} 
U_{2;k,j}^{(n)}}{2}.
\end{align}
This is the discrete counterpart to $H = \frac{\bm \cdot 
\bu}{2}$.
We have the following lemma (see 
Section~\ref{sec:om_proofs_first_equality_scheme1} for a proof).
\begin{lemma}\label{lemma:first_equality_scheme1}
For the discrete energy defined in \eqref{eq:disc_energy_scheme1} the following 
identity holds true for any $n\geq0$:
\begin{align*}
&\frac{1}{\Delta t}\Big( 
\sum_{j=0}^{\cJ-1}\sum_{k=0}^{\cK-1}H_{k,j}^{(n+1)}\Delta x 
\Delta y - \sum_{j=0}^{\cJ-1}\sum_{k=0}^{\cK-1}H_{k,j}^{(n)}\Delta x 
\Delta y\Big) \\
&= \sum_{j=0}^{\cJ-1}\sum_{k=0}^{\cK-1} \Big(
U_{1;k,j}^{(n+\frac12)} \frac{M_{1;k,j}^{(n+1)} - M_{1;k,j}^{(n)}}{\Delta t} + 
U_{2;k,j}^{(n+\frac12)}\frac{M_{2;k,j}^{(n+1)} - M_{2;k,j}^{(n)}}{\Delta 
t}\Big)\Delta x 
\Delta y.
\end{align*}
\end{lemma}
As in the one-dimensional case, it is natural to define a ``discrete 
variational derivative'' which approximates the continuous one by
\begin{align}\label{eq:DVD_expression1}
\frac{\delta H}{\delta(\bM^{(n+1)} , \bM^{(n)})}_{k,j} := 
 \begin{bmatrix}
      U_{1;k,j}^{(n+\frac12)}\\
      U_{2;k,j}^{(n+\frac12)}
     \end{bmatrix}.
\end{align}

We denote by $\widetilde{\Gamma}^{(n+\frac12)}_{\bm}$ the discrete version of 
${\Gamma}_{\bm}$, discretized at time $n+\frac12$. This is not the only 
possible way to discretize the operator; any discretization of ${\Gamma}_{\bm}$ 
which is skew-symmetric with respect to the inner product on the discrete 
spaces would be fine. However our choice of  
$\widetilde{\Gamma}^{(n+\frac12)}_{\bm}$ is quite natural for the scheme we are 
trying to develop, beside having the advantage of being symmetric. 

If we introduce the quantity $\uM_{k,j}^{(\cdot)}:= 
[M_{1;k,j}^{(\cdot)},M_{2;k,j}^{(\cdot)}]$, the scheme can be written in 
compact 
form as 
\begin{align}\label{eq:scheme_1_compact}
\frac{\uM_{k,j}^{(n+1)} - \uM_{k,j}^{(n)}}{\Delta t} = 
-\widetilde{\Gamma}^{(n+\frac12)}_{\bm}\frac{\delta H}{\delta(\bM^{(n+1)} 
, \bM^{(n)})}_{k,j}.
\end{align}

This schemes has the strong disadvantage of being non-linear, and this raises 
some serious practical issues for its usage tout-court 
as an integrator. However, by using it as a 
corrector, in the predictor-corrector scheme developed in 
Section~\ref{section:predictor-corrector}, with a variable number of 
corrections steps, 
which depends on a relative error, we can successfully and efficiently 
implement it, 
saving all its good properties (due to a fixed point argument). We refer to 
this scheme as to Scheme 1.
%

Component-wise, for the $2$-dimensional problem we are considering, it reads:
\begin{align*}
\frac{M_{1;k,j}^{(n+1)} - M^{(n)}_{1;k,j}}{\Delta t} &= - \Big[ 
\frac{M_{1;k,j}^{(n)} +
M_{1;k,j}^{(n+1)} }{2} \delta^{<1>}_k \frac{U_{1;k,j}^{(n)} + 
U_{1;k,j}^{(n+1)}}{2} \\ 
&+ \frac{M_{2;k,j}^{(n)} + M_{2;k,j}^{(n+1)}}{2} \delta^{<1>}_k 
\frac{U_{2;k,j}^{(n)} +
U_{2;k,j}^{(n+1)}}{2} \\
&+ \delta^{<1>}_k \Big( \frac{M_{1;k,j}^{(n)} +
M_{1;k,j}^{(n+1)}}{2}  \cdot \frac{U_{1;k,j}^{(n)} + U_{1;k,j}^{(n+1)}}{2} 
\Big) \\&+ 
\delta^{<1>}_j \Big( \frac{M_{1;k,j}^{(n)} +
M_{1;k,j}^{(n+1)}}{2} \cdot \frac{U_{2;k,j}^{(n)} + U_{2;k,j}^{(n+1)}}{2} \Big)
\Big]
\end{align*}
and
\begin{align*}
\frac{M_{2;k,j}^{(n+1)} - M^{(n)}_{2;k,j}}{\Delta t} &= - \Big[ 
\frac{M_{1;k,j}^{(n)} +
M_{1;k,j}^{(n+1)} }{2} \delta^{<1>}_j \frac{U_{1;k,j}^{(n)} + 
U_{1;k,j}^{(n+1)}}{2} \\ 
&+ \frac{M_{2;k,j}^{(n)} + M_{2;k,j}^{(n+1)}}{2} \delta^{<1>}_j 
\frac{U_{2;k,j}^{(n)} +
U_{2;k,j}^{(n+1)}}{2} \\
&+ \delta^{<1>}_k \Big( \frac{M_{2;k,j}^{(n)} +
M_{2;k,j}^{(n+1)}}{2}  \cdot \frac{U_{1;k,j}^{(n)} + U_{1;k,j}^{(n+1)}}{2} 
\Big) \\&+ 
\delta^{<1>}_j \Big( \frac{M_{2;k,j}^{(n)} +
M_{2;k,j}^{(n+1)}}{2} \cdot \frac{U_{2;k,j}^{(n)} + U_{2;k,j}^{(n+1)}}{2} \Big)
\Big].
\end{align*}
Here $\cdot$ denotes the Hadamard product (component-wise product).

\subsection{Conservation Properties and Solvability}
The first scheme preserve the discrete energy, as expected, and has as a 
by-produce the further advantage of preserving the linear momenta. We can indeed prove 
the following result:
\begin{theorem}\label{thm:conserved_quantities_scheme1}
Under the discrete periodic boundary conditions, the numerical solution 
produced by Scheme 1 conserves the following invariants, for each $n 
=1,2,\ldots$:
\begin{align*}
 \sum_{j=0}^{\cJ-1}\sum_{k=0}^{\cK-1} U_{\cdot,k,j}^{(n)} \Delta x 
\Delta y =  \sum_{j=0}^{\cJ-1}\sum_{k=0}^{\cK-1} U_{\cdot,k,j}^{(0)} \Delta x 
\Delta y,\\
 \sum_{j=0}^{\cJ-1}\sum_{k=0}^{\cK-1} H_{k,j}^{(n)} \Delta x 
\Delta y =  \sum_{j=0}^{\cJ-1}\sum_{k=0}^{\cK-1} H_{k,j}^{(0)} \Delta x 
\Delta y.
\end{align*}
\end{theorem}
\begin{proof}
The core of the proof is based on the skew-symmetry of 
$\widetilde{\Gamma}^{(n+\frac12)}_{\bm}$.
From Lemma~\ref{lemma:first_equality_scheme1} we know that the following holds:
\begin{align*}
&\frac{1}{\Delta t}\sum_{j=0}^{\cJ-1}\sum_{k=0}^{\cK-1}(H_{k,j}^{(n+1)} 
-H_{k,j}^{(n)} )\Delta x 
\Delta y \\
&=\sum_{j=0}^{\cJ-1}\sum_{k=0}^{\cK-1} \Big(
U_{1;k,j}^{(n+\frac12)} \frac{M_{1;k,j}^{(n+1)} - M_{1;k,j}^{(n)}}{\Delta t} + 
U_{2;k,j}^{(n+\frac12)}\frac{M_{2;k,j}^{(n+1)} - M_{2;k,j}^{(n)}}{\Delta 
t}\Big)\Delta x 
\Delta y \\
&=\sum_{j=0}^{\cJ-1}\sum_{k=0}^{\cK-1} \Big( \frac{\delta H}{\delta(\bM^{(n+1)} 
, \bM^{(n)})}_{k,j} \cdot \frac{\uM_{k,j}^{(n+1)} - \uM_{k,j}^{(n)}}{\Delta t} 
\Big) \Delta x 
\Delta y
\end{align*}
We now use the fact that our scheme is defined as in 
\eqref{eq:scheme_1_compact}, so that we can obtain the following:
\begin{align*}
&\frac{1}{\Delta t}\sum_{j=0}^{\cJ-1}\sum_{k=0}^{\cK-1}(H_{k,j}^{(n+1)} 
-H_{k,j}^{(n)} )\Delta x 
\Delta y \\
&=\sum_{j=0}^{\cJ-1}\sum_{k=0}^{\cK-1} \frac{\delta H}{\delta(\bM^{(n+1)} 
, \bM^{(n)})}_{k,j} \cdot \Big(-\widetilde{\Gamma}^{(n+\frac12)}_{\bm} 
\frac{\delta 
H}{\delta(\bM^{(n+1)} 
, \bM^{(n)})}_{k,j}\Big) \Delta x 
\Delta y,
\end{align*}
which in turn is equal to zero since $\widetilde{\Gamma}^{(n+\frac12)}_{\bm}$ 
is skew-symmetric.

To prove the first part of the claim we make use of 
Corollary~\ref{corollary:identities_diff_op}, which ensure that the claim is 
equivalent to show that 
\begin{align*}
 \sum_{j=0}^{\cJ-1}\sum_{k=0}^{\cK-1} M_{\cdot,k,j}^{(n)} \Delta x 
\Delta y =  \sum_{j=0}^{\cJ-1}\sum_{k=0}^{\cK-1} M_{\cdot,k,j}^{(0)} \Delta x 
\Delta y.
\end{align*}
We show this only for the first component, since the same argument applies to 
the second one.
\begin{align*}
&\sum_{j=0}^{\cJ-1}\sum_{k=0}^{\cK-1} 
\frac{M_{1;k,j}^{(n+1)}-M_{1;k,j}^{(n)}}{\Delta t} \Delta x \Delta y \\
&= -\sum_{j=0}^{\cJ-1}\sum_{k=0}^{\cK-1} \Big( 
M_{1;k,j}^{(n+\frac12)}\delta^{<1>}_k U_{1;k,j}^{(n +\frac12 )}  + 
M_{2;k,j}^{(n+\frac12)} \delta^{<1>}_k U_{2;k,j}^{(n+\frac12)}\\
&\qquad \qquad \qquad+ \delta^{<1>}_k ( M_{1;k,j}^{(n+\frac12)} \cdot 
 U_{1;k,j}^{(n+\frac12)} ) +
\delta^{<1>}_j (M_{1;k,j}^{(n+\frac12)} \cdot U_{2;k,j}^{(n+\frac12)} )
\Big) \Delta x \Delta y.
\end{align*}
The last two terms in the sum disappears by means of 
Corollary~\ref{corollary:identities_diff_op}. We remain therefore 
with:
\begin{align*}
-\sum_{j=0}^{\cJ-1}\sum_{k=0}^{\cK-1} \Big( 
M_{1;k,j}^{(n+\frac12)}\delta^{<1>}_k U_{1;k,j}^{(n +\frac12 )}  + 
M_{2;k,j}^{(n+\frac12)} \delta^{<1>}_k U_{2;k,j}^{(n+\frac12)} \Big)\Delta x 
\Delta y.
\end{align*}
If we apply Lemma~\ref{lemma:symmetry_and_skew} with $f_k = 
M_{1;k,j}^{(n+\frac12)}  $ and $ 
g_k = U_{1;k,j}^{(n+\frac12)}$ first, and $f_k = M_{2;k,j}^{(n+\frac12)}  $ and 
$ 
g_k = U_{2;k,j}^{(n+\frac12)}$ then, we get:
\begin{align*}
\sum_{j=0}^{\cJ-1}\sum_{k=0}^{\cK-1} \Big( 
U_{1;k,j}^{(n+\frac12)}\delta^{<1>}_k M_{1;k,j}^{(n +\frac12 )}  + 
U_{2;k,j}^{(n+\frac12)} \delta^{<1>}_k M_{2;k,j}^{(n+\frac12)} \Big)\Delta x 
\Delta y.
\end{align*}
We can now insert the expression for $M_{2;k,j}^{(n+\frac12)}$, thus obtaining 
the following new two terms:
\begin{align*}
&\sum_{j=0}^{\cJ-1}\sum_{k=0}^{\cK-1} \Big( 
U_{1;k,j}^{(n+\frac12)}\delta^{<1>}_k (1 - \alpha^2\delta^{<2>}_{kk} - 
\alpha^2\delta^{<2>}_{jj} 
) U_{1;k,j}^{(n +\frac12 )}  \Big)\Delta x 
\Delta y \\
&+ \sum_{j=0}^{\cJ-1}\sum_{k=0}^{\cK-1} \Big( 
U_{2;k,j}^{(n+\frac12)} \delta^{<1>}_k (1 - \alpha^2\delta^{<2>}_{kk} - 
\alpha^2\delta^{<2>}_{jj} 
)  U_{2;k,j}^{(n+\frac12)} \Big)\Delta x 
\Delta y.
\end{align*}
The skew-symmetry of the product operators $\delta^{<1>}_k(1 - 
\alpha^2\delta^{<2>}_{kk} 
- 
\alpha^2\delta^{<2>}_{jj} 
) $ and $\delta^{<1>}_j (1 - \alpha^2\delta^{<2>}_{kk} - 
\alpha^2\delta^{<2>}_{jj} 
) $ allows us to conclude that both terms are zero, and the lemma is thus 
proved.
\end{proof}

The following theorem ensures the unique local solvability of Scheme 1, given 
that the time step is small enough (for the proof 
see~\ref{sec:om_proofs_solvability_scheme1}):
\begin{theorem}\label{thm:solvability_scheme1}
The scheme defined by \eqref{eq:scheme_1_compact} produces a unique solution at 
time step $(n+1)$ if we choose $\Delta t$ such that
\begin{align}\label{eq:Scheme_1_condition_for_convergence}
\Delta t \leq \frac{\sqrt{2(\sqrt{5} -2)}}{5}
\sqrt{\frac{\Delta x^3 \Delta y^3}{\Delta x^2 + \Delta y^2}} \frac{1}{K},
\end{align}
where by $K$ we denote $\norm{}{M^{(n)}}$.
In the particular case in which we use the same discretization step in both the 
spatial dimensions, the condition reads
\begin{align*}
\Delta t \leq \frac{\sqrt{\sqrt{5} -2}}{5} \frac{ \Delta x^2}{r_a} .
\end{align*}
\end{theorem}
\section{Second Scheme: Explicit, Energy-Momentum Conserving}\label{section:second_scheme}
The second scheme that we want to develop is based on an alternative 
discretization of the energy, obtained by averaging. We do this in such a way 
that the scheme becomes explicit and 
multi-step.

We can immediately notice that now the the discrete 
operator at the right-hand side is different from the one introduced in the 
previous section, since the discretization is centred 
around $n$ rather than around $n+\frac12$.

\subsection{Derivation of the Scheme}
We use the same notation introduced in Section~\ref{section:first_scheme}. The 
discrete energy function is now given by
\begin{align}\label{eq:disc_energy_scheme2}
H_{k,j}^{(n+\frac12 )} = \frac{ M_{1;k,j}^{(n+1)} U_{1;k,j}^{(n)} + 
M_{1;k,j}^{(n)} 
U_{1;k,j}^{(n+1)} + M_{2;k,j}^{(n+1)} 
U_{2;k,j}^{(n)} + M_{2;k,j}^{(n)} 
U_{2;k,j}^{(n+1)}}{4}.
\end{align}
The following Lemma holds (see \ref{sec:om_proofs_first_equality_scheme2}) :
\begin{lemma}\label{lemma:first_equality_scheme2}:
For the discrete energy defined in \eqref{eq:disc_energy_scheme2} the following 
identity holds true for any $n\geq0$:
\begin{align*}
&\sum_{j=0}^{\cJ-1}\sum_{k=0}^{\cK-1}H_{k,j}^{(n +\frac12 )}\Delta x 
\Delta y - \sum_{j=0}^{\cJ-1}\sum_{k=0}^{\cK-1}H_{k,j}^{(n- \frac12 )}\Delta x 
\Delta y \\
&= \sum_{j=0}^{\cJ-1}\sum_{k=0}^{\cK-1}\Big(  
\frac{M_{1;k,j}^{(n+1)} - M_{1;k,j}^{(n-1)}}{2} U_{1;k,j}^{(n)}
+ \frac{M_{2;k,j}^{(n+1)} - M_{2;k,j}^{(n-1)}}{2} U_{2;k,j}^{(n)}
\Big)   \Delta x  
\Delta y.
\end{align*}
\end{lemma}
It follows from the Lemma that a natural way to define the discrete 
variational derivative is the following:
\begin{align}\label{eq:DVD_expression2}
\frac{\delta H}{\delta(\bM^{(n+1)} , \bM^{(n)}, \bM^{(n-1)})}_{k,j} := 
 \begin{bmatrix}
      U_{1;k,j}^{(n)}\\
      U_{2;k,j}^{(n)}
     \end{bmatrix}.
\end{align}
We denote by $\widetilde{\Gamma}^{(n)}_{\bm}$ the discrete version of 
${\Gamma}_{\bm}$ which is now centred 
around $n$ rather than around $n+\frac12$, as previously noticed. The scheme 
can be written in compact 
form as: 
\begin{align}\label{eq:scheme_2_compact}
\frac{\uM_{k,j}^{(n+1)} - \uM_{k,j}^{(n-1)}}{2\Delta t} = 
-\widetilde{\Gamma}^{(n)}_{\bm} \frac{\delta H}{\delta(\bM^{(n+1)} 
, \bM^{(n)}, \bM^{(n-1)} )}_{k,j}.
\end{align}

Component-wise the scheme \eqref{eq:scheme_2_compact} becomes
\begin{align*}
\frac{M^{(n+1)}_{1;k,j} - M^{(n-1)}_{1;k,j}}{2 \Delta t} &= -\Big[ 
M_{1;k,j}^{(n)} \cdot
(\delta^{<1>}_k
U_{1;k,j}^{(n)}) + M_{2;k,j}^{(n)} \cdot (\delta^{<1>}_k
U_{2;k,j}^{(n)}) \\ 
&\quad + \delta^{<1>}_k (M_{1;k,j}^{(n)} \cdot U_{1;k,j}^{(n)}) + 
\delta^{<1>}_j (M_{1;k,j}^{(n)} \cdot U_{2;k,j}^{(n)}) \Big], \\
\frac{M^{(n+1)}_{2;k,j} - M^{(n-1)}_{2;k,j}}{2 \Delta t} &= 
-\Big[M_{1;k,j}^{(n)} \cdot
(\delta^{<1>}_j
U_{1;k,j}^{(n)}) + M_{2;k,j}^{(n)} \cdot (\delta^{<1>}_j
U_{2;k,j}^{(n)}) \\ 
&\quad + \delta^{<1>}_k (M_{2;k,j}^{(n)} \cdot U_{1;k,j}^{(n)}) + 
\delta^{<1>}_j (M_{2;k,j}^{(n)} \cdot U_{2;k,j}^{(n)}) \Big].
\end{align*}

For the sake of implementation we derive now an explicit expression for the 
time stepping, since we have an explicit scheme. By substituting the expression 
for $M$ in the equations, the first component
becomes:
\begin{align*}
&\frac{\bQ U^{(n+1)}_{1;k,j} - \bQ U^{(n-1)}_{1;k,j}}{2 \Delta t} = \\
&\quad -\Big[ \bQ U^{(n)}_{1;k,j} \cdot
(\delta^{<1>}_{k}
U_{1;k,j}^{(n)}) + \bQ U^{(n)}_{2;k,j} \cdot
(\delta^{<1>}_{k} U_{2;k,j}^{(n)}) \\
&\quad + \delta^{<1>}_{k} (  \bQ U^{(n)}_{1;k,j} \cdot U_{1;k,j}^{(n)}) + 
\delta^{<1>}_j ( \bQ U^{(n)}_{1;k,j} \cdot 
U_{2;k,j}^{(n)})
\Big].
\end{align*}
Similarly, for the second component, the scheme reads
\begin{align*}
&\frac{\bQ  U^{(n+1)}_{2;k,j} -  \bQ U^{(n-1)}_{2;k,j}}{2 \Delta t} =\\
&\quad -\Big[  \bQ U^{(n)}_{1;k,j} \cdot
(\delta^{<1>}_j
U_{1;k,j}^{(n)}) + \bQ U^{(n)}_{2;k,j} \cdot
(\delta^{<1>}_j U_{2;k,j}^{(n)})\\
&\quad + \delta^{<1>}_k ( \bQ U^{(n)}_{2;k,j} \cdot U_{1;k,j}^{(n)})  + 
\delta^{<1>}_j ( \bQ U^{(n)}_{2;k,j} \cdot 
U_{2;k,j}^{(n)})
\Big].
\end{align*}
We can thus write explicitly the scheme, as 
\begin{align*} 
&\bQ U^{(n+1)}_{1;k,j} =  \bQ U^{(n-1)}_{1;k,j} -2\Delta t \Big[  \bQ 
U^{(n)}_{1;k,j} \cdot
(\delta^{<1>}_k U^{(n)}_{1;k,j})  + \bQ U^{(n)}_{2;k,j} 
\cdot
(\delta^{<1>}_k U^{(n)}_{2;k,j}) \\
&\qquad \qquad + \delta^{<1>}_k  \bQ U^{(n)}_{1;k,j} \cdot U^{(n)}_{1;k,j}) + 
\delta^{<1>}_j (  \bQ U^{(n)}_{1;k,j} \cdot
U^{(n)}_{2;k,j}) 
\Big],
\end{align*}
and
\begin{align*}
&\bQ  U^{(n+1)}_{2;k,j} =  \bQ U^{(n-1)}_{2;k,j}  -2\Delta t \Big[  \bQ 
U^{(n)}_{1;k,j} \cdot
(\delta^{<1>}_j
U_{1;k,j}^{(n)})  + \bQ U^{(n)}_{2;k,j} \cdot
(\delta^{<1>}_j U_{2;k,j}^{(n)}) \\
&\qquad \qquad + \delta^{<1>}_k ( \bQ U^{(n)}_{2;k,j} \cdot U_{1;k,j}^{(n)}) + 
\delta^{<1>}_j ( \bQ U^{(n)}_{2;k,j} \cdot 
U_{2;k,j}^{(n)})
\Big].
\end{align*}
Notice that at each step a
solution of a system is required. Indeed, even by solving the system having $M$
as unknown rather than $U$, at the step $(n+1)$ both $U^{(n)}$ and $M^{(n)}$ are
required in order to construct the right-hand side of the scheme. We will refer
to this scheme as to Scheme 2.

\subsection{Conservation Properties}

A result about solvability follows immediately from the last two explicit 
expressions derived for Scheme 2, which are always well defined since the 
discrete operator $\bQ$ is invertible.
\begin{theorem}
Scheme 2 has a unique numerical solution for each $n\geq2$. The results of 
existence and uniqueness does not depend on $\Delta x$, $\Delta y$, $\Delta t$.
\end{theorem}

\begin{theorem}\label{thm:conserved_quantities_scheme2}
Under the discrete periodic boundary conditions, the numerical solution 
produced by Scheme 2 conserves the following invariants, for each $n 
=1,2,\ldots$:
\begin{align*}
 \sum_{j=0}^{\cJ-1}\sum_{k=0}^{\cK-1} U_{\cdot,k,j}^{(n)} \Delta x 
\Delta y =  \sum_{j=0}^{\cJ-1}\sum_{k=0}^{\cK-1} U_{\cdot,k,j}^{(0)} \Delta x 
\Delta y,\\
 \sum_{j=0}^{\cJ-1}\sum_{k=0}^{\cK-1} H_{k,j}^{(n+\frac12)} \Delta x 
\Delta y =  \sum_{j=0}^{\cJ-1}\sum_{k=0}^{\cK-1} H_{k,j}^{(\frac12)} \Delta x 
\Delta y.
\end{align*}
\end{theorem}
\begin{proof}
The proof mimics the one of Theorem~\ref{thm:conserved_quantities_scheme1}, and 
it is based as before on the skew-symmetry of the operator 
$\widetilde{\Gamma}^{(n)}_{\bm}$.

From before we know that
\begin{align*}
&\frac{1}{\Delta t}\sum_{j=0}^{\cJ-1}\sum_{k=0}^{\cK-1}(H_{k,j}^{(n+\frac12)} 
-H_{k,j}^{(n-\frac12)} )\Delta x 
\Delta y \\
&\frac{1}{\Delta t}\sum_{j=0}^{\cJ-1}\sum_{k=0}^{\cK-1}\Big(  
\frac{M_{1;k,j}^{(n+1)} - M_{1;k,j}^{(n-1)}}{2} U_{1;k,j}^{(n)}
+ \frac{M_{2;k,j}^{(n+1)} - M_{2;k,j}^{(n-1)}}{2} U_{2;k,j}^{(n)}
\Big)   \Delta x  
\Delta y\\
&=\sum_{j=0}^{\cJ-1}\sum_{k=0}^{\cK-1} \Big( \frac{\delta H}{\delta(\bM^{(n+1)} 
, \bM^{(n)}, \bM^{(n-1)} ) }_{k,j} \cdot \frac{\uM_{k,j}^{(n+1)} - 
\uM_{k,j}^{(n-1)}}{2\Delta t} 
\Big) \Delta x 
\Delta y.
\end{align*}
We now use the fact that our scheme is defined as in 
\eqref{eq:scheme_2_compact}, so that we can obtain the following:
\begin{align*}
&\frac{1}{\Delta t}\sum_{j=0}^{\cJ-1}\sum_{k=0}^{\cK-1}(H_{k,j}^{(n+\frac12)} 
-H_{k,j}^{(n-\frac12)} )\Delta x 
\Delta y \\
&=\sum_{j=0}^{\cJ-1}\sum_{k=0}^{\cK-1}  \frac{\delta H}{\delta(\bM^{(n+1)} 
, \bM^{(n)}, \bM^{(n-1)} ) }_{k,j} \\
&\qquad \qquad \cdot 
\Big(-\widetilde{\Gamma}^{(n)}_{\bm} \frac{\delta H}{\delta(\bM^{(n+1)} 
, \bM^{(n)}, \bM^{(n-1)} ) }_{k,j}
\Big) \Delta x 
\Delta y,
\end{align*}
which in turn is equal to zero since $\widetilde{\Gamma}^{(n+\frac12)}_{\bm}$ 
is skew-symmetric. This concludes the second part of the claim.

We show the validity of the first claim only for the first component:
\begin{align*}
&\sum_{j=0}^{\cJ-1}\sum_{k=0}^{\cK-1} 
\frac{M_{1;k,j}^{(n+1)}-M_{1;k,j}^{(n-1)}}{2 \Delta t} \Delta x \Delta y \\
&=-\sum_{j=0}^{\cJ-1}\sum_{k=0}^{\cK-1} \Big( M_{1;k,j}^{(n)} \cdot
(\delta^{<1>}_k
U_{1;k,j}^{(n)}) + M_{2;k,j}^{(n)} \cdot (\delta^{<1>}_k
U_{2;k,j}^{(n)}) \\ 
&\quad + \delta^{<1>}_k (M_{1;k,j}^{(n)} \cdot U_{1;k,j}^{(n)}) + 
\delta^{<1>}_j (M_{1;k,j}^{(n)} \cdot U_{2;k,j}^{(n)})\Big) \Delta x \Delta y.
\end{align*}
As before, the last two terms in the sum disappears by means of the 
skew-symmetry of the operators $\delta^{<1>}_j $ and $\delta^{<1>}_k$.
We remain therefore with:
\begin{align*}
&-\sum_{j=0}^{\cJ-1}\sum_{k=0}^{\cK-1} \Big( M_{1;k,j}^{(n)} \cdot
(\delta^{<1>}_k
U_{1;k,j}^{(n)}) + M_{2;k,j}^{(n)} \cdot (\delta^{<1>}_k
U_{2;k,j}^{(n)}) \Big) \Delta x \Delta y,
\end{align*}
which, as in Theorem~\ref{thm:conserved_quantities_scheme1}, reduces to
\begin{align*}
&-\sum_{j=0}^{\cJ-1}\sum_{k=0}^{\cK-1} \Big( U_{1;k,j}^{(n)} \cdot
(\delta^{<1>}_k
M_{1;k,j}^{(n)}) + U_{2;k,j}^{(n)} \cdot (\delta^{<1>}_k
M_{2;k,j}^{(n)}) \Big) \Delta x \Delta y.
\end{align*}
By using the expression for $M_{2;k,j}^{(n+\frac12)}$ we obtain:
\begin{align*}
&-\sum_{j=0}^{\cJ-1}\sum_{k=0}^{\cK-1} \Big( U_{1;k,j}^{(n)} \cdot
(\delta^{<1>}_k (1 - \alpha^2\delta^{<2>}_{kk} - \alpha^2\delta^{<2>}_{jj} ) 
U_{1;k,j}^{(n)}) \Big) 
\Delta x \Delta y\\
&\quad -\sum_{j=0}^{\cJ-1}\sum_{k=0}^{\cK-1} \Big( U_{2;k,j}^{(n)} \cdot 
(\delta^{<1>}_k (1 - \alpha^2\delta^{<2>}_{kk} - \alpha^2\delta^{<2>}_{jj} ) 
U_{2;k,j}^{(n)})
\Big) \Delta x \Delta y.
\end{align*}
The skew-symmetry of the operators involved yields allows us to conclude that 
the above quantity is equal to zero. Having 
$\sum_{j=0}^{\cJ-1}\sum_{k=0}^{\cK-1} 
\frac{M_{1;k,j}^{(n+1)}-M_{1;k,j}^{(n-1)}}{2 \Delta t} \Delta x \Delta y = 0$ 
implies that $ \sum_{j=0}^{\cJ-1}\sum_{k=0}^{\cK-1} 
U_{1;k,j}^{(n+1)} \Delta x \Delta y = \sum_{j=0}^{\cJ-1}\sum_{k=0}^{\cK-1} 
U_{1;k,j}^{(n-1)} \Delta x \Delta y.$ If the first step of the scheme is 
conservative as well, that is, if  $ \sum_{j=0}^{\cJ-1}\sum_{k=0}^{\cK-1} 
U_{1;k,j}^{(1)} \Delta x \Delta y = \sum_{j=0}^{\cJ-1}\sum_{k=0}^{\cK-1} 
U_{1;k,j}^{(0)} \Delta x \Delta y$, then the claim follows. 
\end{proof}

\section{Third Scheme: Linearly Implicit, Energy Conserving}\label{section:third_scheme}
The third scheme is based on the same idea used to derive the second scheme, 
that is to say on discretizing the energy by averaging. The discrete energy 
used in this case gives also rise to a linear multi-step scheme, but 
in this case the scheme is implicit.

As it happens for Scheme 2, the discretization of the differential operator is 
now centred around $n$ rather than around $n+\frac12$.

\subsection{Derivation of the Scheme}
We define a discrete energy function given by:
\begin{align}\label{eq:disc_energy_scheme3}
H_{k,j}^{(n+\frac12 )} = \frac{ M_{1;k,j}^{(n+1)} U_{1;k,j}^{(n+1)} + 
M_{1;k,j}^{(n)} 
U_{1;k,j}^{(n)} + M_{2;k,j}^{(n+1)} 
U_{2;k,j}^{(n+1)} + M_{2;k,j}^{(n)} 
U_{2;k,j}^{(n)}}{4}.
\end{align}
The following Lemma holds (see \ref{sec:om_proofs_first_equality_scheme3}) :
\begin{lemma}\label{lemma:first_equality_scheme3}:
For the discrete energy defined in \eqref{eq:disc_energy_scheme3} the following 
identity holds true for any $n\geq0$:
\begin{align*}
&\sum_{j=0}^{\cJ-1}\sum_{k=0}^{\cK-1}H_{k,j}^{(n +\frac12 )}\Delta x 
\Delta y - \sum_{j=0}^{\cJ-1}\sum_{k=0}^{\cK-1}H_{k,j}^{(n- \frac12 )}\Delta x 
\Delta y \\
&=\sum_{j=0}^{\cJ-1}\sum_{k=0}^{\cK-1}\Big(  
\frac{M_{1;k,j}^{(n+1)} - M_{1;k,j}^{(n-1)}}{2} 
\frac{ U_{1;k,j}^{(n+1)} + U_{1;k,j}^{(n-1)} }{2}\\
&\qquad \qquad+ \frac{M_{2;k,j}^{(n+1)} - M_{2;k,j}^{(n-1)}}{2}
\frac{ U_{2;k,j}^{(n+1)} + U_{2;k,j}^{(n-1)} }{2}
\Big)   \Delta x  
\Delta y.
\end{align*}
\end{lemma}
We therefore define the following discrete variational derivative:
\begin{align}\label{eq:DVD_expression3}
\frac{\delta H}{\delta(\bM^{(n+1)} , \bM^{(n)}, \bM^{(n-1)})}_{k,j} := 
 \begin{bmatrix}
      \frac{U_{1;k,j}^{(n+1)} + U_{1;k,j}^{(n-1)}  }{2}\\
      \frac{U_{2;k,j}^{(n+1)} + U_{2;k,j}^{(n-1)}  }{2}
     \end{bmatrix}.
\end{align}

The scheme is defined component-wise as:
\begin{align*}
&\frac{M^{(n+1)}_{1;k,j} - M^{(n-1)}_{1;k,j}}{2 \Delta t} = \\
&\quad -\Big[ M_{1;k,j}^{(n)} \cdot
(\delta^{<1>}_k \frac{U^{(n+1)}_{1;k,j} + U^{(n-1)}_{1;k,j}}{2}) + 
M_{2;k,j}^{(n)} \cdot (\delta^{<1>}_k \frac{U^{(n+1)}_{2;k,j} + 
U^{(n-1)}_{2;k,j}}{2}) \\ 
&\quad + \delta^{<1>}_{1;k,j} (M_{1;k,j}^{(n)} \cdot
\frac{U^{(n+1)}_k + U^{(n-1)}_{1;k,j}}{2}) 
+\delta^{<1>}_j (M_{1;k,j}^{(n)} \cdot 
\frac{U^{(n+1)}_{2;k,j} + U^{(n-1)}_{2;k,j}}{2}) \Big],
\end{align*}
and 
\begin{align*}
&\frac{M^{(n+1)}_{2;k,j} - M^{(n-1)}_{2;k,j}}{2 \Delta t} = \\
&\quad -\Big[M_{1;k,j}^{(n)} \cdot (\delta^{<1>}_j \frac{U^{(n+1)}_{1;k,j} 
+ U^{(n-1)}_{1;k,j}}{2}) + M_{2;k,j}^{(n)} \cdot (\delta^{<1>}_j
\frac{U^{(n+1)}_{2;k,j} + U^{(n-1)}_{2;k,j}}{2}) \\ 
&\quad + \delta^{<1>}_k (M_{2;k,j}^{(n)} \cdot
\frac{U^{(n+1)}_{1;k,j} + U^{(n-1)}_{1;k,j}}{2}) +\delta^{<1>}_j 
(M_{2;k,j}^{(n)} \cdot \frac{U^{(n+1)}_{2;k,j} + U^{(n-1)}_{2;k,j}}{2}) \Big].
\end{align*}

We focus for a moment only on the
first component. By substituting the expression for $M$ in the equation and by 
further simplifying, we get: 
\begin{align*} 
& \bQ U^{(n+1)}_{1;k,j} - \bQ U^{(n-1)}_{1;k,j} = \\
&\ -\Delta t \Big[ \bQ U^{(n)}_{1;k,j} \cdot
\Big(\delta^{<1>}_k (U^{(n+1)}_{1;k,j} + U^{(n-1)}_{1;k,j}) \Big) + \bQ 
U^{(n)}_{2;k,j} \cdot
\Big(\delta^{<1>}_k (U^{(n+1)}_{2;k,j} + U^{(n-1)}_{2;k,j}) \Big) \\
&\quad +  \delta^{<1>}_k \Big(\bQ U^{(n)}_{1;k,j} \cdot (U^{(n+1)}_{1;k,j} 
+ U^{(n-1)}_{1;k,j})\Big) + \delta^{<1>}_j \Big( \bQ U^{(n)}_{1;k,j} \cdot
(U^{(n+1)}_{2;k,j} + U^{(n-1)}_{2;k,j})\Big) 
\Big].
\end{align*}
By moving at the left-hand side the $n+1$-indexed terms, the following
expression for the left-hand side at the first component, $LHS_1$, is achieved
\begin{align*} 
LHS_1&:= \bQ U^{(n+1)}_{1;k,j}  + \Delta t \Big[ \bQ U^{(n)}_{1;k,j} \cdot
(\delta^{<1>}_k U^{(n+1)}_{1;k,j} ) + \bQ U^{(n)}_{2;k,j}
\cdot(\delta^{<1>}_k U^{(n+1)}_{2;k,j} ) \\ 
&+ \delta^{<1>}_k ( \bQ U^{(n)}_{1;k,j} \cdot U^{(n+1)}_{1;k,j} ) + 
\delta^{<1>}_j (
\bQ U^{(n)}_{1;k,j} \cdot U^{(n+1)}_{2;k,j}) \Big].
\end{align*}
Similarly, the right-hand side becomes:
\begin{align*} 
RHS_1&:= \bQ U^{(n-1)}_{1;k,j} - \Delta t \Big[ \bQ U^{(n)}_{1;k,j} \cdot
(\delta^{<1>}_k U^{(n-1)}_{1;k,j} )+ \bQ U^{(n)}_{2;k,j} \cdot
(\delta^{<1>}_k U^{(n-1)}_{2;k,j})  \\ 
&+\delta^{<1>}_k ( \bQ U^{(n)}_{1;k,j} \cdot U^{(n-1)}_{1;k,j})
 + \delta^{<1>}_j ( \bQ U^{(n)}_{1;k,j} \cdot
U^{(n-1)}_{2;k,j}) \Big],
\end{align*}
We can use the same kind of techniques for the second component:
\begin{align*}
&\bQ U^{(n+1)}_{2;k,j} - \bQ U^{(n-1)}_{2;k,j} = \\
&\ -\Delta t\ \Big[\bQ 
U_{1;k,j}^{(n)} \cdot
\Big(\delta^{<1>}_j
(U^{(n+1)}_{1;k,j} - U^{(n-1)}_{1;k,j}) \Big) + \bQ U_{2;k,j}^{(n)} \cdot 
\Big(\delta^{<1>}_j
(U^{(n+1)}_{2;k,j} - U^{(n-1)}_{2;k,j}) \Big) \\ 
&\quad + \delta^{<1>}_k \Big(\bQ U_{2;k,j}^{(n)} \cdot
(U^{(n+1)}_{1;k,j} - U^{(n-1)}_{1;k,j}) \Big) +\delta^{<1>}_k \Big(\bQ 
U_{2;k,j}^{(n)} \cdot (U^{(n+1)}_{2;k,j} - U^{(n-1)}_{2;k,j}) \Big)
\Big].
\end{align*}
We thus obtain:
\begin{align*} 
LHS_2&:= \bQ U^{(n+1)}_{2;k,j}  + \Delta t \Big[ \bQ U^{(n)}_{1;k,j} \cdot
(\delta^{<1>}_j U^{(n+1)}_{1;k,j} ) + \bQ U^{(n)}_{2;k,j}
\cdot(\delta^{<1>}_j U^{(n+1)}_{2;k,j} ) \\ 
&+ \delta^{<1>}_k ( \bQ U^{(n)}_{2;k,j} \cdot U^{(n+1)}_{1;k,j} ) + 
\delta^{<1>}_j (
\bQ U^{(n)}_{2;k,j}
\cdot
U^{(n+1)}_{2;k,j}) 
\Big].
\end{align*}
and 
\begin{align*} 
RHS_2&:= \bQ U^{(n-1)}_{2;k,j} - \Delta t \Big[ \bQ U^{(n)}_{1;k,j} \cdot
(\delta^{<1>}_j U^{(n-1)}_{1;k,j} )+ \bQ U^{(n)}_{2;k,j} \cdot
(\delta^{<1>}_j U^{(n-1)}_{2;k,j})  \\ 
&\quad +\delta^{<1>}_k ( \bQ U^{(n)}_{2;k,j} \cdot U^{(n-1)}_{1;k,j})
 + \delta^{<1>}_j ( \bQ U^{(n)}_{2;k,j} \cdot
U^{(n-1)}_{2;k,j}) \Big].
\end{align*}

\subsection{Conservation properties}
This third scheme, although formally similar to the second one, present the 
disadvantage of not preserving the linear momenta. In the next theorem we prove 
indeed 
how the conservation of the energy occurs and, sketch why 
conservation of linear momenta fails.

\begin{theorem}\label{thm:conserved_quantities_scheme3}
Under the discrete periodic boundary conditions, the numerical solution 
produced by Scheme 3 conserves the following invariant, for each $n 
=1,2,\ldots$:
\begin{align*}
 \sum_{j=0}^{\cJ-1}\sum_{k=0}^{\cK-1} H_{k,j}^{(n+\frac12)} \Delta x 
\Delta y =  \sum_{j=0}^{\cJ-1}\sum_{k=0}^{\cK-1} H_{k,j}^{(\frac12)} \Delta x 
\Delta y.
\end{align*}
\end{theorem}
\begin{proof}
It follows the very same lines of the previous ones
\end{proof}
\begin{remark}
We want to remark that, in general, we do not expect conservation of the 
discrete momentum
\begin{align*}
 \sum_{j=0}^{\cJ-1}\sum_{k=0}^{\cK-1} U_{\cdot,k,j}^{(n)} \Delta x 
\Delta y,
\end{align*}
as it happened for the other two schemes. The proof of conservation of this 
quantity fails since we find ourselves to evaluate, for example, the quantity
\begin{align*}
&-\sum_{j=0}^{\cJ-1}\sum_{k=0}^{\cK-1} \Big[ U_{1;k,j}^{(n)} \cdot
\delta^{<1>}_k (1 - \alpha^2\delta^{<2>}_{kk} - \alpha^2\delta^{<2>}_{jj} ) 
\Big( \frac{U_{1;k,j}^{(n+1)} + U_{1;k,j}^{(n-1)}}{2} \Big) 
\Big] 
\Delta x \Delta y\\
&\quad -\sum_{j=0}^{\cJ-1}\sum_{k=0}^{\cK-1} \Big[ U_{2;k,j}^{(n)} \cdot 
\delta^{<1>}_k (1 - \alpha^2\delta^{<2>}_{kk} - \alpha^2\delta^{<2>}_{jj} ) 
\Big( \frac{U_{2;k,j}^{(n+1)} + U_{2;k,j}^{(n-1)}}{2} \Big) 
\Big] \Delta x \Delta y,
\end{align*}
and we can no longer rely on the ``skew-symmetry'' trick.
\end{remark}

\section{Predictor-Corrector Method}\label{section:predictor-corrector}

As pointed out in Section~\ref{section:first_scheme}, the main issue in 
implementing Scheme 1 is the presence of a non linear term. 
An option worth considering in order to implement the scheme, is the use of a 
fixed-point iteration algorithm, namely of a predictor-corrector routine, based 
on the schemes already introduced. We start by using a quick 
predictor routine to approximate $U^{(n+1),p}$, which is 
then used as initial guess for the fixed-point iteration which linearises 
Scheme 1. We can thus compute $U^{(n+1),c}$ from the now linearised Scheme 
1, which is then used as new initial guess. We thus produce a series of 
values for $U^{(n+1),c}$ tending to $U^{(n+1)}$. 

Although different choices for the predictor are possible, we found 
it convenient to use Scheme 2, which has the lowest computational cost per iteration. 
The predictor-corrector method we use is listed in 
Algorithm~\ref{alg:predictor_corrector}.

\begin{algorithm}
	\DontPrintSemicolon
	\SetKwData{MaxIt}{max_iter}
	\KwData{Initial condition $U^{(0)}\in \bR^{2\times \cK \times \cJ}$}
	\KwResult{Discrete solution $U^{(1)},\ldots,U^{(N)}\in \bR^{2\times \cK \times \cJ}$}
	\BlankLine
	Produce $M^{(1)}$ and $U^{(1)}$ by a one-step method (e.g.\ Runge-Kutta)\;
	\For{$n\leftarrow 2$ \KwTo $N$}{
		\tcp*[h]{Predictor}
		{\tiny
		\begin{align*}
			\frac{M^{(n+1),p}_{\{1,2\};k,j} - 
			M^{(n-1)}_{\{1,2\};k,j}}{2 \Delta t} &= -\Big( 
			M_{1;k,j}^{(n)} \cdot
			(\delta^{<1>}_{\{k,j\}}
			U_{1;k,j}^{(n)}) + M_{2;k,j}^{(n)} \cdot 
			(\delta^{<1>}_{\{k,j\}} U_{2;k,j}^{(n)}) \\ 
			&\quad + \delta^{<1>}_k (M_{\{1,2\};k,j}^{(n)} \cdot 
			U_{1;k,j}^{(n)}) + 
			\delta^{<1>}_j (M_{\{1,2\};k,j}^{(n)} \cdot 
			U_{2;k,j}^{(n)}) \Big) \\
		\end{align*}
		}
		\;
		\For{$i\leftarrow 1$ \KwTo number of corrector iterations}{
			\tcp*[h]{Corrector}
			{\tiny
			\begin{align*}
		    \frac{M_{ \{1,2 \} ;k,j}^{(n+1),c} - 
		    M^{(n)}_{\{1,2 \};k,j}}{\Delta t} &= - \frac{1}{2}\Big( 
		    ( M_{1;k,j}^{(n)} + M_{1;k,j}^{(n+1),p}) 
		    \delta^{<1>}_{\{k,j \}} 
		    (U_{1;k,j}^{(n)} + U_{1;k,j}^{(n+1),p}) \\ 
		    &+ (M_{2;k,j}^{(n)} + M_{2;k,j}^{(n+1),p}) 
		   \delta^{<1>}_{\{k,j \}}
		   (U_{2;k,j}^{(n)} + U_{2;k,j}^{(n+1),p}) \\
		  &+ \delta^{<1>}_k 
		  ((M_{\{1,2 \};k,j}^{(n)} + M_{\{1,2 \};k,j}^{(n+1),p})  \cdot 
		  (U_{1;k,j}^{(n)} + U_{1;k,j}^{(n+1),p}) ) 
		  \\&+\delta^{<1>}_j 
		( (M_{\{1,2 \};k,j}^{(n)} + M_{\{1,2 \};k,j}^{(n+1),p}) \cdot 
		(U_{2;k,j}^{(n)} + U_{2;k,j}^{(n+1),p}) )
				\Big) 
			\end{align*}
			}
			\;
			\tcp{Corrector iteration update}
				$M^{(n+1),p} = M^{(n+1),c}$ \;
		}
		\tcp{Time-step update}
		$M^{(n+1)} = M^{(n+1),p}$ \;
	}
	\caption{Predictor-corrector method for Scheme~1}\label{alg:predictor_corrector}
\end{algorithm}

It is worth noticing that the number of iteration of the corrector might be 
variable, by introducing a control over the relative residual. Although this 
might be be a good choice to test the mathematical properties of Scheme 
1, for concrete purposes one would like to keep the number of corrector iterations as low as possible, so 
that the overall cost of the method is comparable with the cost of Scheme 2 
and 3, although the conservation property are not ensured anymore. In 
Section~\ref{section:numerics} we test both predictor-corrector implementation 
of 
Scheme 1 with fix and with variable number of iterations. For more 
information about this topic, we refer the reader to \cite{Daisuke}, where 
predictor-corrector schemes based on the DVDM are investigated more in detail.


\section{Numerical Results}\label{section:numerics}
We devote this section to the presentation of numerical results. We first test 
the quality of our schemes, by empirically verifying all the properties that 
we 
discussed in the previous section, and then use our scheme to solve problems 
where singular wave fronts interact with each other, in the spirit of what done 
in \cite{Holm} and \cite{Particle}.


\subsection{Conservation Properties}
The first tests presented in this section are about the empirical verification 
of 
the conservation properties. We 
choose to test our schemes with a very simple initial profile given by the 
following expression:
\begin{align*}
&u_1(t=0,\bx) = 0.5((2+\pi^2)+\sin(\pi x_1)),\\
&u_2(t=0,\bx) = 0.
\end{align*}
The reason to do so is that we can run the simulation for relatively large 
values of the final time $T$, in this particular case equal to $50$, and expect 
the second component $u_2$ to remain zero throughout the simulation.
The factor $0.5$ comes from a rescaling of the problem, while the vertical shift is 
introduced for the sake of visualization of $|U|$.

We fix the ratio between temporal and spatial discretization so that $\Delta t 
= \Delta x^2$, and we work on the domain $\Omega=[-1,1]\times [-1,1]$. We 
choose a spatial 
discretization with $20 \times 20$ grid points. The coarseness of the spatial 
grid does not play any role in the conservation of the energy and of the 
linear momenta, and we therefore do not lose any generality with this choice.

We include the results obtained by means of an explicit fourth order 
Runge-Kutta scheme as a possible term of comparison. Scheme 1 is implemented in 
the 
predictor-corrector routine described in the previous section, with a variable 
number of corrector routines until a relative tolerance of $1e-14$ is reached. 
An alternative version of this scheme, implemented with a fixed number of 
correction routines is also included.

For the multistep schemes, the first step is performed by means of Scheme 1 
solved through MATLAB's built-in function {\it fsolve}. The conservation is 
measured in terms of total variation and of the discrete 
$\norm{\infty}{\cdot}$-norm. 
We report in Table~\ref{tab:energy} the results about conservation of the 
energy and in Table~\ref{tab:momentum1} and \ref{tab:momentum2} the results 
about conservation of the linear momenta. In Figure~\ref{fig:Energies} we show 
the evolution of $|H^{(n)}-H^{(0)}|$ as a function of $n\Delta t$, where 
$H^{(n)}$ denotes for each scheme the corresponding total discrete energy at 
time-step $n$. 


\begin{figure}[h!]
        \centering
                \includegraphics[width=\textwidth]{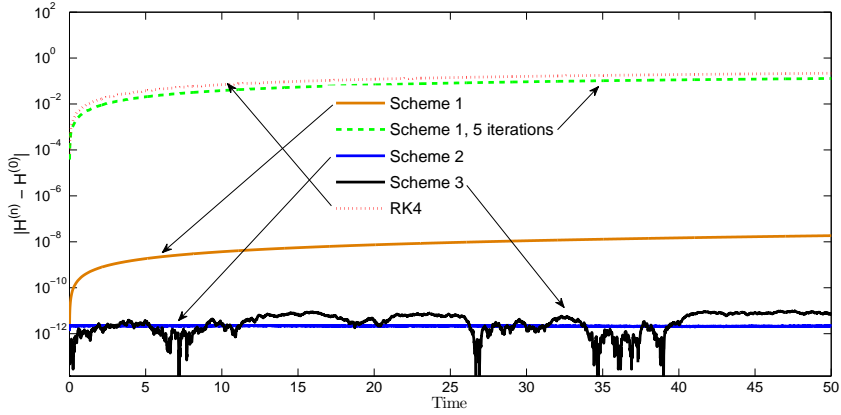}
\caption{Evolution of $| H^{(n)} - H^{(0)}|$ as a function of $n\Delta t$ }
\label{fig:Energies}
\end{figure}

\begin{table}[htbp]
\caption{Conservation of the discrete energy}
\label{tab:energy} 
\begin{center}
  \begin{tabular}{ | l || c | c | }
    \hline
 & Total Variation & $\norm{\infty}{\cdot}$ \\ 
\hline
Scheme 1 
    & $1.8529\cdot 10^{-8}$ & $1.8529\cdot 10^{-8}$ \\ 
    \hline
Scheme 1\footnotemark[1]
    & $0.1290$ & $0.1290$ \\ 
    \hline
Scheme 2
    & $2.1306\cdot 10^{-10}$ & $2.3448\cdot 10^{-12}$ \\
    \hline
Scheme 3 
    & $5.8814\cdot 10^{-10}$ & $1.3628\cdot 10^{-11}$  \\ 
    \hline
RK4 
    & $0.2185$ & $0.2185$ \\
\hline
  \end{tabular}
\end{center}
\end{table}

\begin{table}[htbp]
\caption{Conservation of the linear momentum in $x$-direction}
\label{tab:momentum1} 
\begin{center}
  \begin{tabular}{ | l || c | c | }
    \hline
 & Total Variation &$\norm{\infty}{\cdot}$\\ 
\hline
Scheme 1 
    & $3.1130\cdot 10^{-9}$ & $3.1127\cdot 10^{-9}$ \\ 
    \hline
Scheme 1\footnotemark[1]
    & $3.1118\cdot 10^{-9}$ & $3.1115\cdot 10^{-9}$ \\ 
    \hline
Scheme 2
    & $2.6427\cdot 10^{-9}$ & $1.2150\cdot 10^{-12}$ \\
    \hline
Scheme 3 
    & $5.7786$ & $ 0.0180$  \\ 
    \hline
RK4 
    & $1.7469\cdot 10^{-11}$ & $7.8160\cdot 10^{-14}$ \\
\hline
  \end{tabular}
\end{center}
\end{table}

\begin{table}[htbp]
\caption{Conservation of the linear momentum in $y$-direction}
\label{tab:momentum2} 
\begin{center}
  \begin{tabular}{ | l || c | c | }
    \hline
 & Total Variation & $\norm{\infty}{\cdot}$\\ 
\hline
Scheme 1 
    & $2.6557\cdot 10^{-16}$ & $8.0264\cdot 10^{-17}$ \\ 
    \hline
Scheme 1\footnotemark[1]
    & $3.1510\cdot 10^{-16}$ & $1.1311\cdot 10^{-16}$ \\ 
    \hline
Scheme 2
    & $1.7778\cdot 10^{-16}$ & $1.4135\cdot 10^{-17}$ \\
    \hline
Scheme 3 
    & $8.6174\cdot 10^{-10}$ & $8.7079\cdot 10^{-11}$  \\ 
    \hline
RK4 
    & $2.6717\cdot 10^{-19}$ & $2.2399\cdot 10^{-20}$ \\
\hline
  \end{tabular}
\end{center}
\end{table}

\footnotetext[1]{With a fixed number of $5$ corrector iterations.} 


\subsection{Interaction of Singular Waves Fronts} 

The initial data is modelled on the basis of the singular wave fronts described 
in \cite{Holm}. We focus in particular on the first series of
numerical experience, where the authors consider a collection of wave profiles
that have constant magnitude along a direction and have a cross section with
Gaussian profile. We consider initial profiles such
that $ \abs{U} = \ee^{-\frac{\abs{\bx}}{\sigma}} $ for various $\sigma > 0$. 
The initial profile is 
smooth but close to singular, and it has bounded support.
In order to produce such profiles, we use a suitable smooth cut-off and we 
adopt a strategy similar to the one presented in \cite{Particle}.

With this kind of configuration, it is
meaningful to consider short times for the evolution of the system, since we 
do not want the wave front to hit the boundary. For
most of our simulation a final time of at most $T=1.5$ suffices while for
some tests, smaller times such as $T=1.25$ or even $T=0.8$ might be more 
suitable.

To be consistent with the references~\cite{Particle,Holm},
we test all our initial profiles on a grid with
$1025 \times 1025$ points. 
Some other tests, such as the reversibility tests, 
are instead conducted on
the coarser grid $200 \times 200$, since the wave
profiles will be qualitatively close enough to their counterparts on finer
grids and since the outcome of our analysis will not be affected by the
discretization chosen.

It is worth to preliminary remark that the profile is 
stable for
$\alpha = \sigma$, with a stable peakon curve segment that retains its
integrity. For $\alpha < \sigma$ the profile is unstable and the peakon segment
breaks into narrower curved peakons, {\it contact curves}, each of which of 
width
$\alpha$.

All the numerical results presented in the rest of the manuscript are obtained
by using the initial profiles depicted in Figure~\ref{fig:initial_profiles}.

In Figure~\ref{fig:plate} the profile has velocity parallel to
the outward normal vector, oriented to the right. In Figure~\ref{fig:parallel}
the velocity field has the same orientation, and the leftmost wave profile has
twice the magnitude of the rightmost one. Finally, in Figure~\ref{fig:star}
each of the wave fronts has velocity parallel to its outward normal vector, and
all of them are oriented towards the same direction, i.e., clock-wise.

The qualitative behaviour of the schemes presented in this manuscript are all 
similar, and therefore we only present the results for one of them, namely for 
Scheme 2. We notice in 
Figure~\ref{fig:ev_scheme2_alpha=sigma,sigma/2} - 
\ref{fig:ev_scheme2_alpha=sigma/4_star}
 how the evolution of the solution produced with our schemes is consistent with 
what already observed in \cite{Holm}.



\begin{figure}
        \centering
        \begin{subfigure}[b]{0.49\textwidth}
                \includegraphics[width=\textwidth]{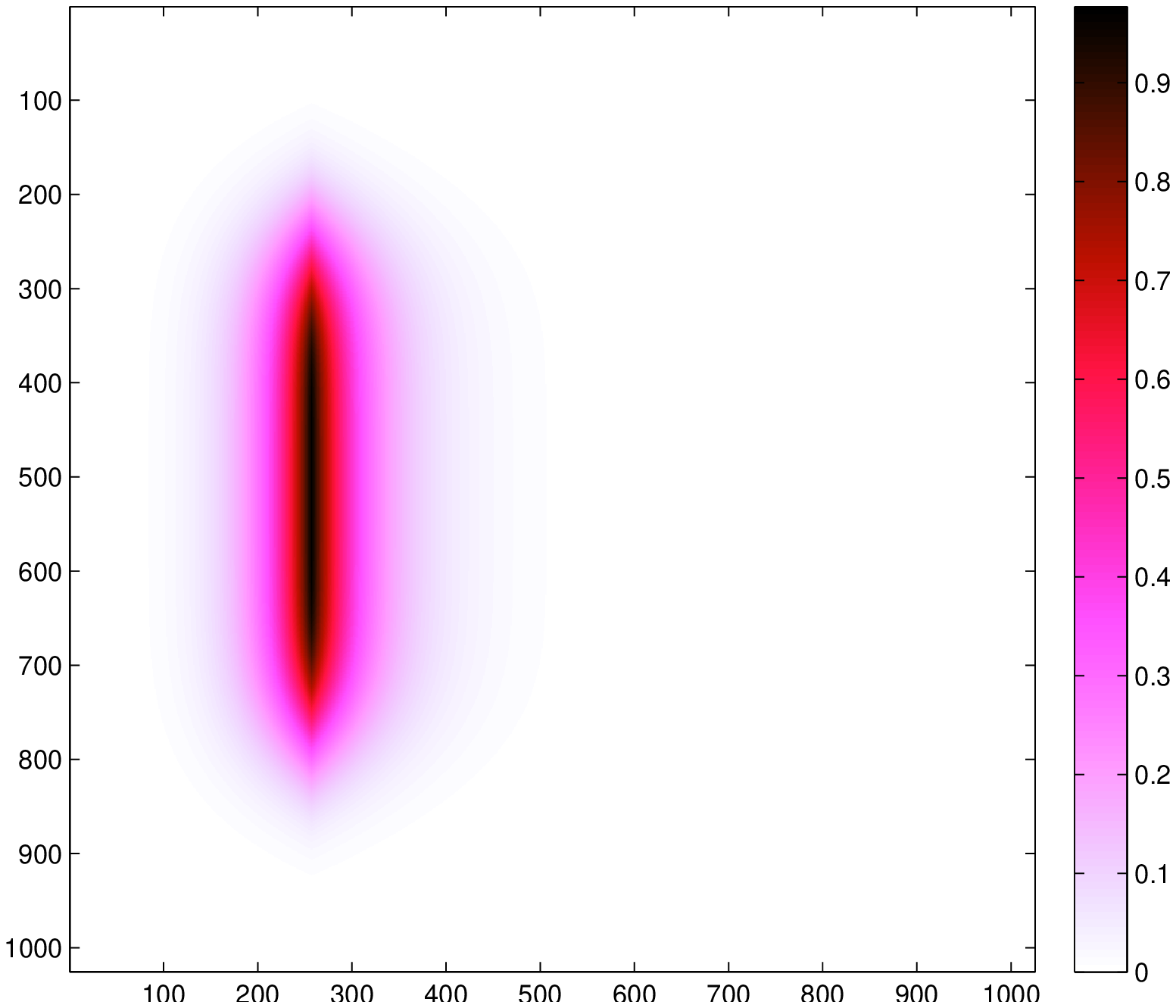}
                \caption{``Plate'' profile}
                \label{fig:plate}
        \end{subfigure}
        \begin{subfigure}[b]{0.49\textwidth}
                \includegraphics[width=\textwidth]{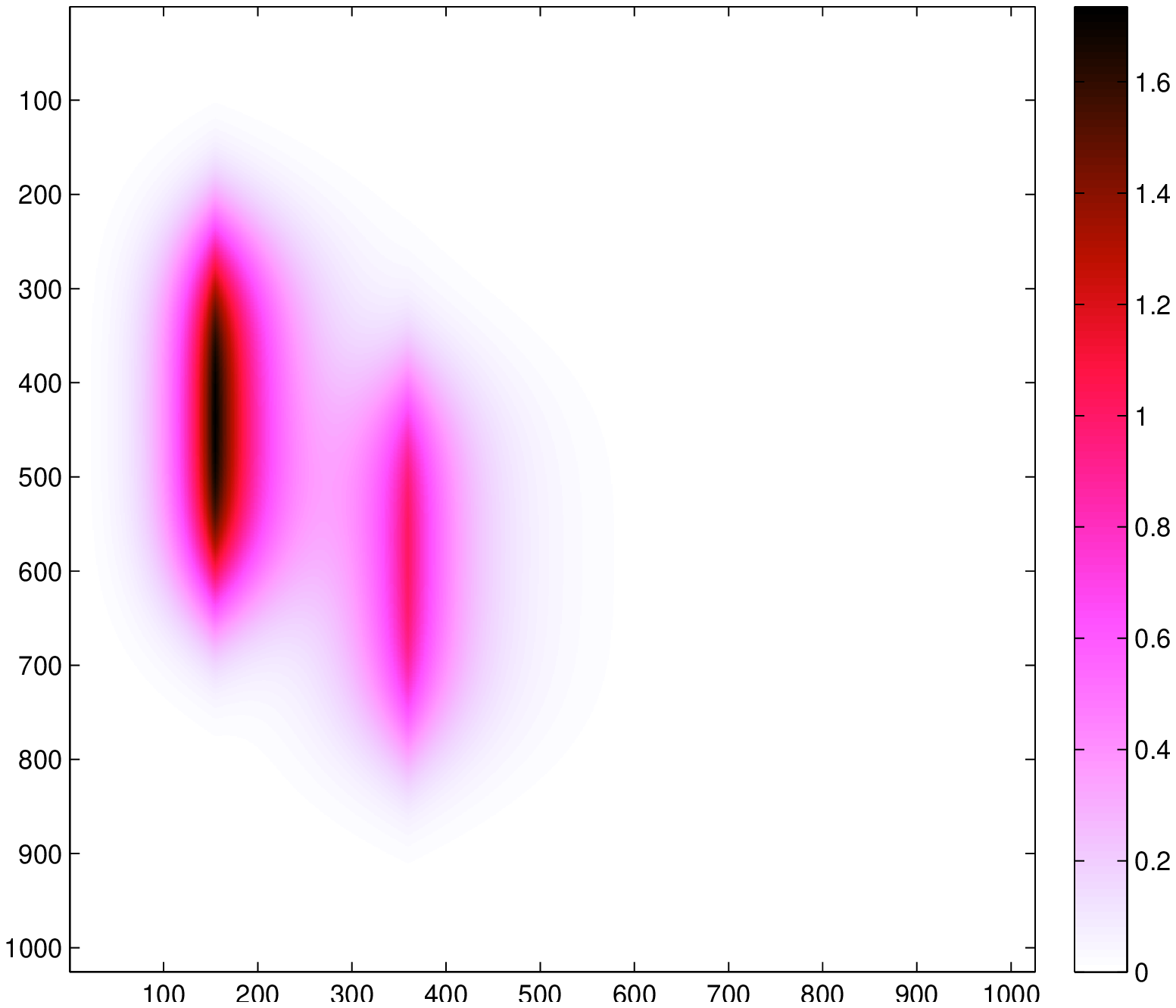}
                \caption{``Parallel'' profile}
                \label{fig:parallel}
        \end{subfigure}
        \begin{subfigure}[b]{0.49\textwidth}
                \includegraphics[width=\textwidth]{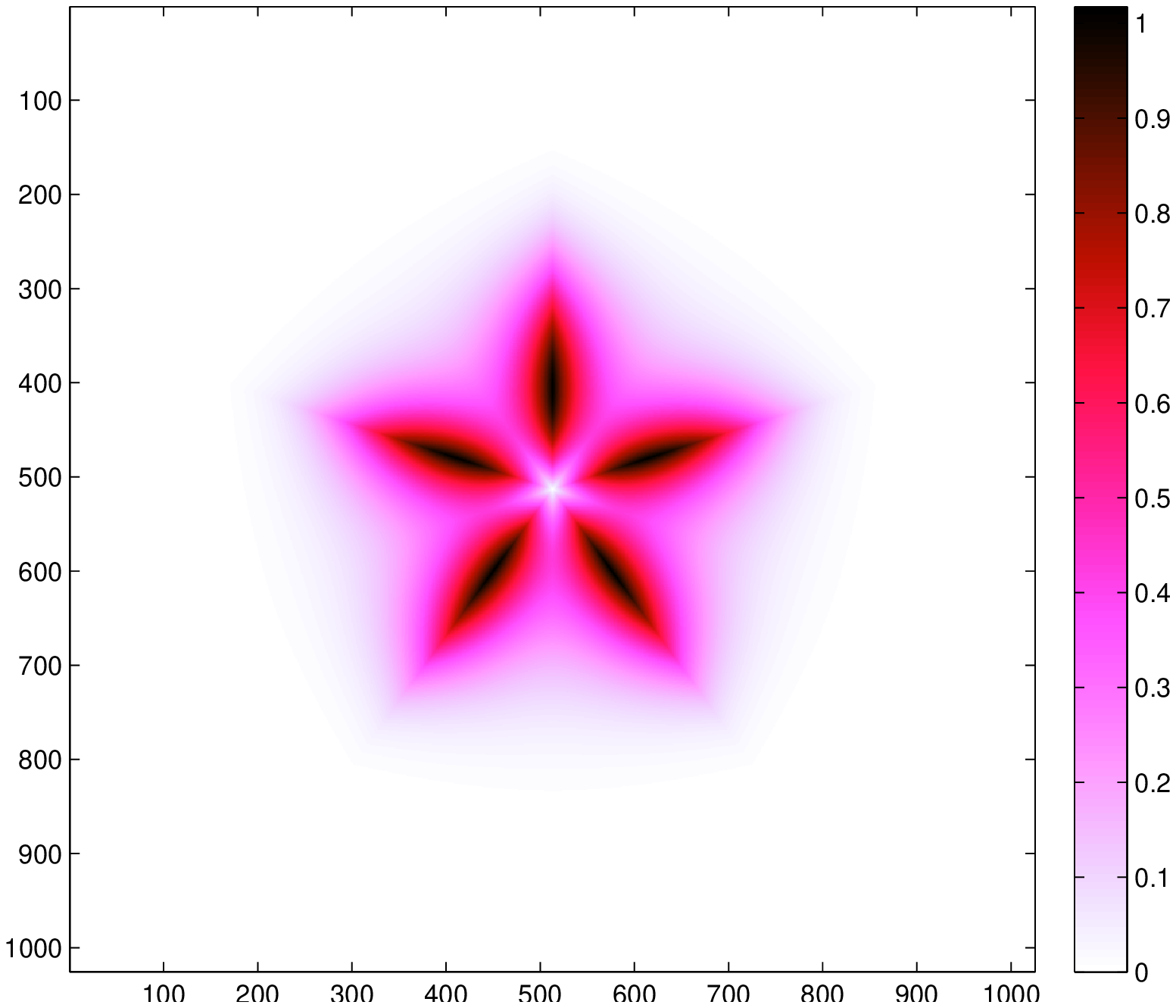}
                \caption{``Star'' profile}
                \label{fig:star}
        \end{subfigure}
\caption{Initial profiles}\label{fig:initial_profiles}
\end{figure}


\begin{figure}
        \centering
        \begin{subfigure}[b]{0.49\textwidth}
                \includegraphics[width=\textwidth]{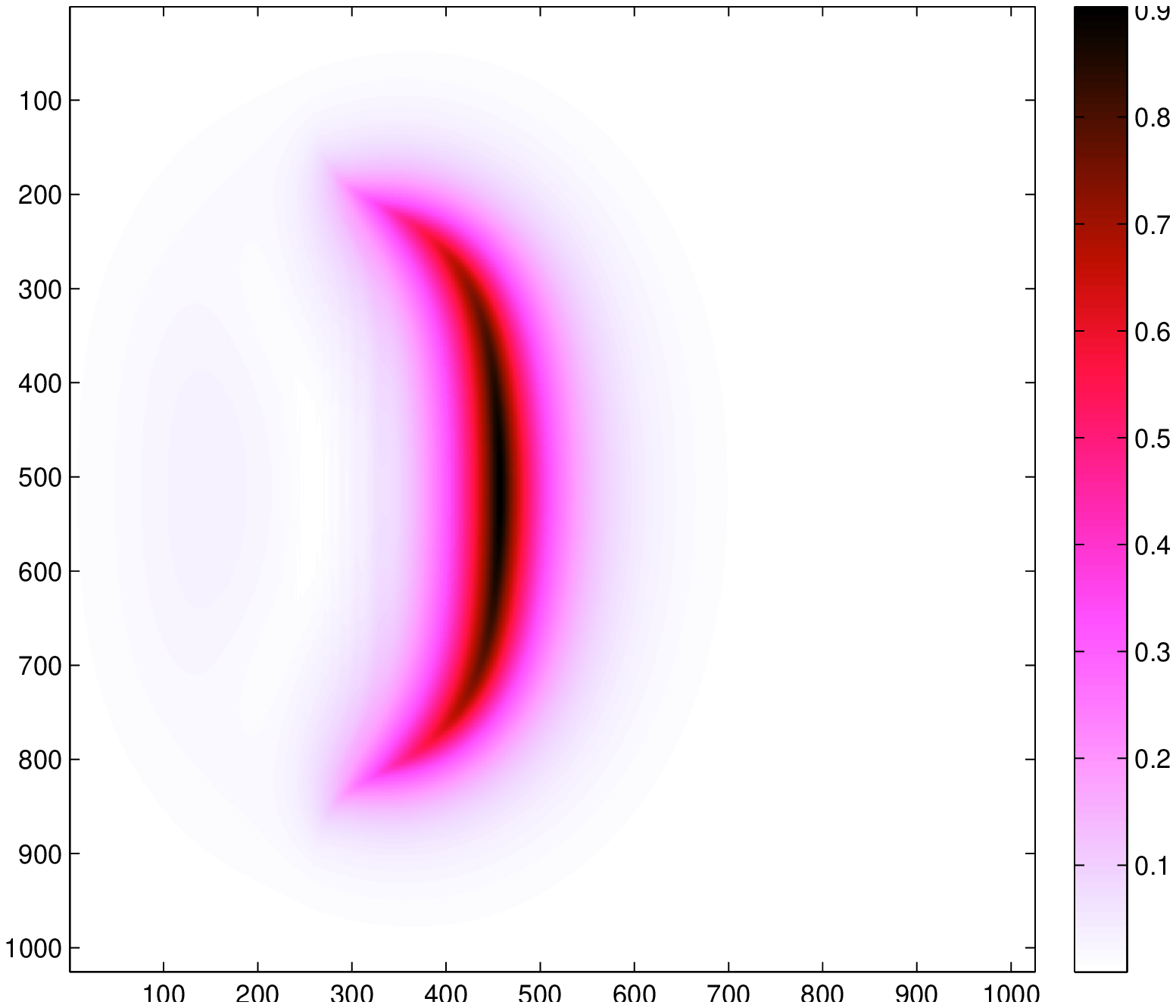}
                \caption{``Plate'', $T=0.4167$, $\alpha = \sigma$}
                \label{fig:plate13}
        \end{subfigure}
        \begin{subfigure}[b]{0.49\textwidth}
                \includegraphics[width=\textwidth]{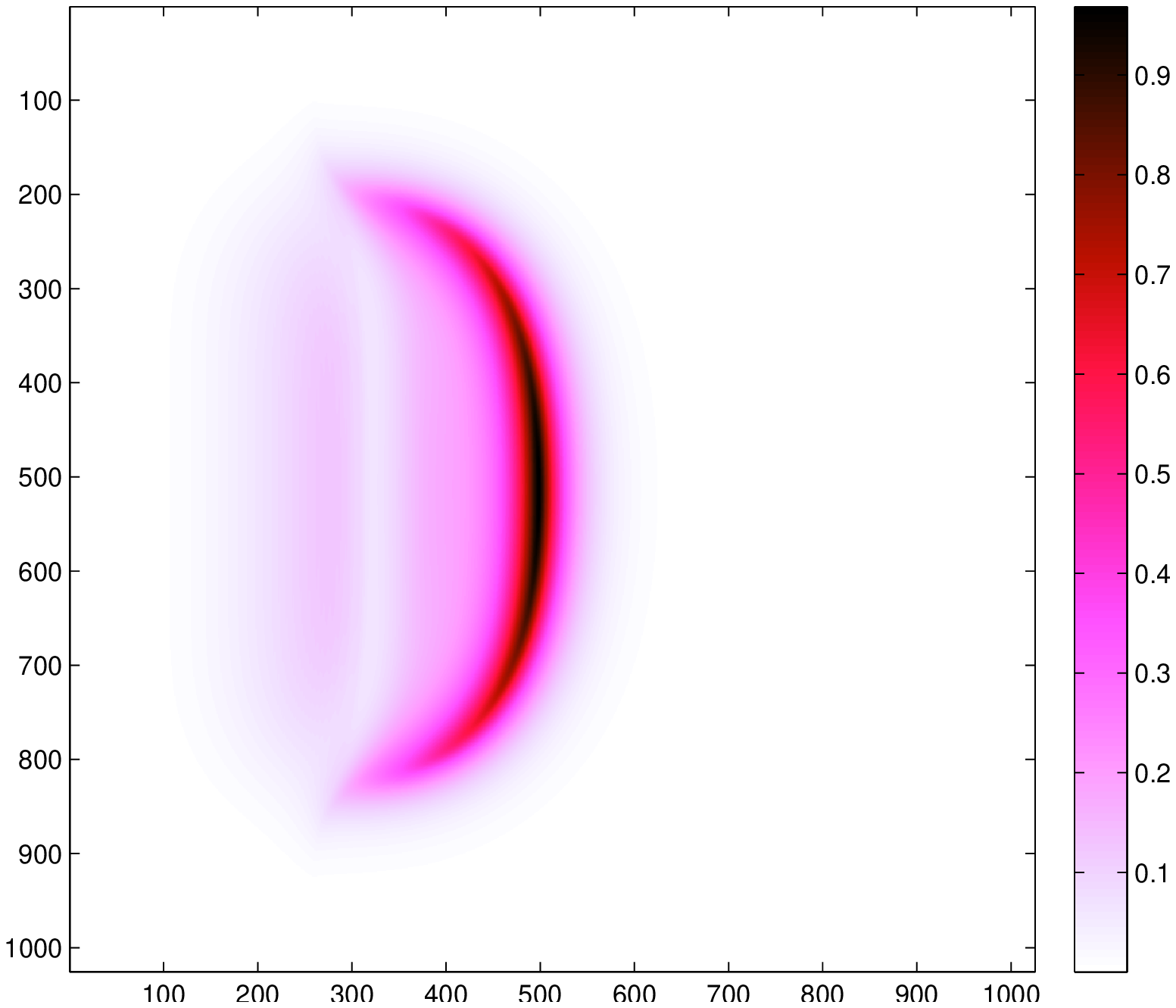}
                \caption{``Plate'', $T=0.4167$, $\alpha = \frac{\sigma}{2}$}
                \label{fig:plate23}
        \end{subfigure}
        \begin{subfigure}[b]{0.49\textwidth}
                \includegraphics[width=\textwidth]{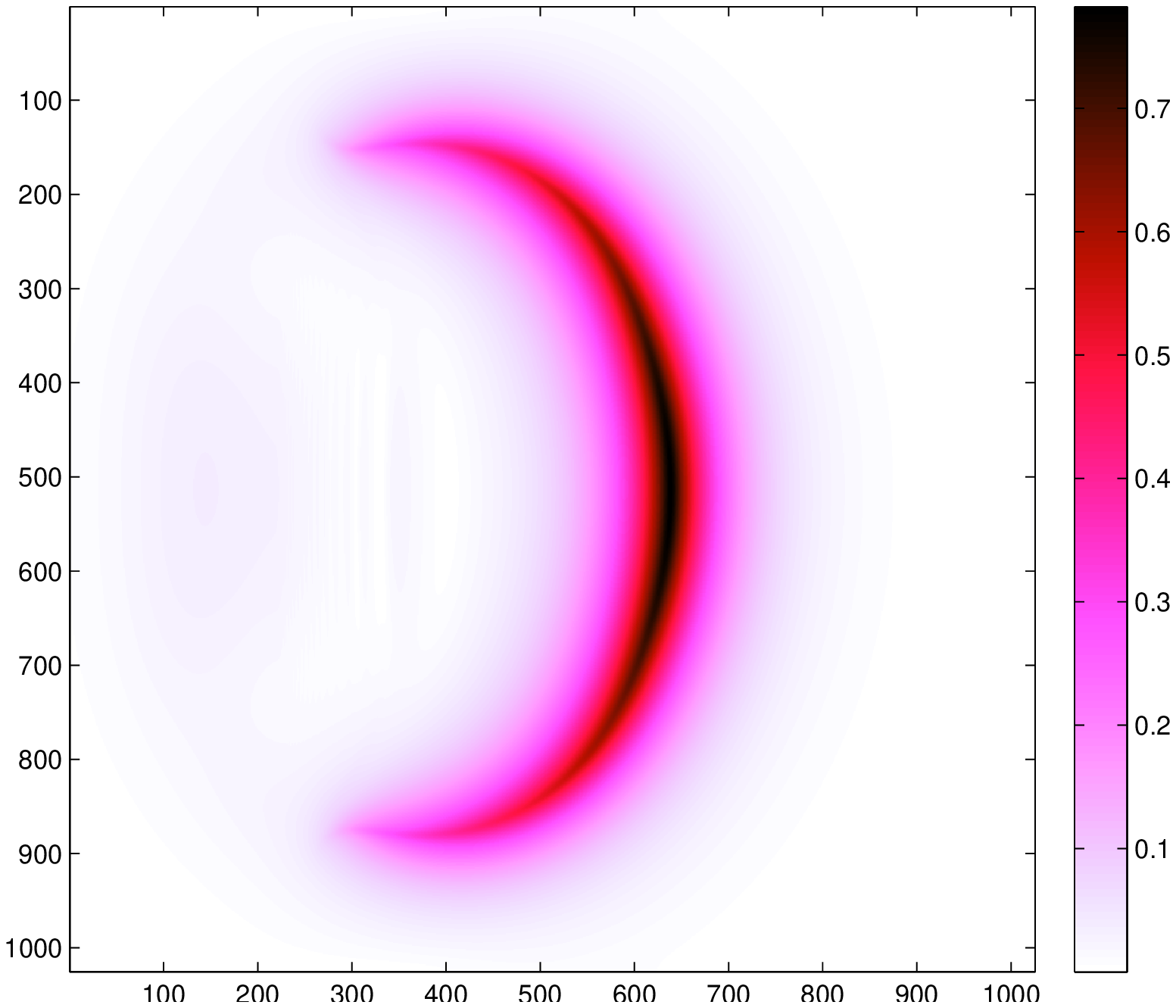}
                \caption{``Plate'', $T=0.8333$, $\alpha = \sigma$}
                \label{fig:plate15}
        \end{subfigure}
        \begin{subfigure}[b]{0.49\textwidth}
                \includegraphics[width=\textwidth]{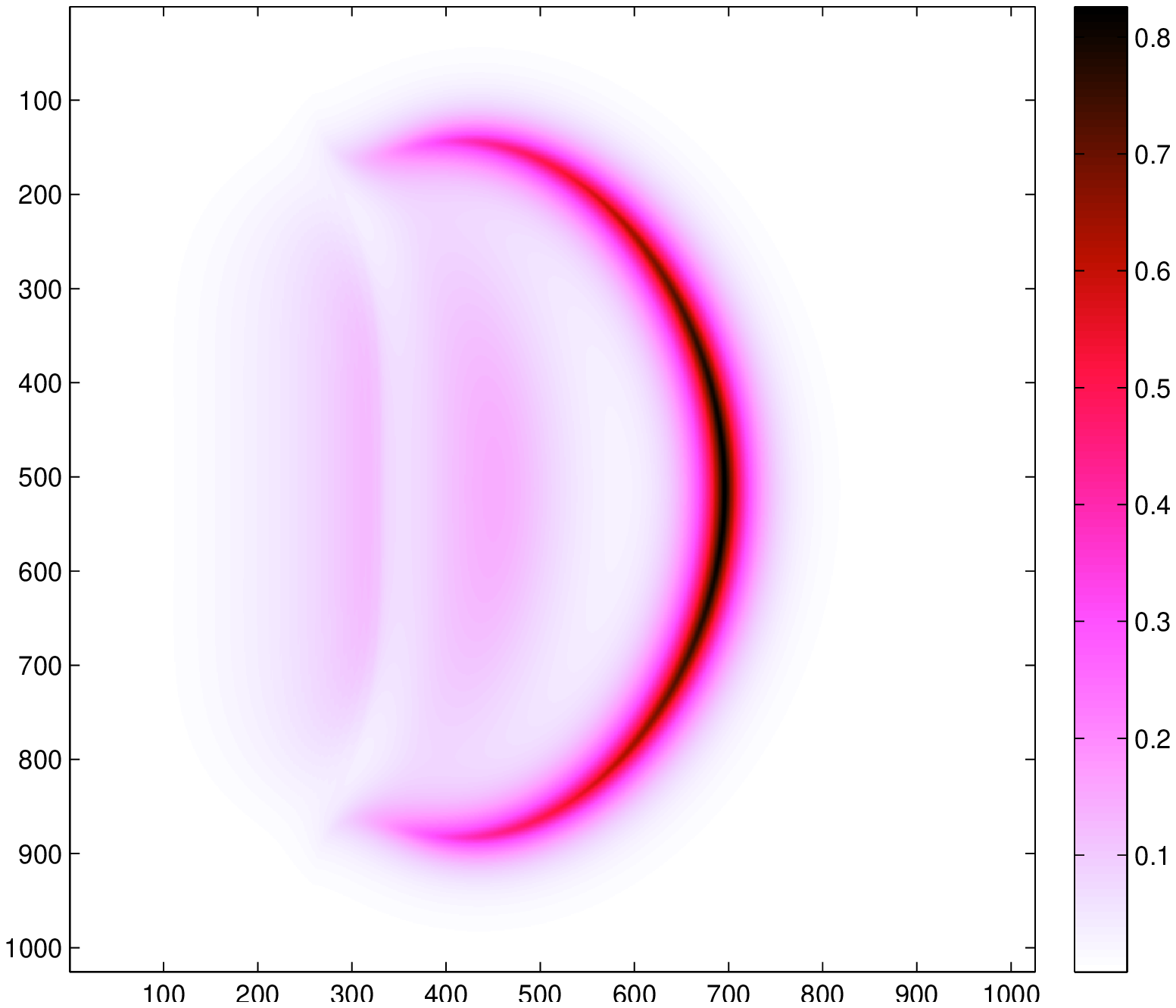}
                \caption{``Plate'', $T=0.8333$, $\alpha = \frac{\sigma}{2}$}
                \label{fig:plate25}
        \end{subfigure}
        \begin{subfigure}[b]{0.49\textwidth}
                \includegraphics[width=\textwidth]{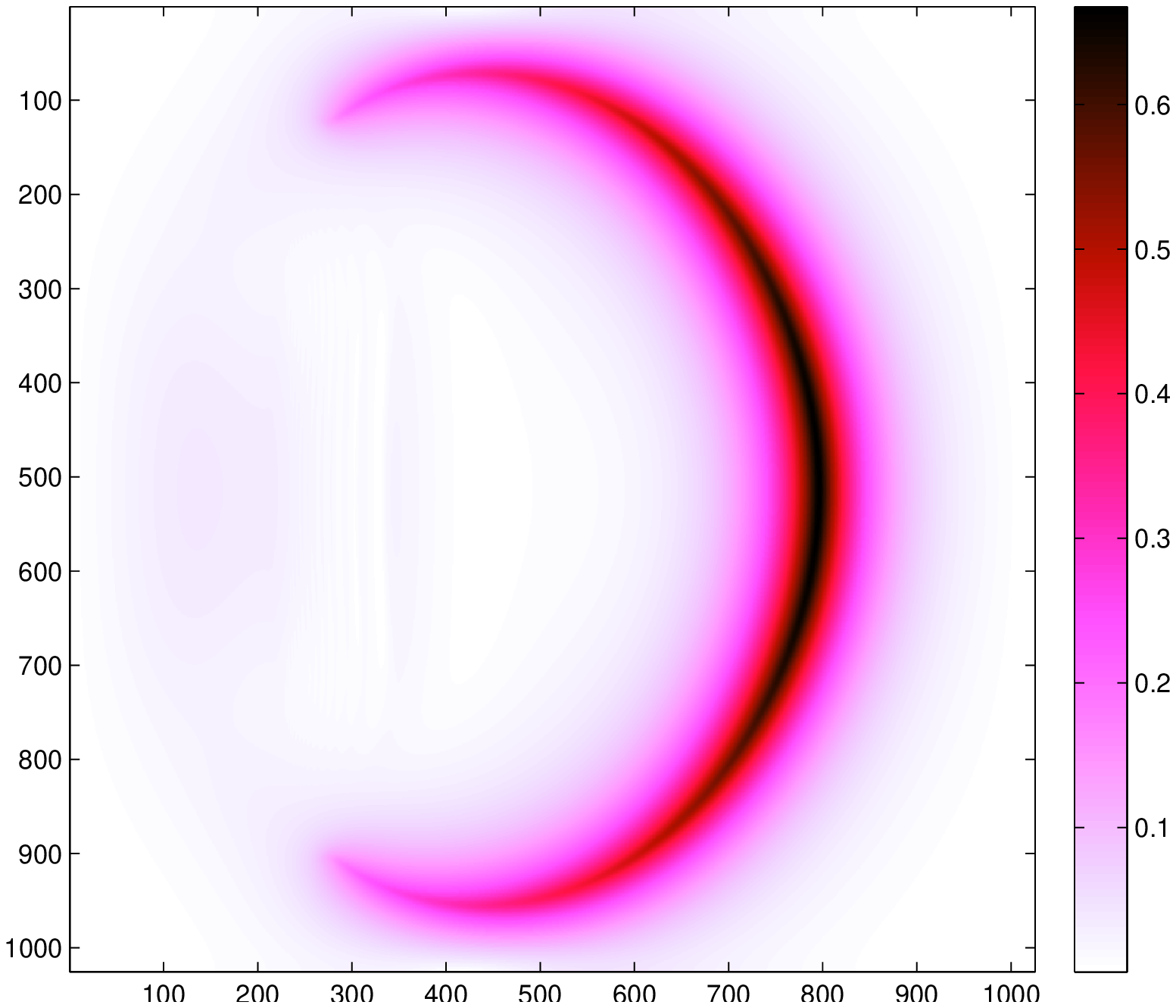}
                \caption{``Plate'', $T=1.25$, $\alpha = \sigma$}
                \label{fig:plate17}
        \end{subfigure}
        \begin{subfigure}[b]{0.49\textwidth}
                \includegraphics[width=\textwidth]{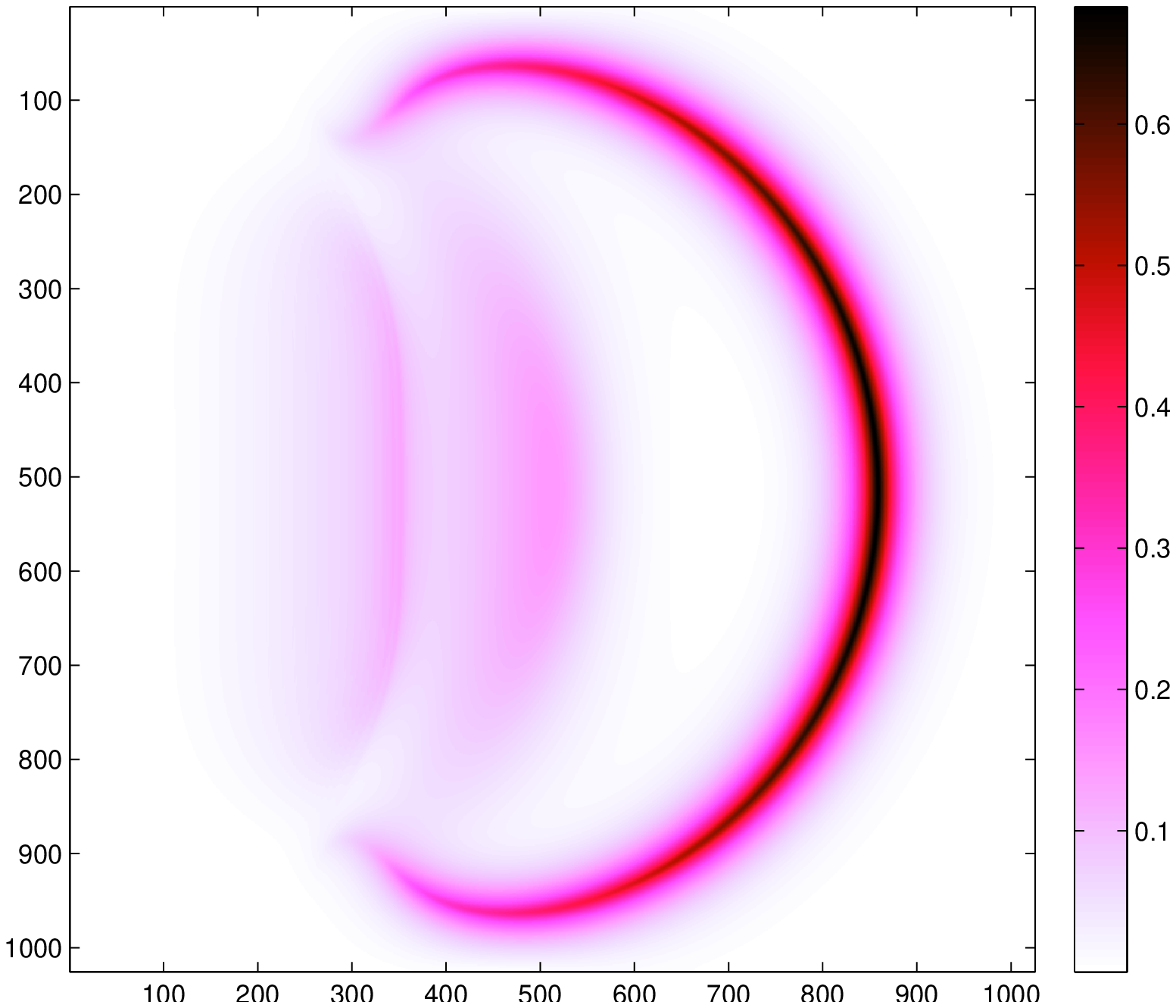}
                \caption{``Plate'', $T=1.25$, $\alpha = \frac{\sigma}{2}$}
                \label{fig:plate7}
        \end{subfigure}
\caption{Evolution of ``Plate'' on a grid $1025 \times 1025$, Scheme 
2} 
\label{fig:ev_scheme2_alpha=sigma,sigma/2}
\end{figure}


\begin{figure}
        \centering
        \begin{subfigure}[b]{0.49\textwidth}
                \includegraphics[width=\textwidth]{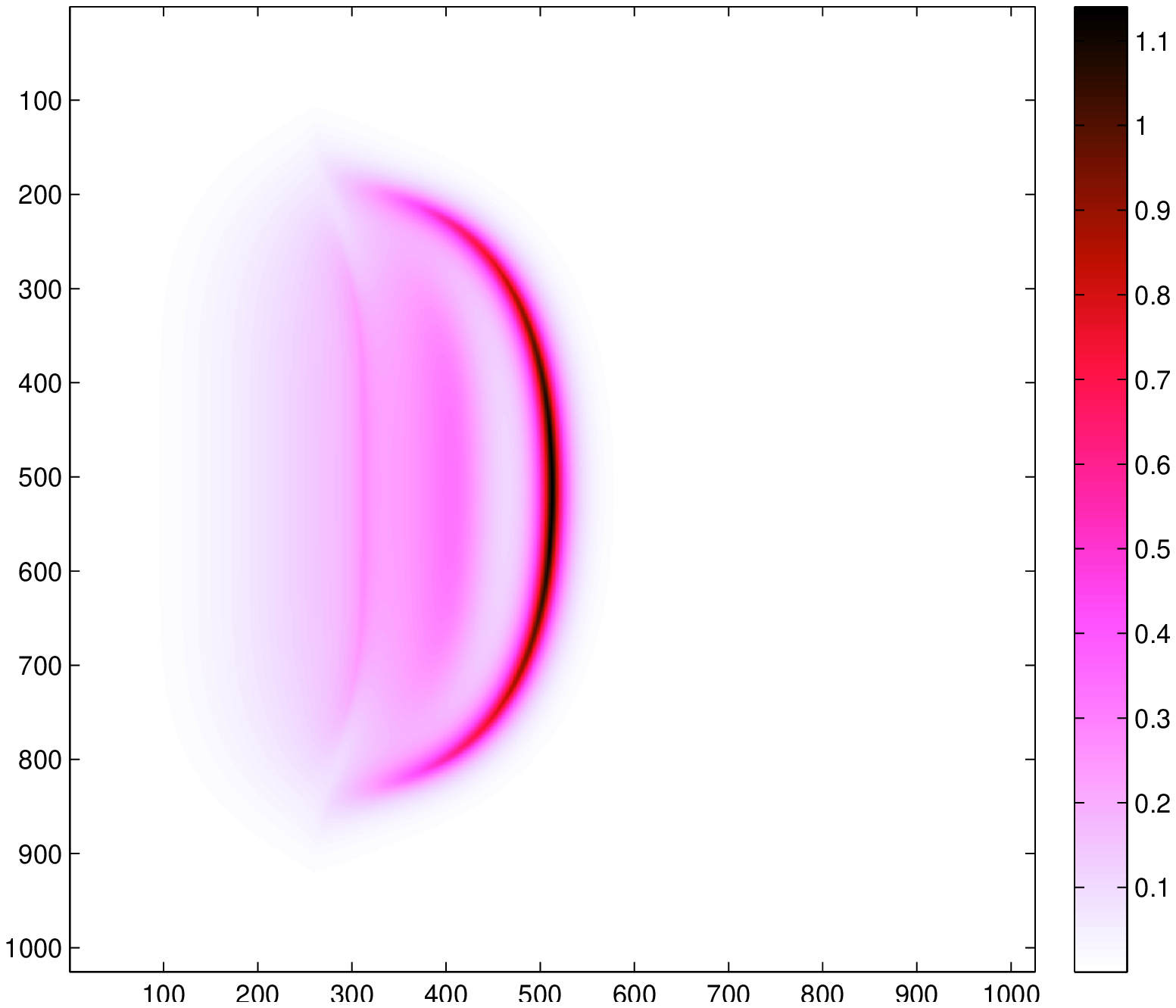}
                \caption{``Plate'', $T=0.3667$, $\alpha = \frac{\sigma}{4}$}
                \label{fig:plate43}
        \end{subfigure}
        \begin{subfigure}[b]{0.49\textwidth}
                \includegraphics[width=\textwidth]{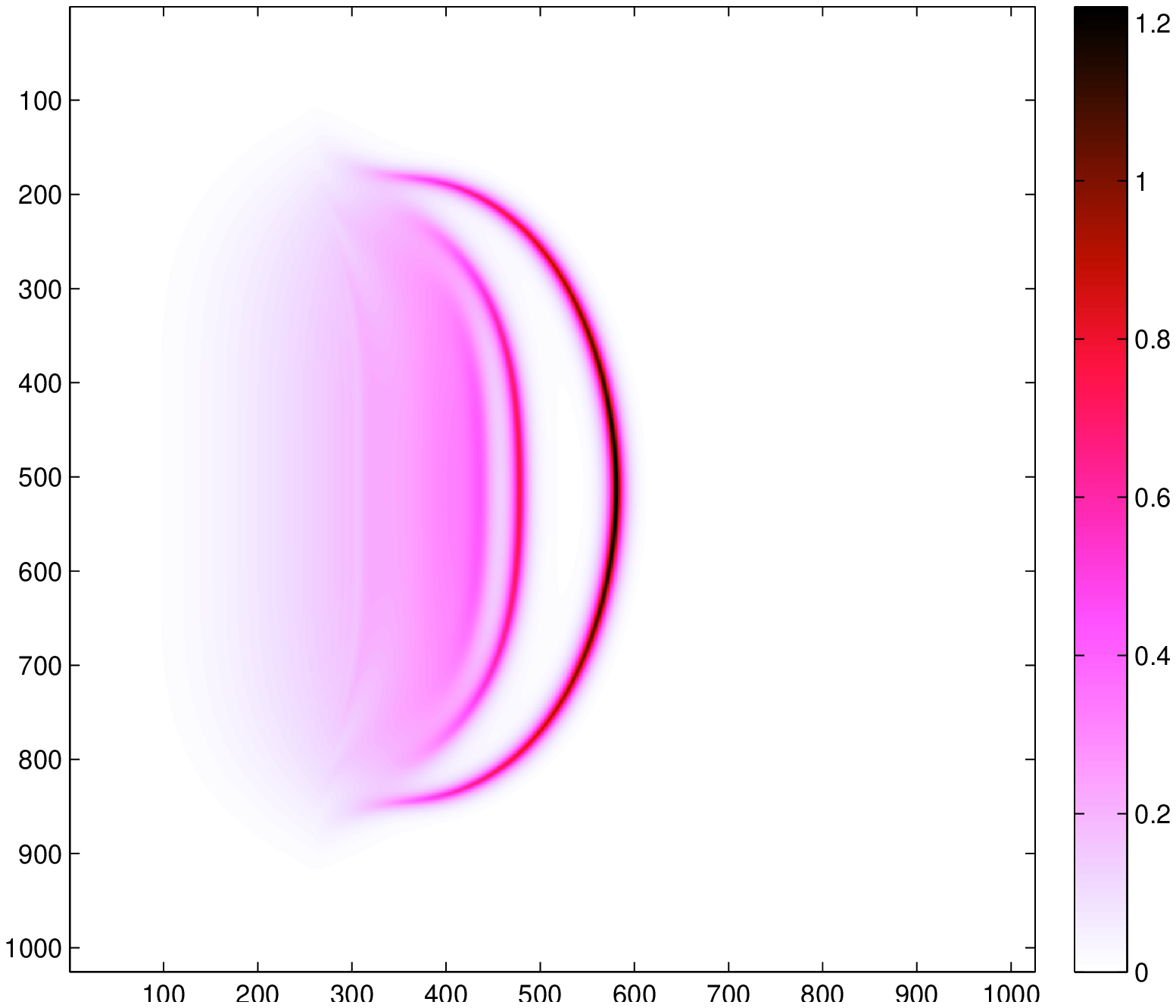}
                \caption{``Plate'', $T=0.3667$, $\alpha = \frac{\sigma}{8}$}
                \label{fig:plate83}
        \end{subfigure}
        \begin{subfigure}[b]{0.49\textwidth}
                \includegraphics[width=\textwidth]{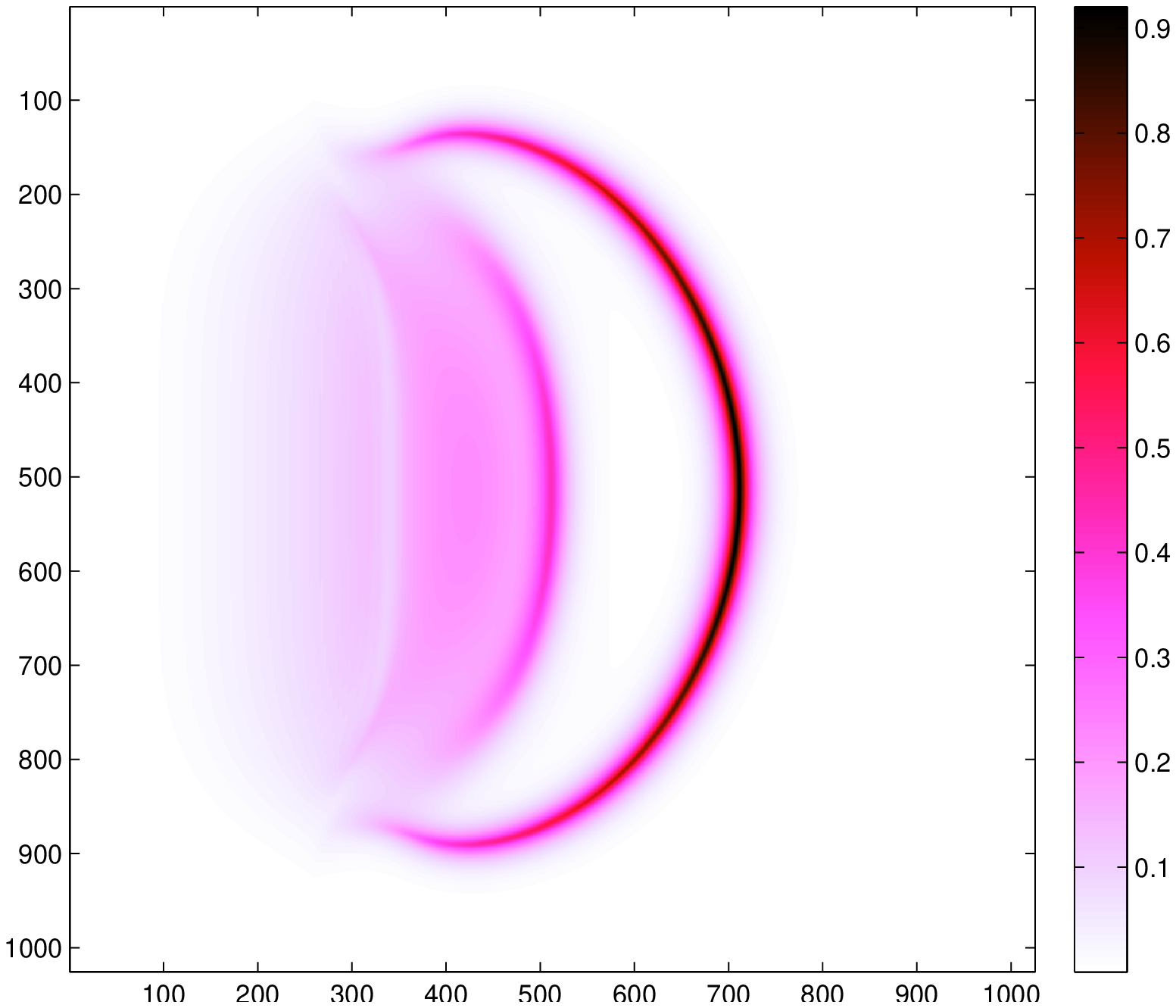}
                \caption{``Plate'', $T=0.7333$, $\alpha = \frac{\sigma}{4}$}
                \label{fig:plate45}
        \end{subfigure}
        \begin{subfigure}[b]{0.49\textwidth}
                \includegraphics[width=\textwidth]{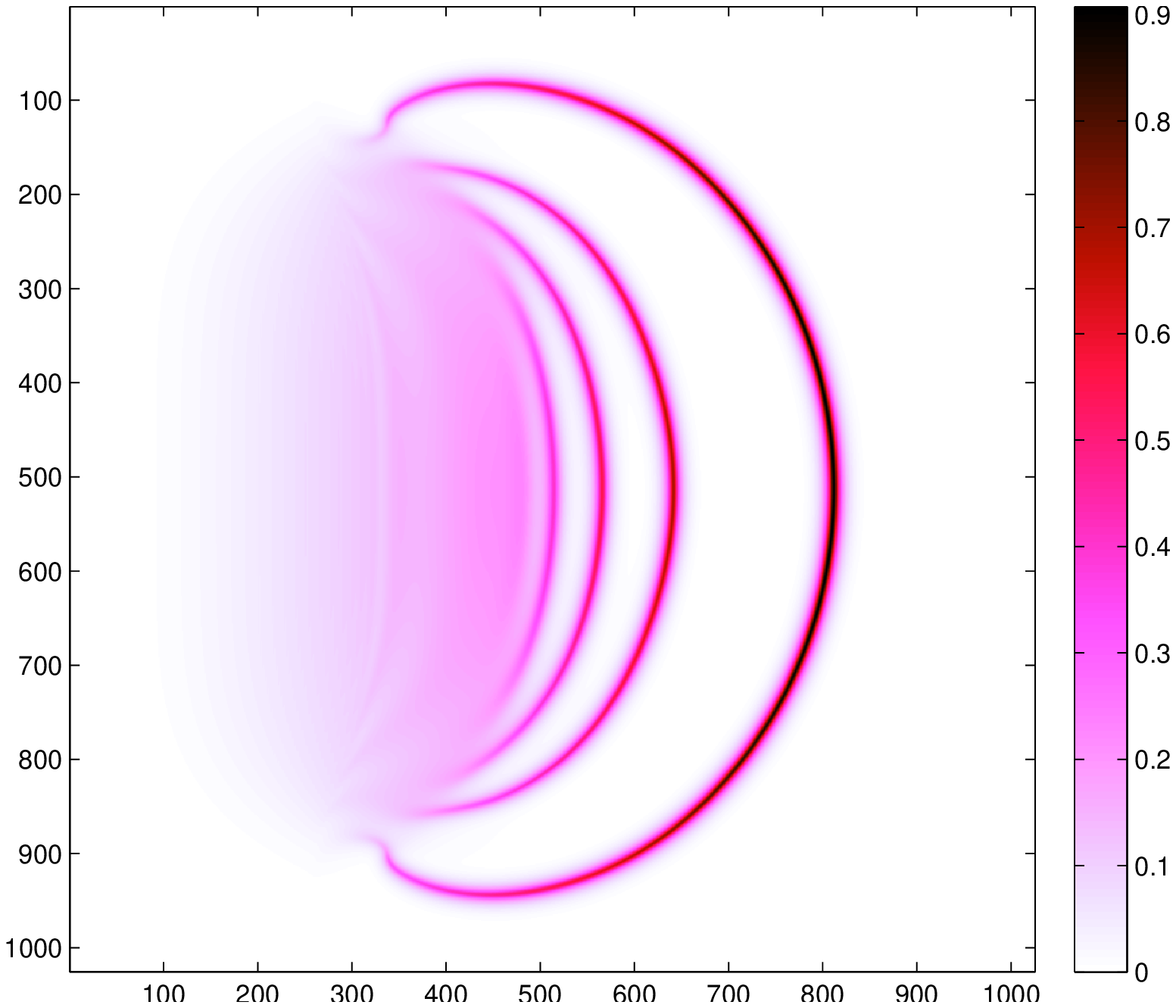}
                \caption{``Plate'', $T=0.7333$, $\alpha = \frac{\sigma}{8}$}
                \label{fig:plate85}
        \end{subfigure}
        \begin{subfigure}[b]{0.49\textwidth}
                \includegraphics[width=\textwidth]{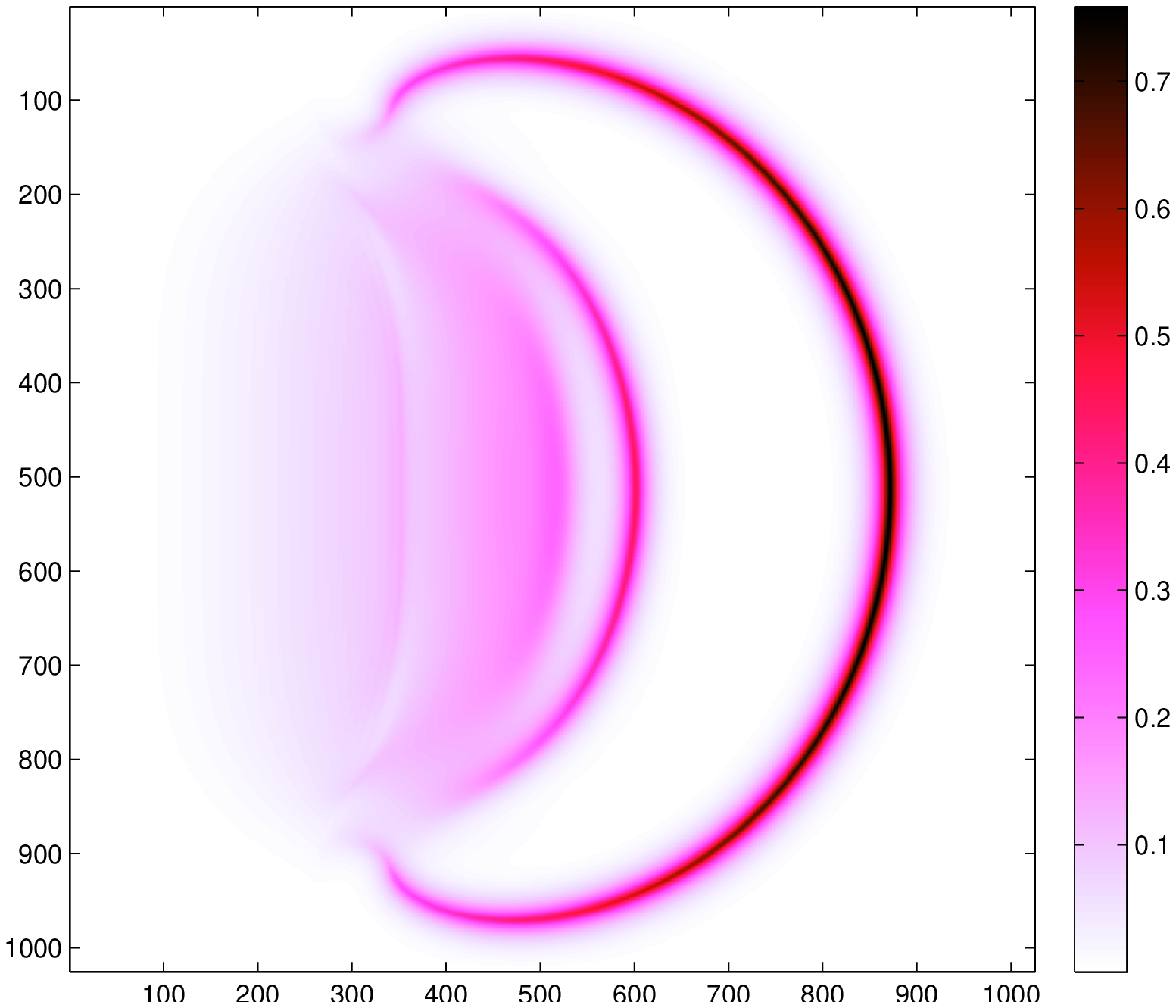}
                \caption{``Plate'', $T=1.1$, $\alpha = \frac{\sigma}{4}$}
                \label{fig:plate47}
        \end{subfigure}
        \begin{subfigure}[b]{0.49\textwidth}
                \includegraphics[width=\textwidth]{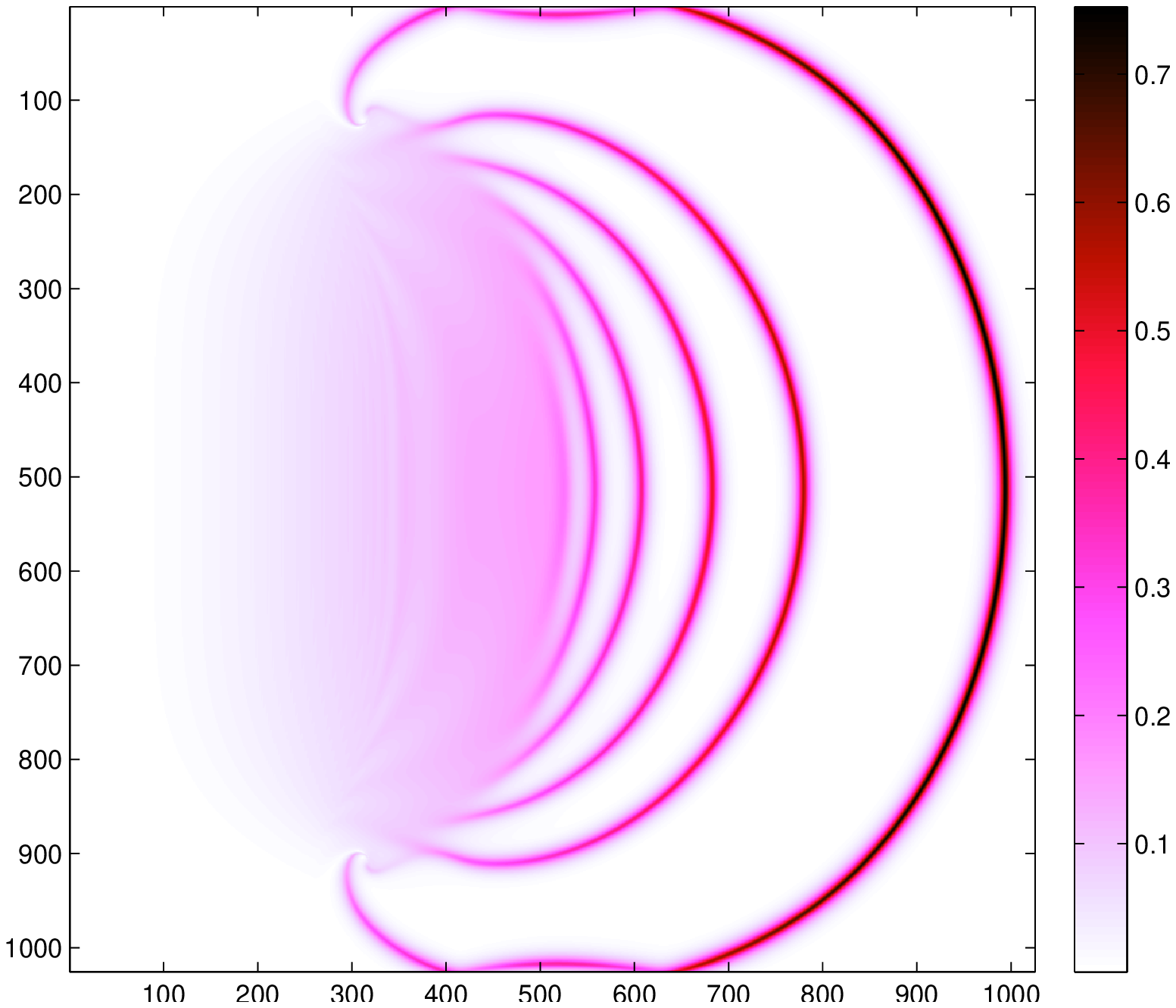}
                \caption{``Plate'', $T=1.1$, $\alpha = \frac{\sigma}{8}$}
                \label{fig:plate87}
        \end{subfigure}
\caption{Evolution of ``Plate'' on a grid $1025 \times 1025$, Scheme 
2}
\label{fig:ev_scheme2_alpha=sigma/4,sigma/8}
\end{figure}


\begin{figure}
\centering
  \begin{subfigure}[b]{0.49\textwidth}
    \includegraphics[width=\textwidth]{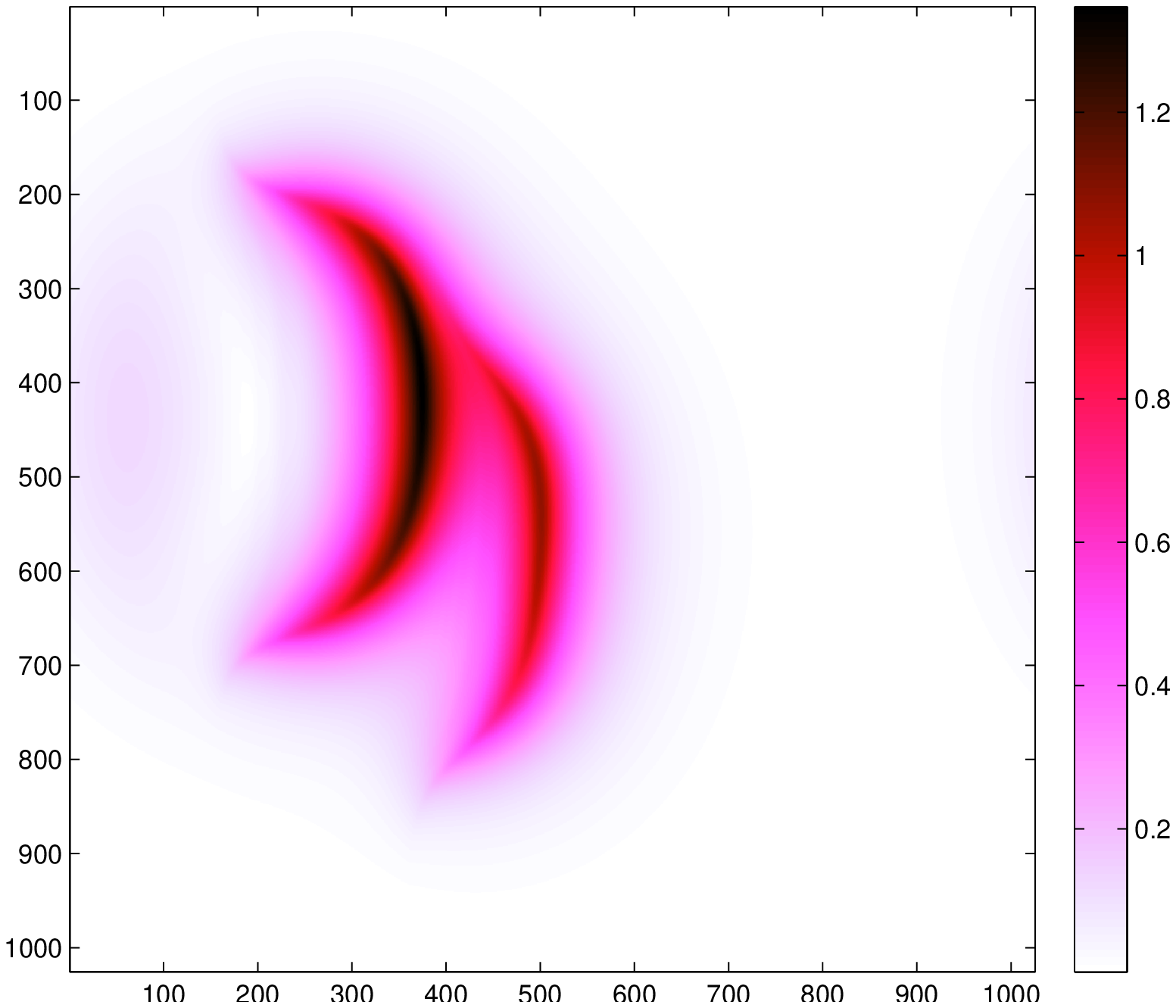}
                \caption{``Parallel'', $T=0.266$, $\alpha = \sigma$}
    \label{fig:parallel13}
  \end{subfigure}
  \begin{subfigure}[b]{0.49\textwidth}
    \includegraphics[width=\textwidth]{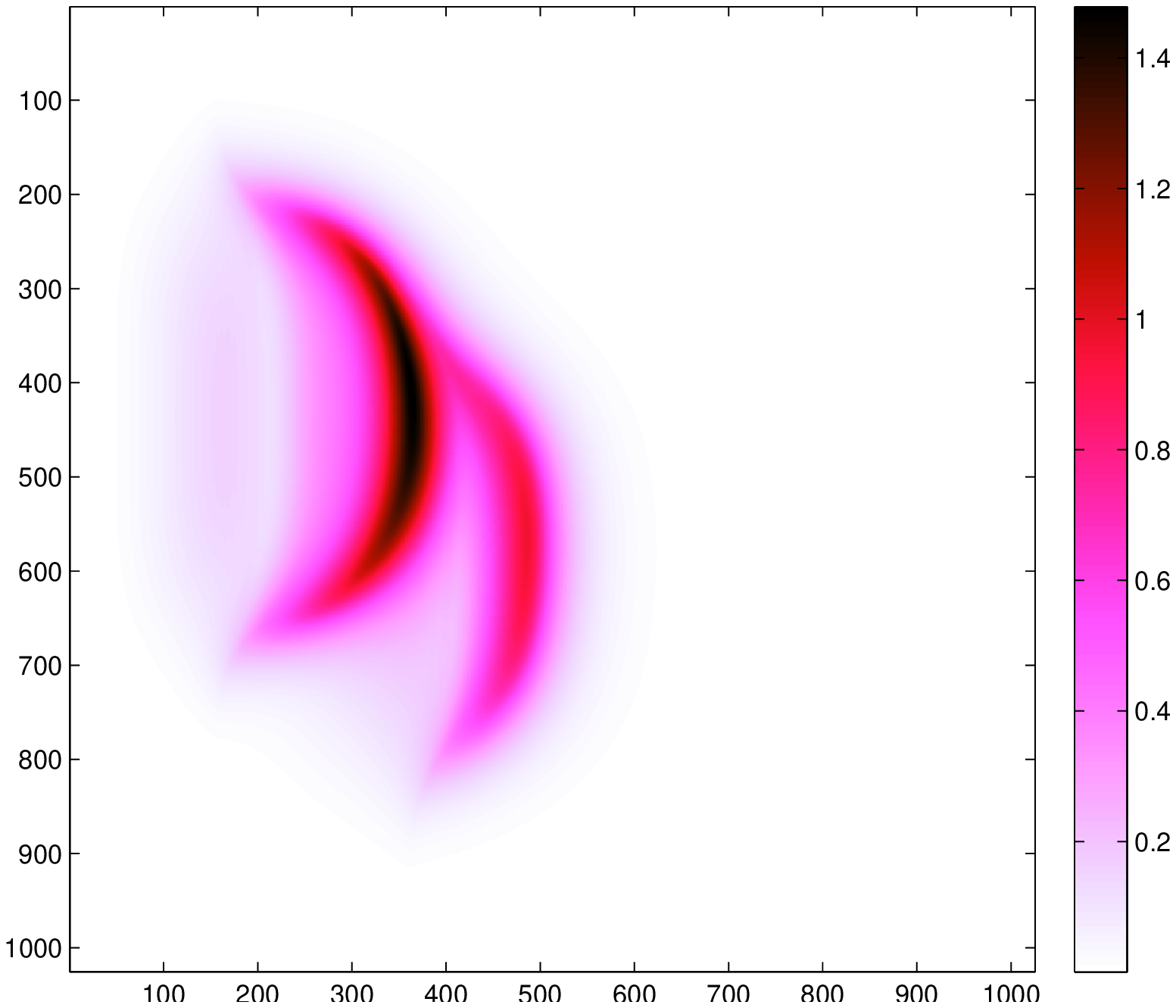}
                \caption{``Parallel'', $T=0.266$, $\alpha = \frac{\sigma}{2}$}
    \label{fig:parallel23}
  \end{subfigure}
    \begin{subfigure}[b]{0.49\textwidth}
    \includegraphics[width=\textwidth]{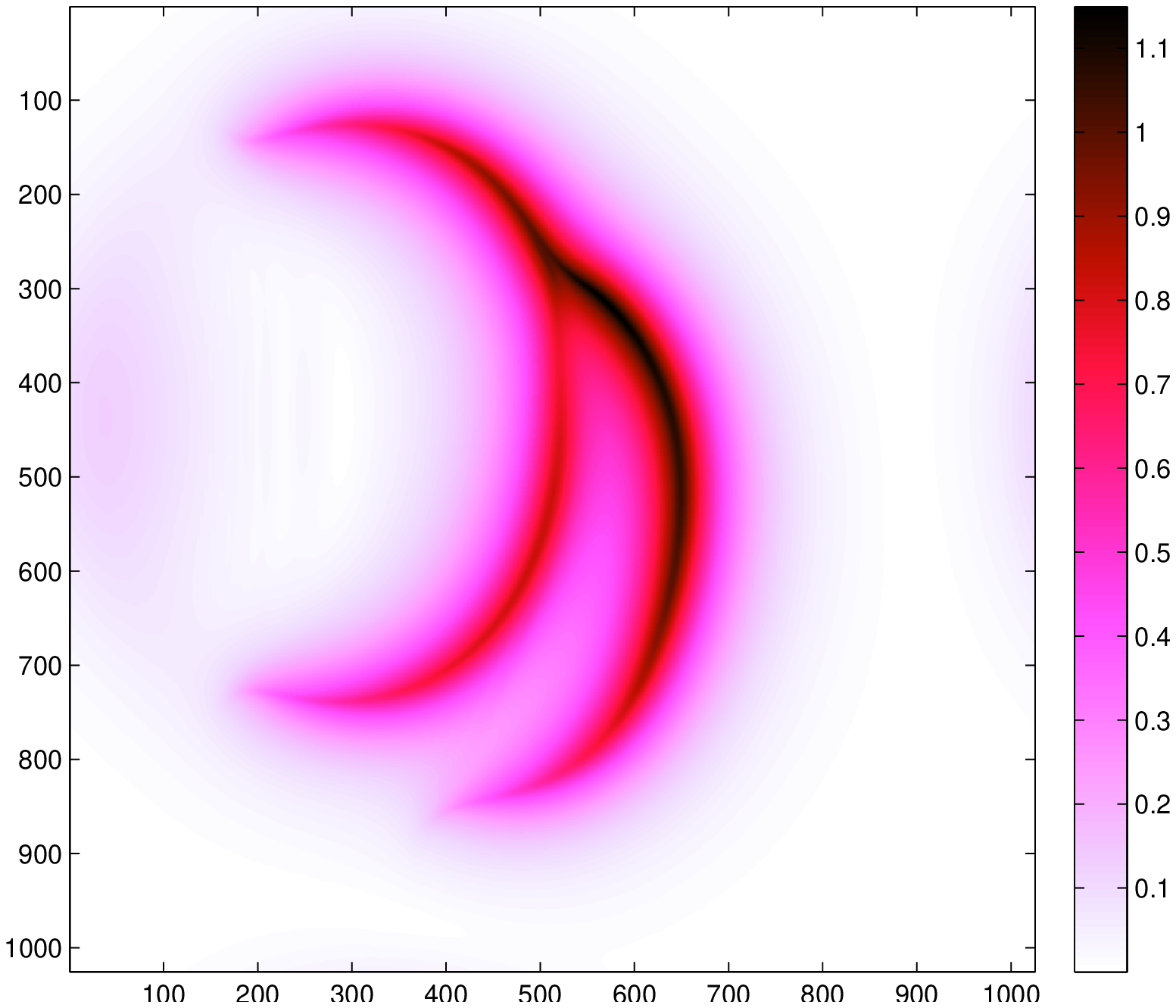}
                \caption{``Parallel'', $T=0.533$, $\alpha = \sigma$}
    \label{fig:parallel15}
  \end{subfigure}
  \begin{subfigure}[b]{0.49\textwidth}
    \includegraphics[width=\textwidth]{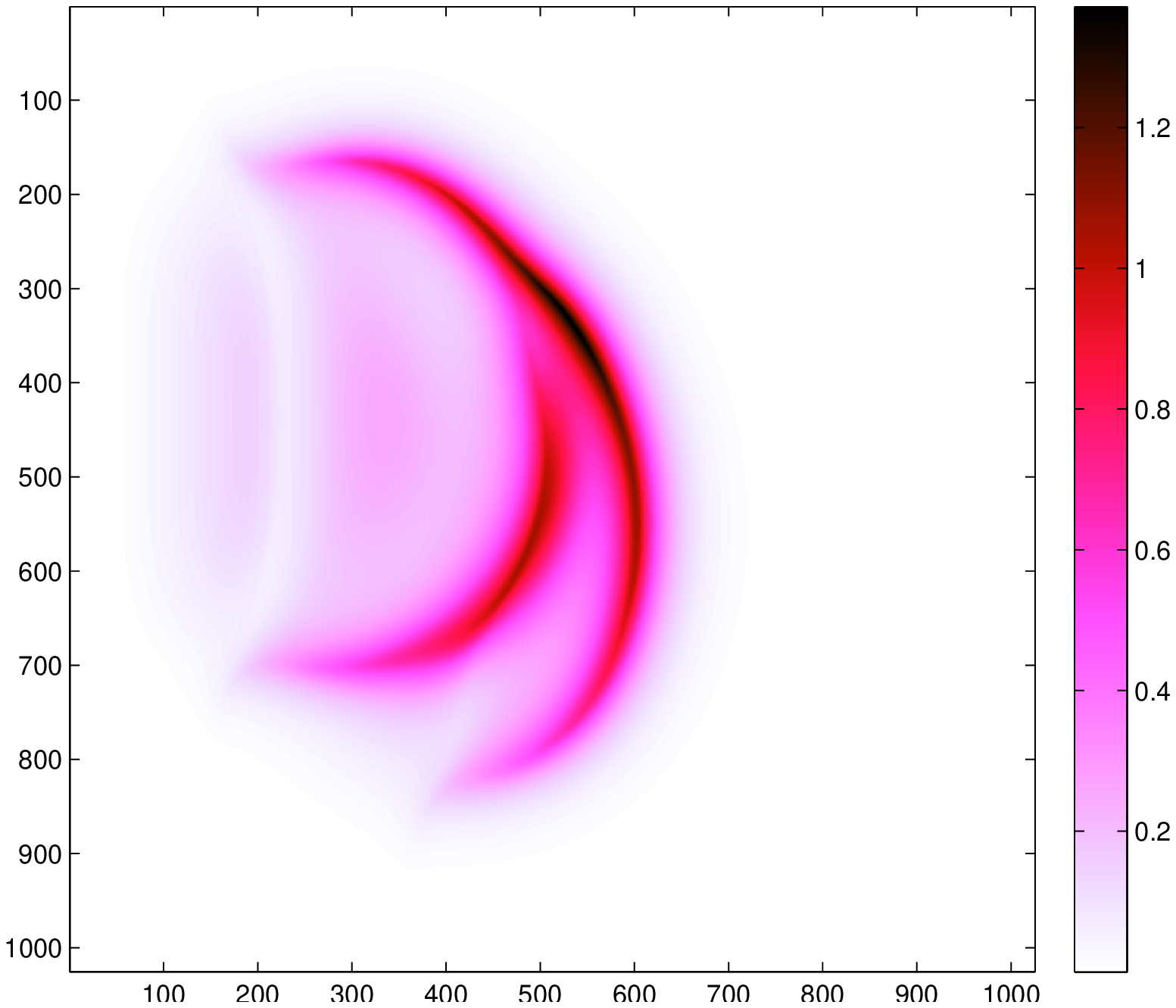}
                \caption{``Parallel'', $T=0.533$, $\alpha = \frac{\sigma}{2}$}
    \label{fig:parallel25}
  \end{subfigure}
    \begin{subfigure}[b]{0.49\textwidth}
    \includegraphics[width=\textwidth]{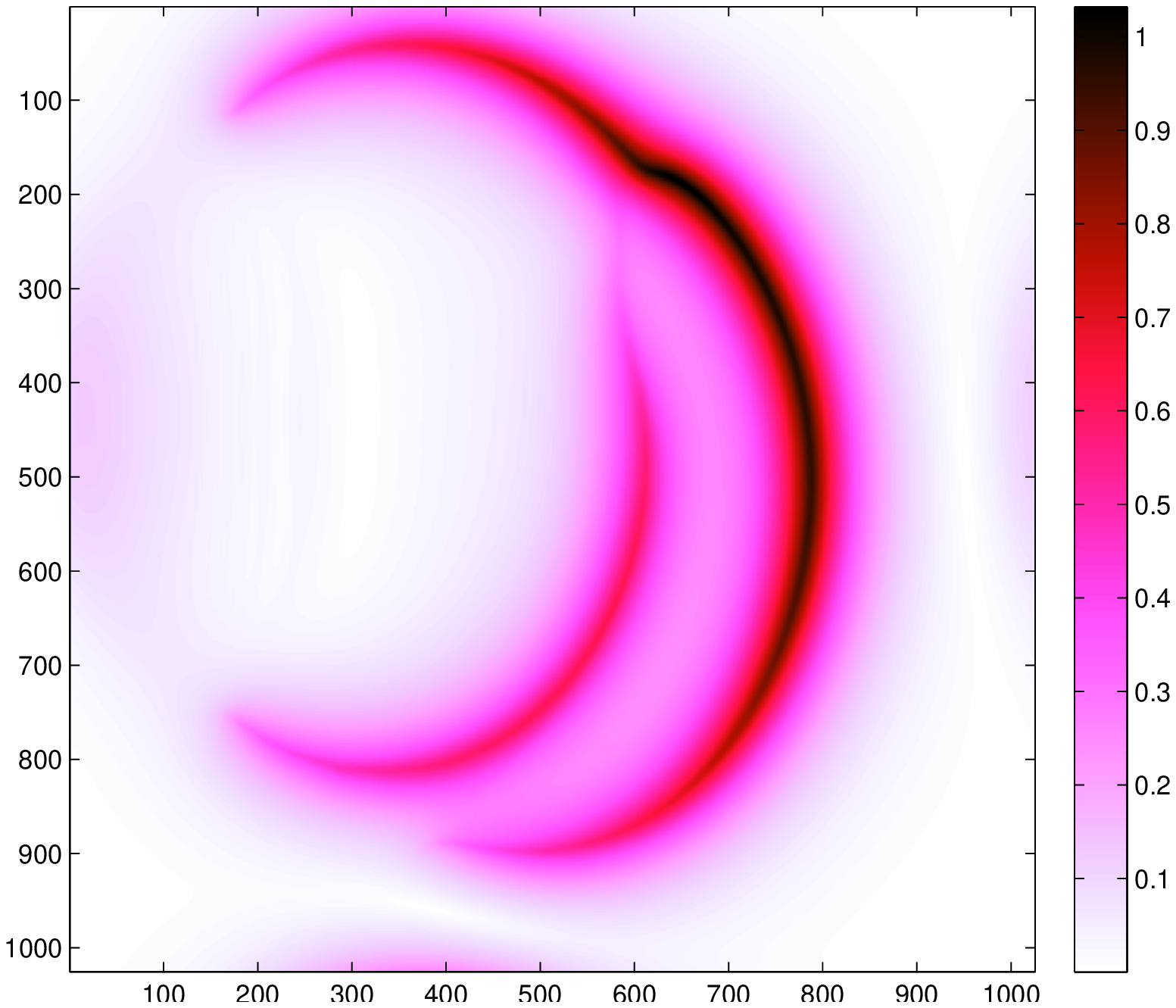}
                \caption{``Parallel'', $T=0.8$, $\alpha = \sigma$}
    \label{fig:parallel17}
  \end{subfigure}
  \begin{subfigure}[b]{0.49\textwidth}
    \includegraphics[width=\textwidth]{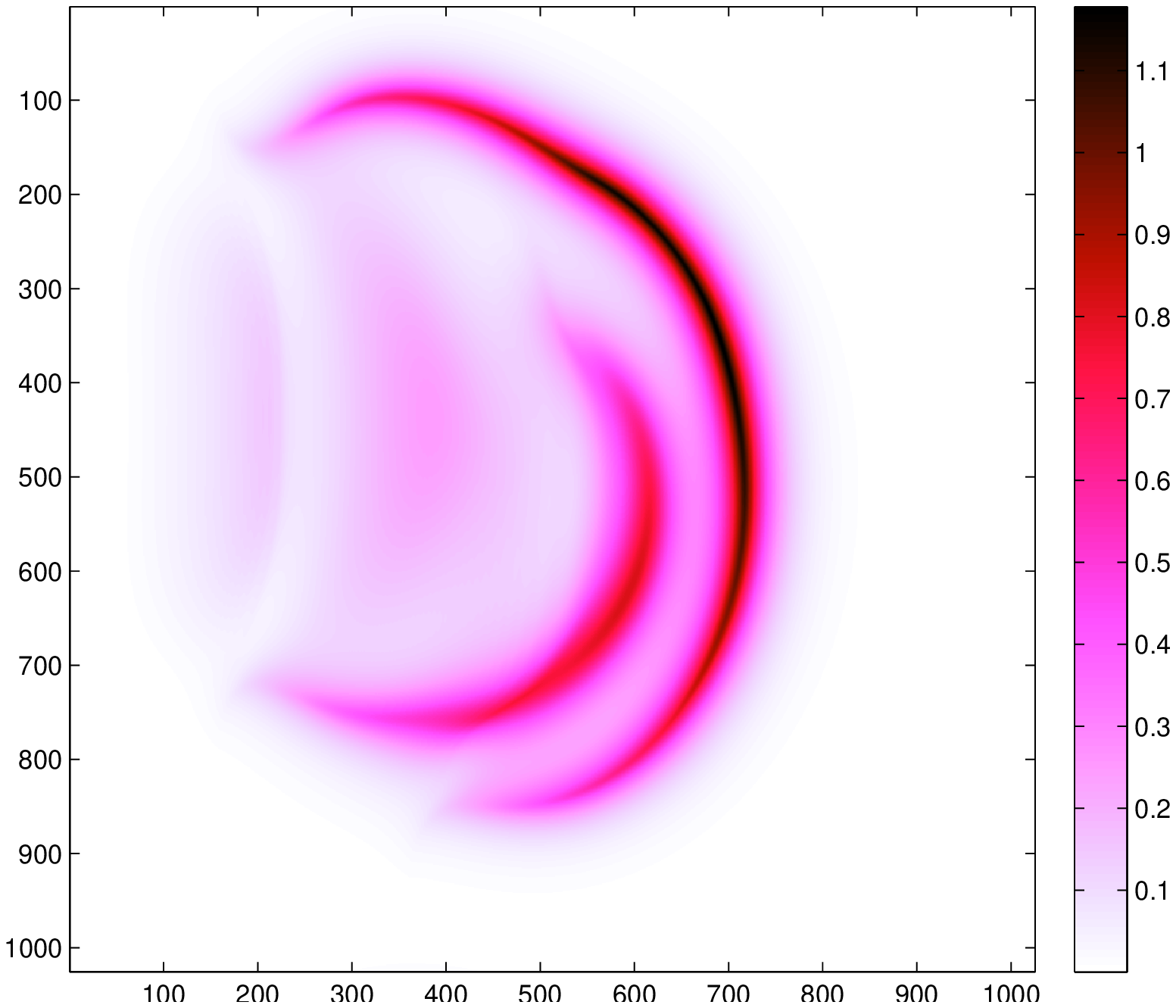}
                \caption{``Parallel'', $T=0.8$, $\alpha = \frac{\sigma}{2}$}
    \label{fig:parallel27}
  \end{subfigure}
  \caption{Evolution of ``Parallel'' on a grid $1025 \times 1025$, Scheme 
2}
\label{fig:ev_scheme2_alpha=sigma,sigma/2_parallel}
\end{figure}


\begin{figure}
        \centering
        \begin{subfigure}[b]{0.49\textwidth}
                \includegraphics[width=\textwidth]{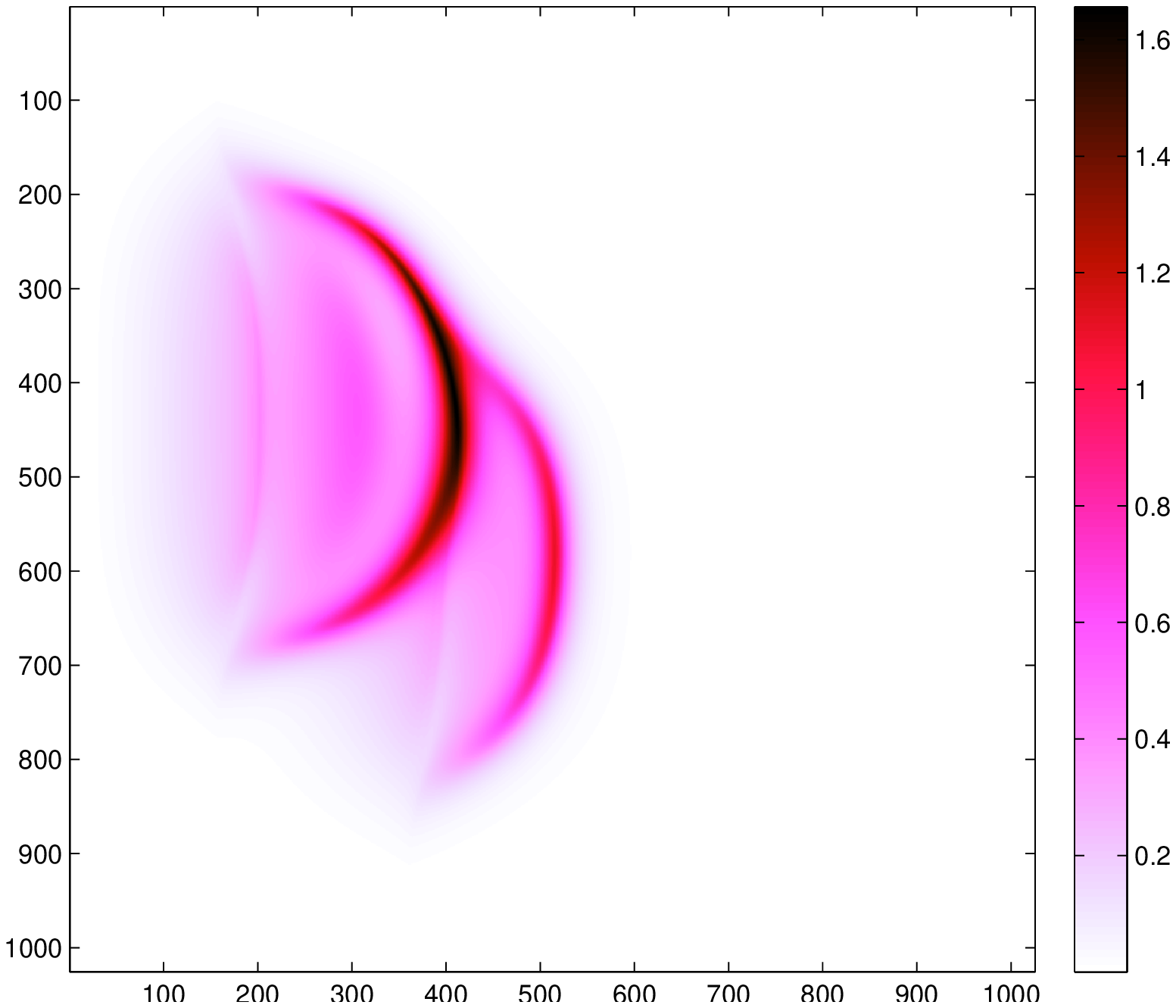}
                \caption{``Parallel'', $T=0.266$, $\alpha = \frac{\sigma}{4}$}
                \label{fig:par1_23}
        \end{subfigure}
        \begin{subfigure}[b]{0.49\textwidth}
                \includegraphics[width=\textwidth]{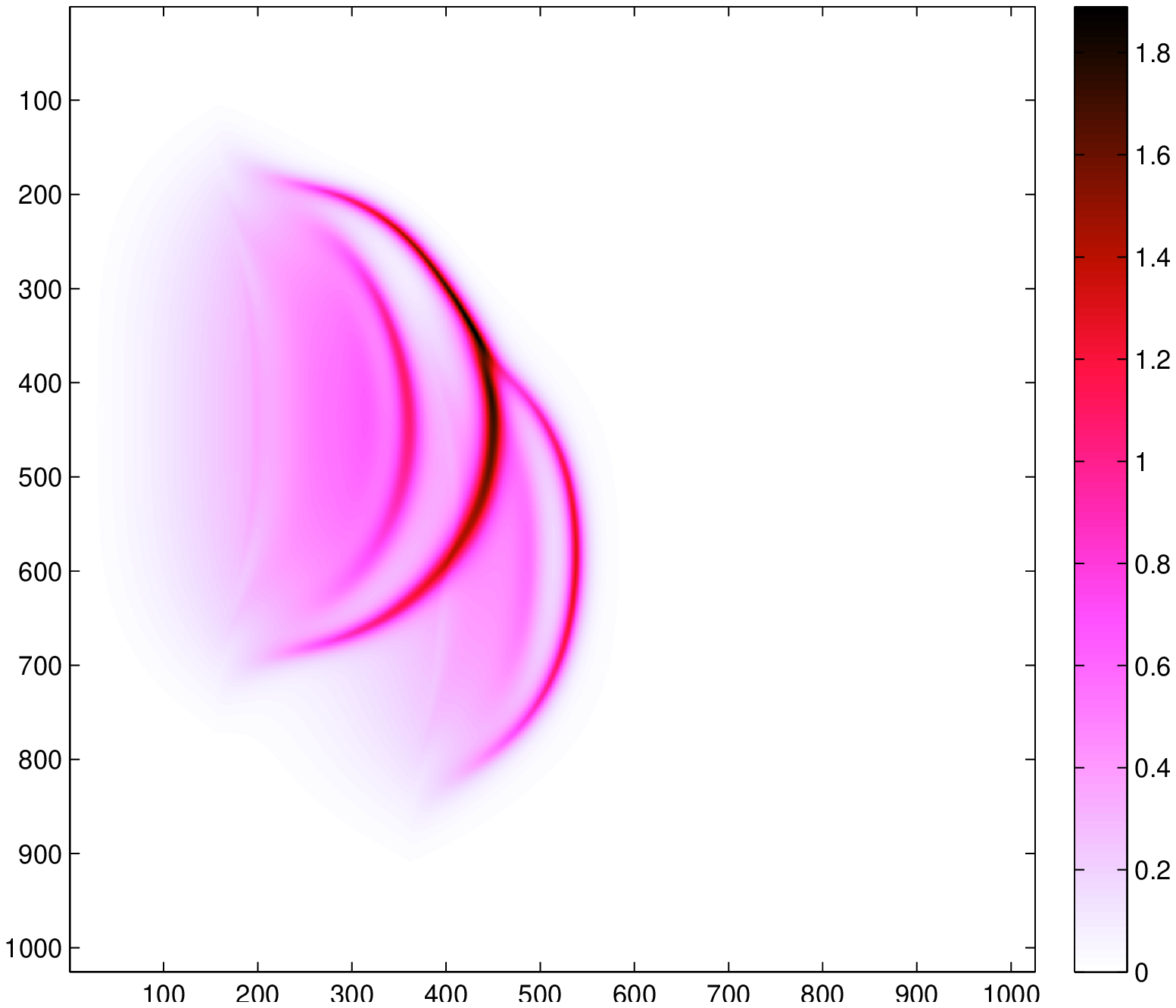}
                \caption{``Parallel'', $T=0.266$, $\alpha = \frac{\sigma}{8}$}
                \label{fig:par1_23_surf}
        \end{subfigure}
        \begin{subfigure}[b]{0.49\textwidth}
                \includegraphics[width=\textwidth]{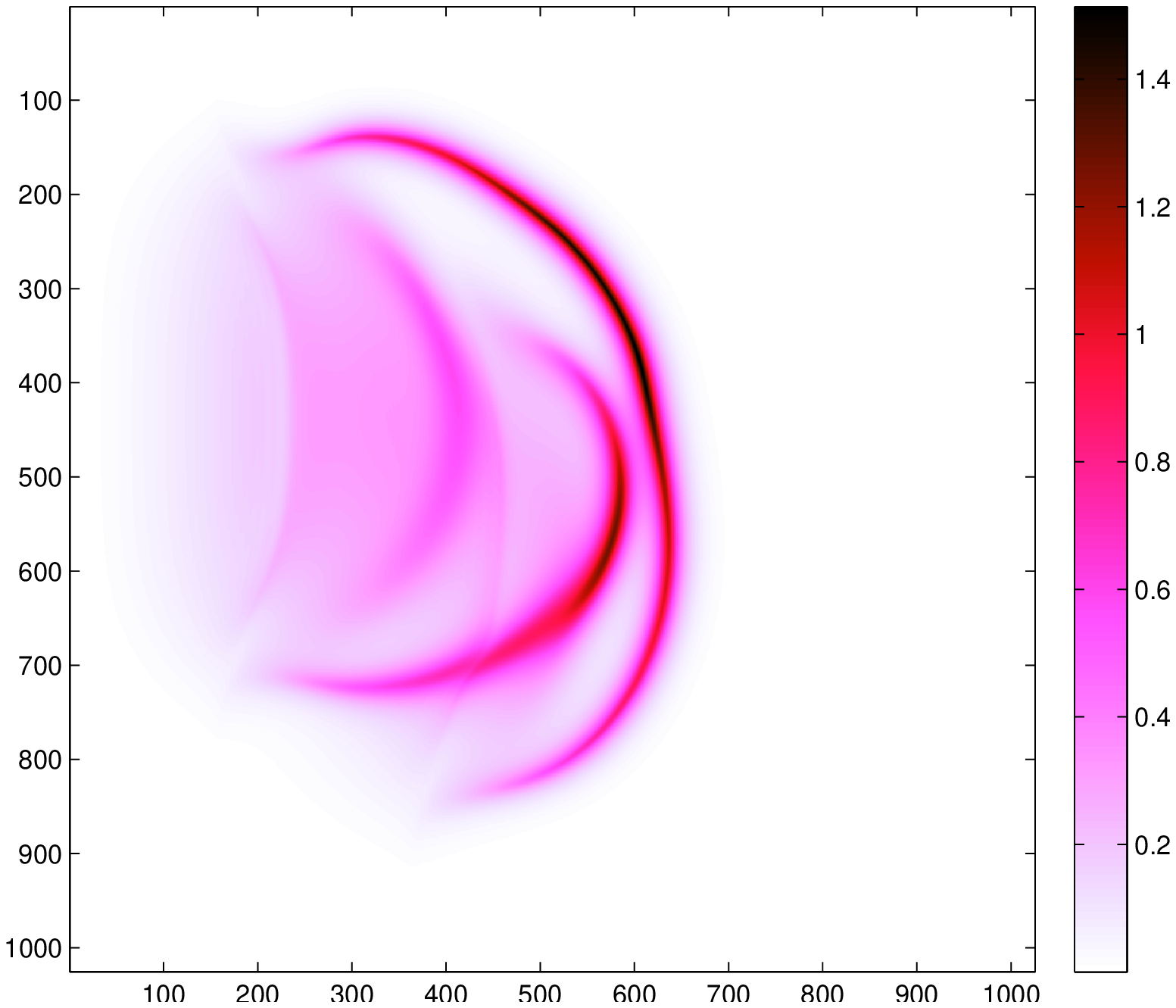}
                \caption{``Parallel'', $T=0.533$, $\alpha = \frac{\sigma}{4}$}
                \label{fig:par1_25}
        \end{subfigure}
        \begin{subfigure}[b]{0.49\textwidth}
                \includegraphics[width=\textwidth]{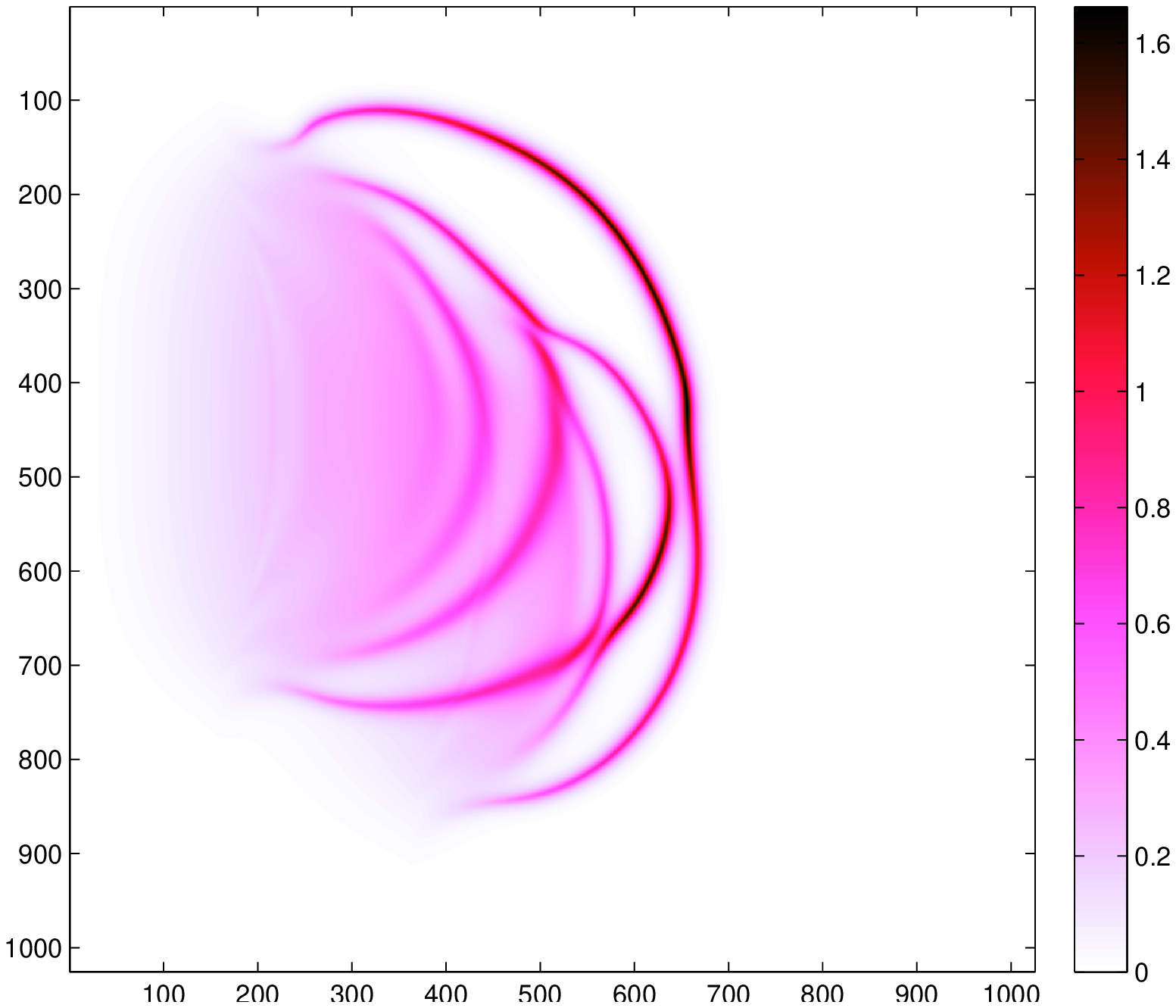}
                \caption{``Parallel'', $T=0.533$, $\alpha = \frac{\sigma}{8}$}
                \label{fig:par1_25_surf}
        \end{subfigure}
        \begin{subfigure}[b]{0.49\textwidth}
                \includegraphics[width=\textwidth]{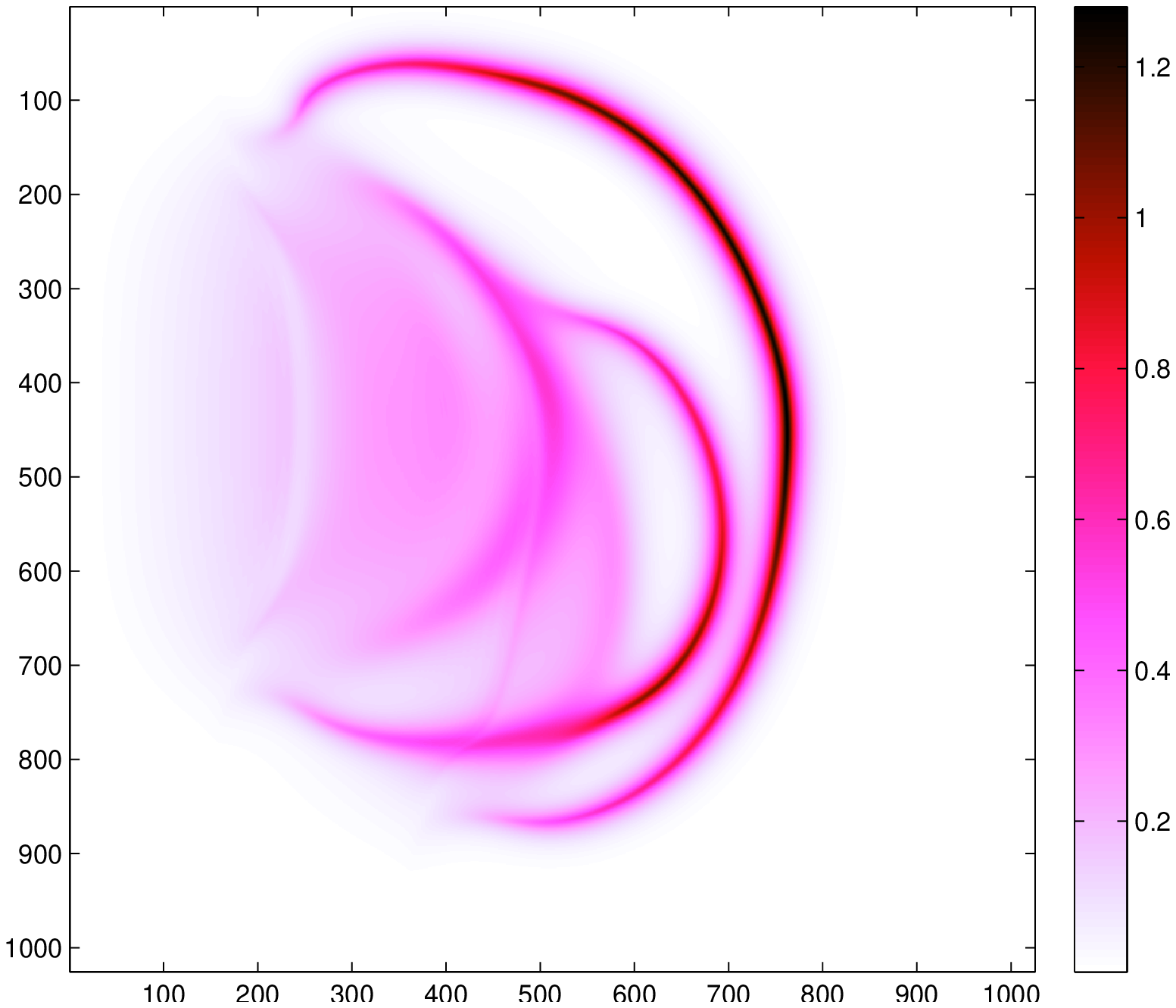}
                \caption{``Parallel'', $T=0.8$, $\alpha = \frac{\sigma}{4}$}
                \label{fig:par1_2f}
        \end{subfigure}
        \begin{subfigure}[b]{0.49\textwidth}
                \includegraphics[width=\textwidth]{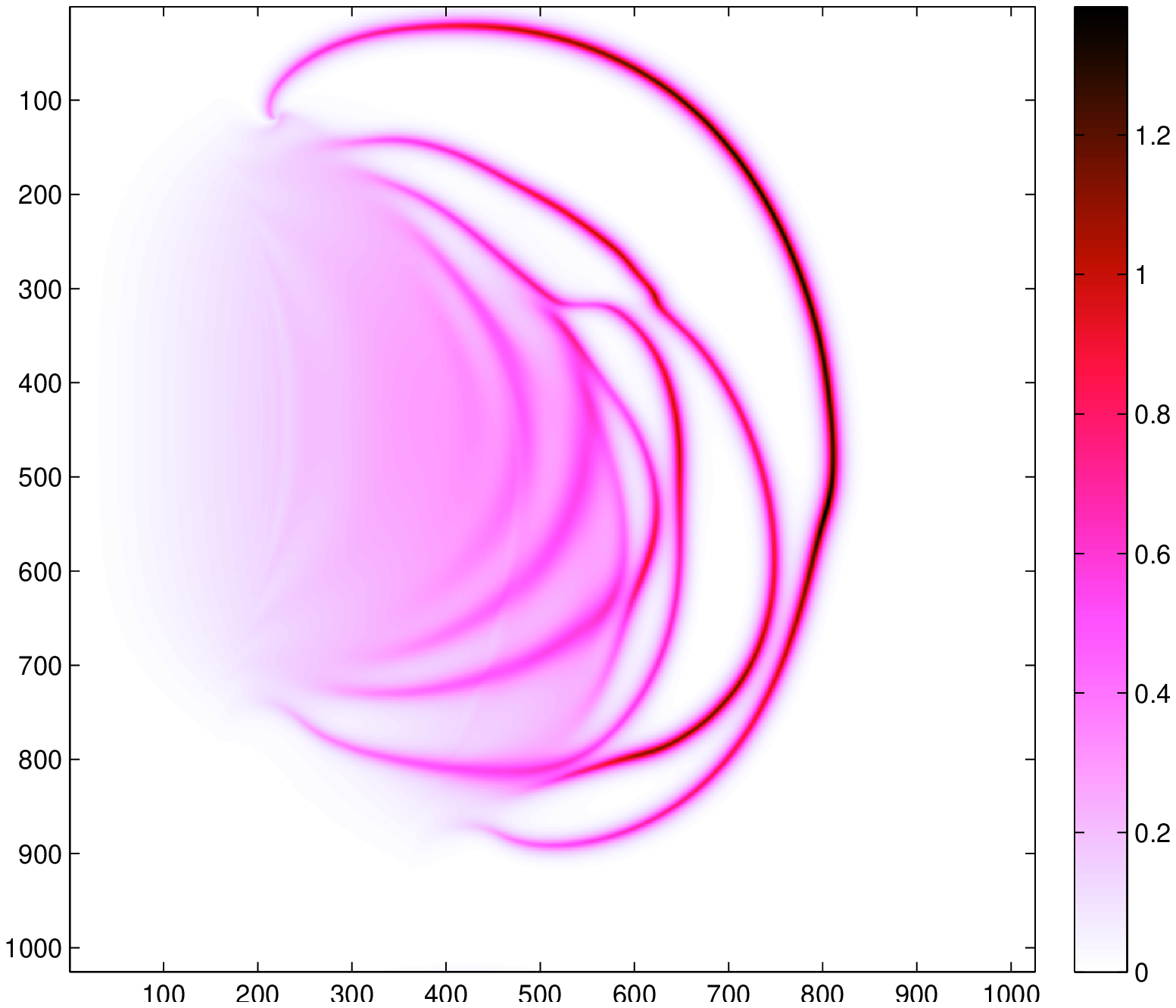}
                \caption{``Parallel'', $T=0.8$, $\alpha = \frac{\sigma}{8}$}
                \label{fig:par1_2f_surf}
        \end{subfigure}
\caption{Evolution of ``Parallel'' on a grid $1025 \times 1025$, Scheme 
2}
\label{fig:ev_scheme2_alpha=sigma/4,sigma/8_parallel}
\end{figure}


\begin{figure}
        \centering
        \begin{subfigure}[b]{0.49\textwidth}
                \includegraphics[width=\textwidth]{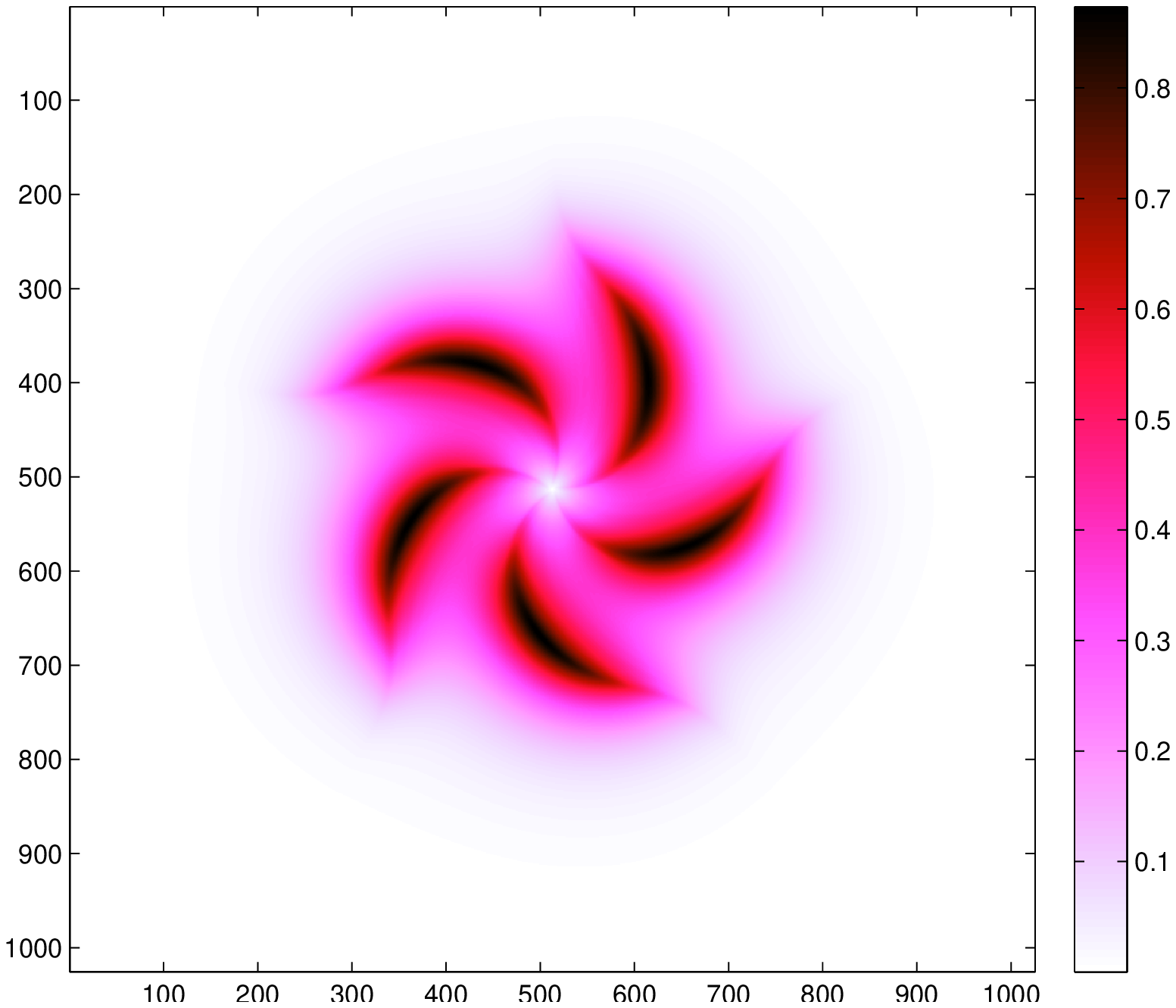}
                \caption{``Star'', $T=0.2$, $\alpha = \sigma$}
                \label{fig:star1_22}
        \end{subfigure}
        \begin{subfigure}[b]{0.49\textwidth}
                \includegraphics[width=\textwidth]{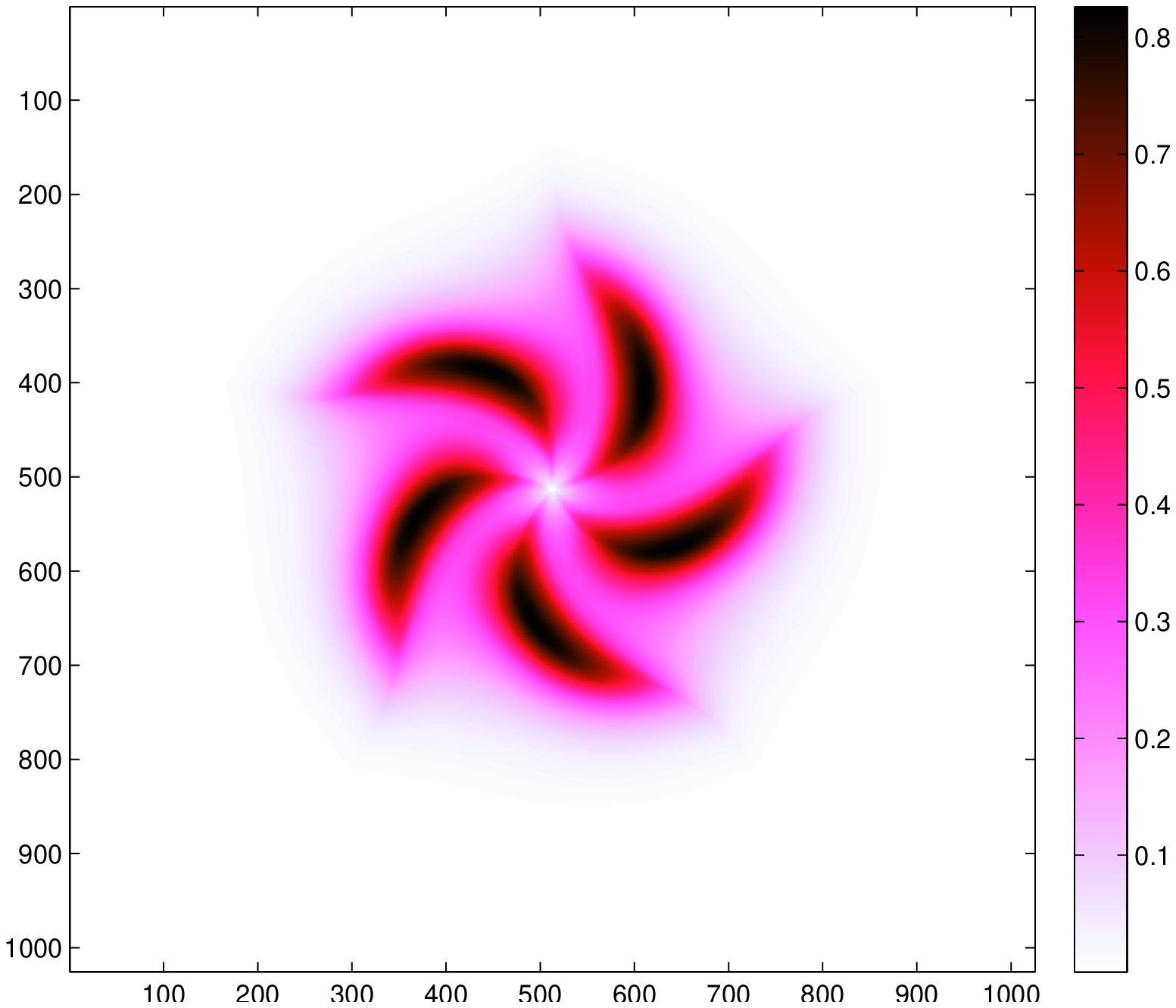}
                \caption{``Star'', $T=0.2$, $\alpha = \frac{\sigma}{2}$}
                \label{fig:star1_22_surf}
        \end{subfigure}
        \begin{subfigure}[b]{0.49\textwidth}
                \includegraphics[width=\textwidth]{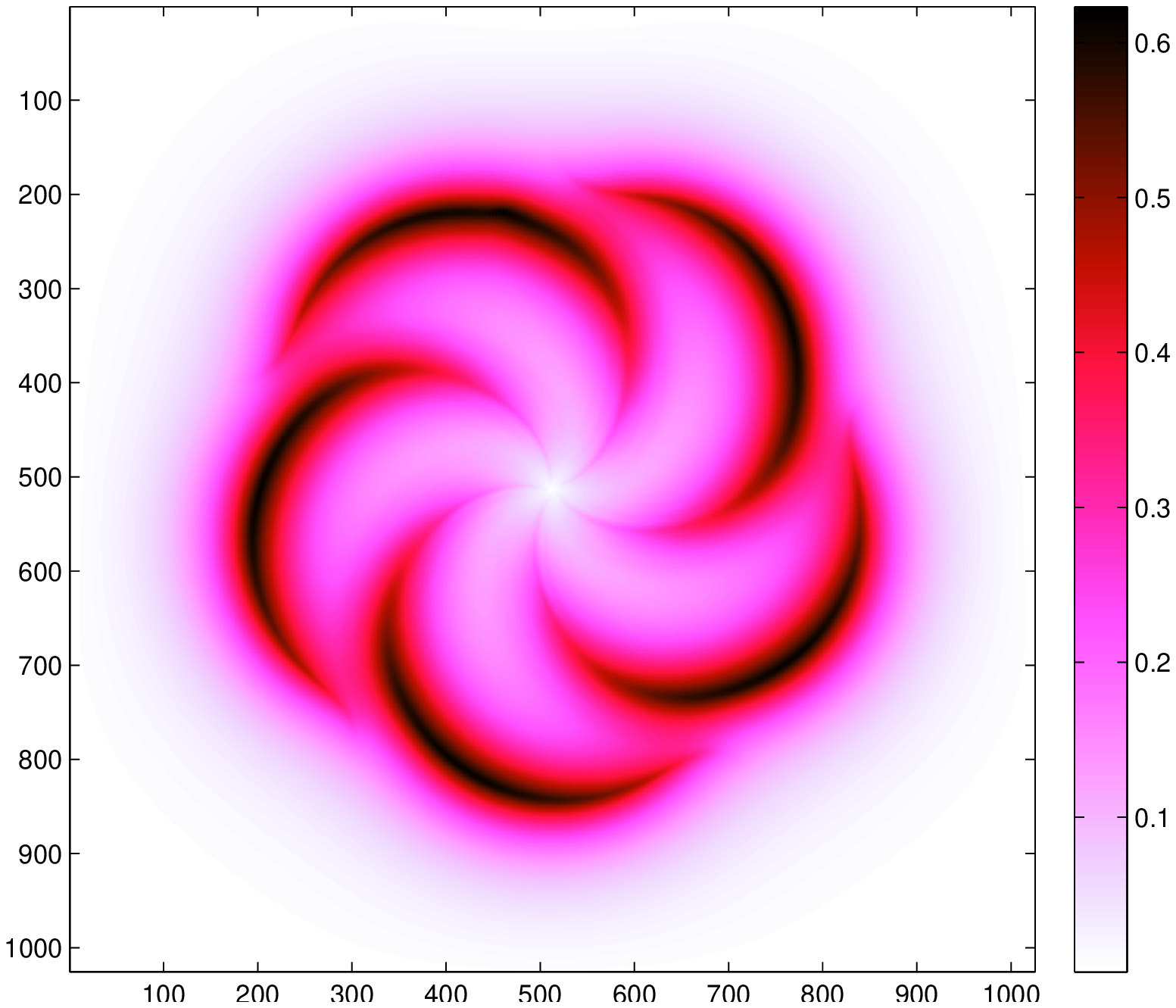}
                \caption{``Star'', $T=0.6$, $\alpha = \sigma$}
                \label{fig:star1_24}
        \end{subfigure}
        \begin{subfigure}[b]{0.49\textwidth}
                \includegraphics[width=\textwidth]{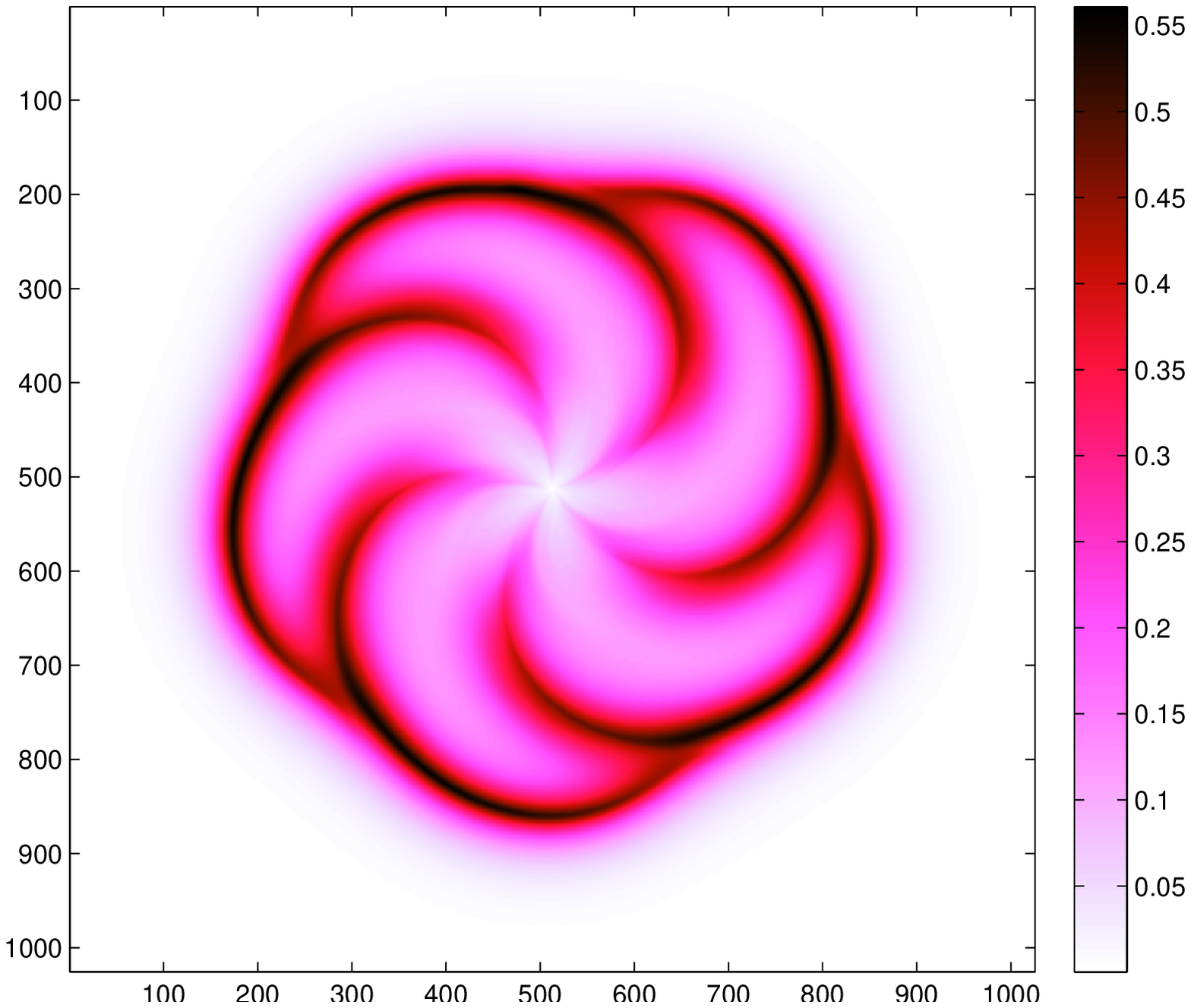}
                \caption{``Star'', $T=0.6$, $\alpha = \frac{\sigma}{2}$}
                \label{fig:star1_24_surf}
        \end{subfigure}
        \begin{subfigure}[b]{0.49\textwidth}
                \includegraphics[width=\textwidth]{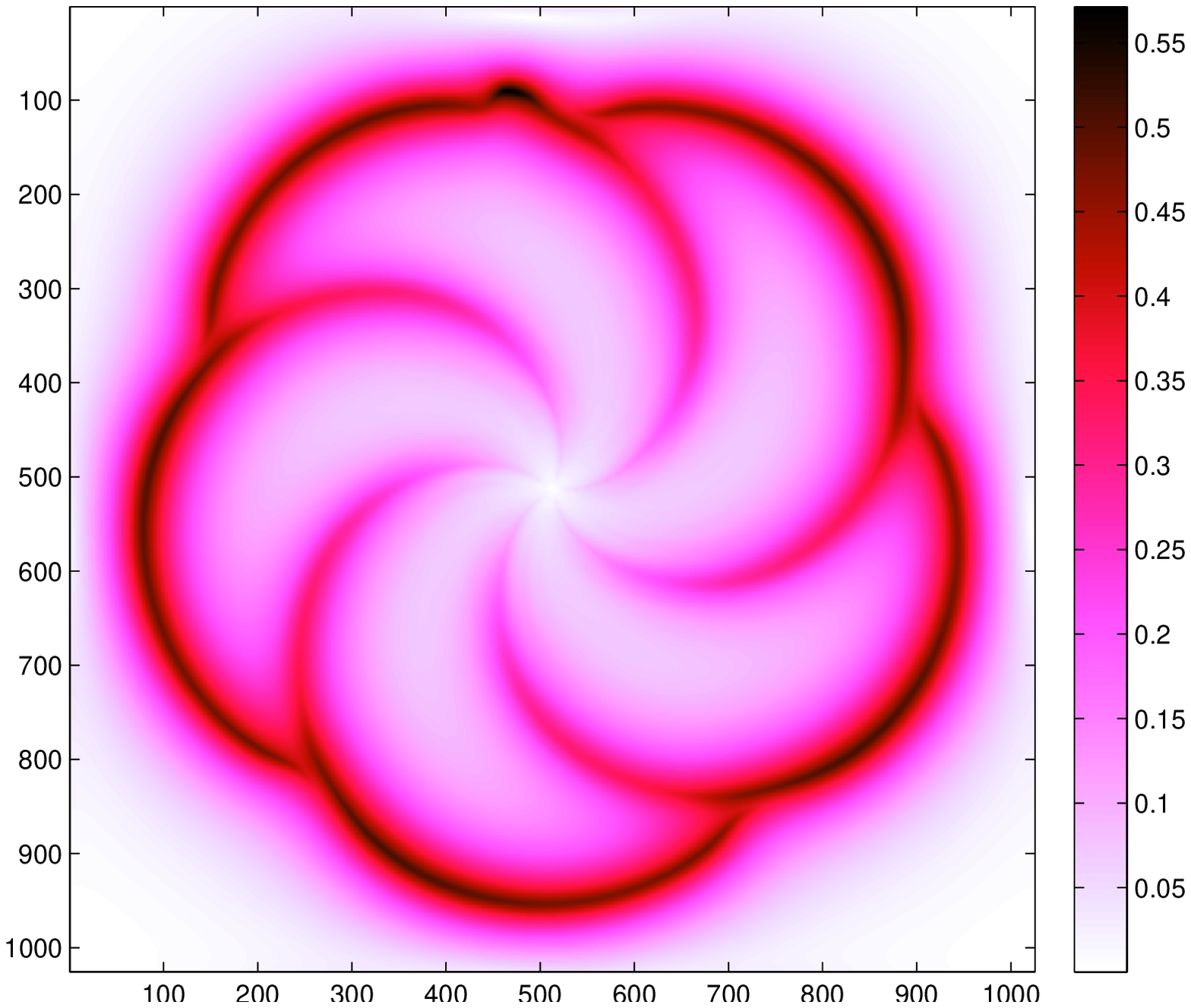}
                \caption{``Star'', $T=1$, $\alpha = \sigma$}
                \label{fig:star1_26}
        \end{subfigure}
        \begin{subfigure}[b]{0.49\textwidth}
                \includegraphics[width=\textwidth]{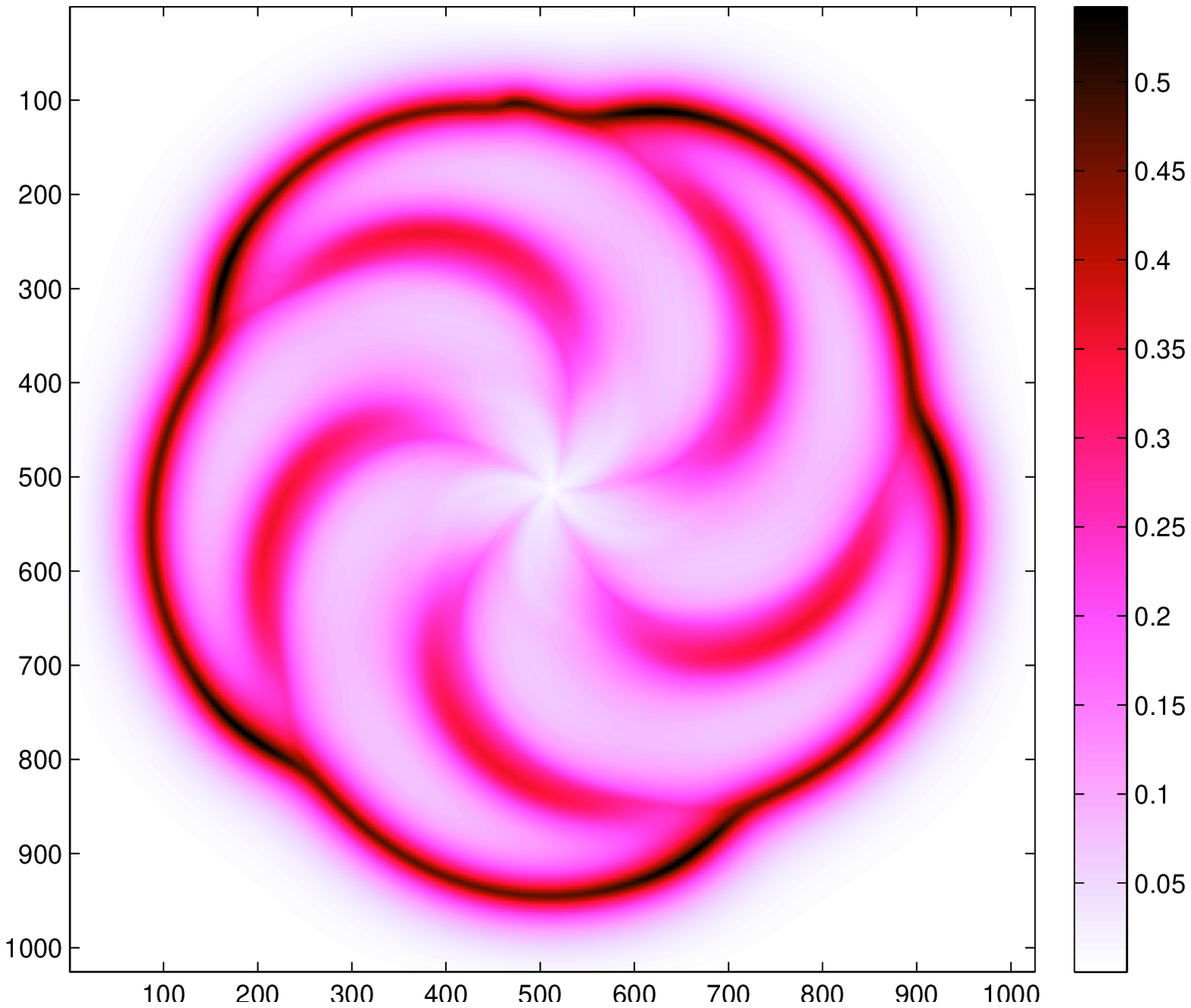}
                \caption{``Star'', $T=1$, $\alpha = \frac{\sigma}{2}$}
                \label{fig:star1_26_surf}
        \end{subfigure}
\caption{Evolution of ``Star'' on a grid $1025 \times 1025$, Scheme 
2}
\label{fig:ev_scheme2_alpha=sigma_star}
\end{figure}


\begin{figure}
        \centering
        \begin{subfigure}[b]{0.49\textwidth}
                \includegraphics[width=\textwidth]{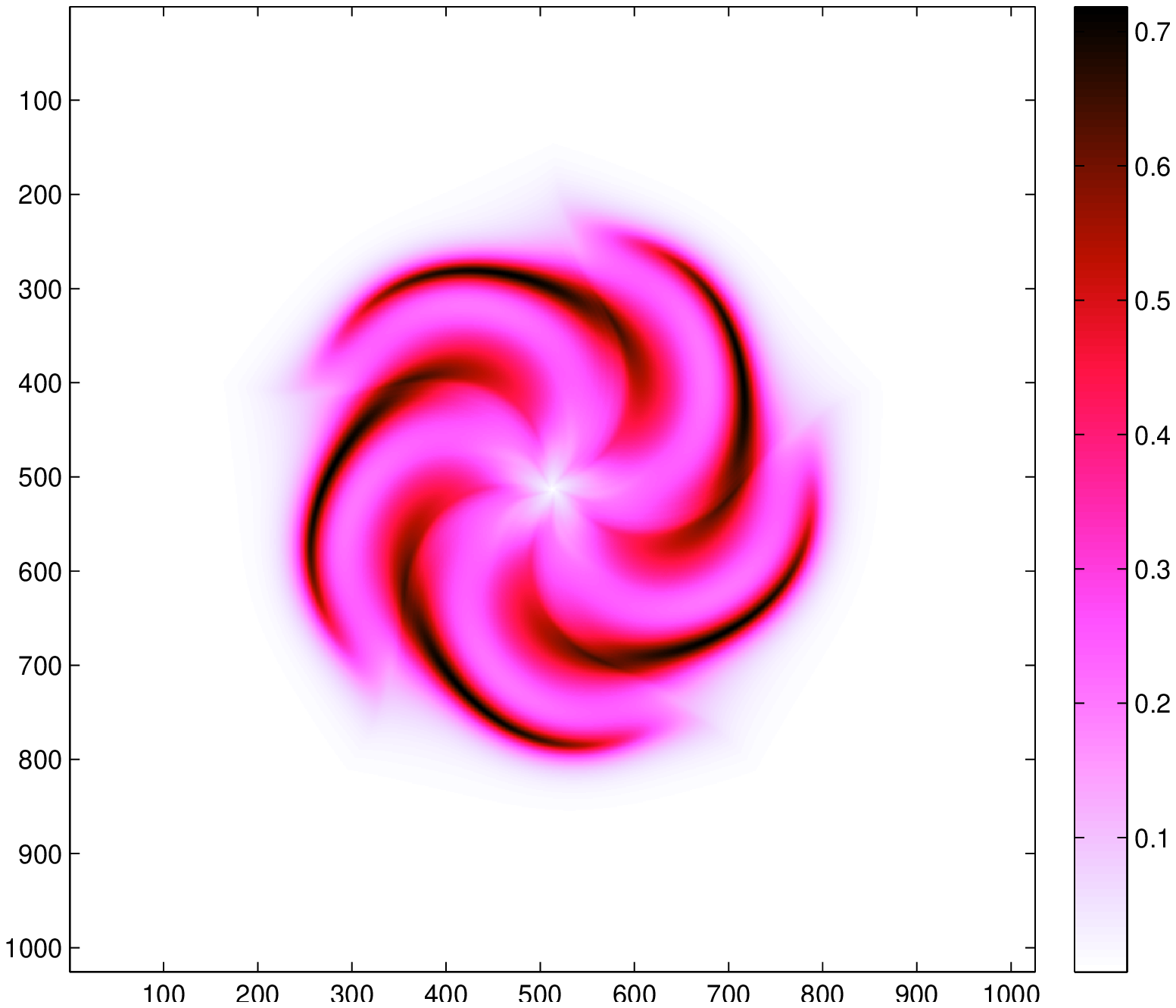}
                \caption{``Star'', $T=0.333$, $\alpha = \frac{\sigma}{4}$}
                \label{fig:star4_22}
        \end{subfigure}
        \begin{subfigure}[b]{0.49\textwidth}
                \includegraphics[width=\textwidth]{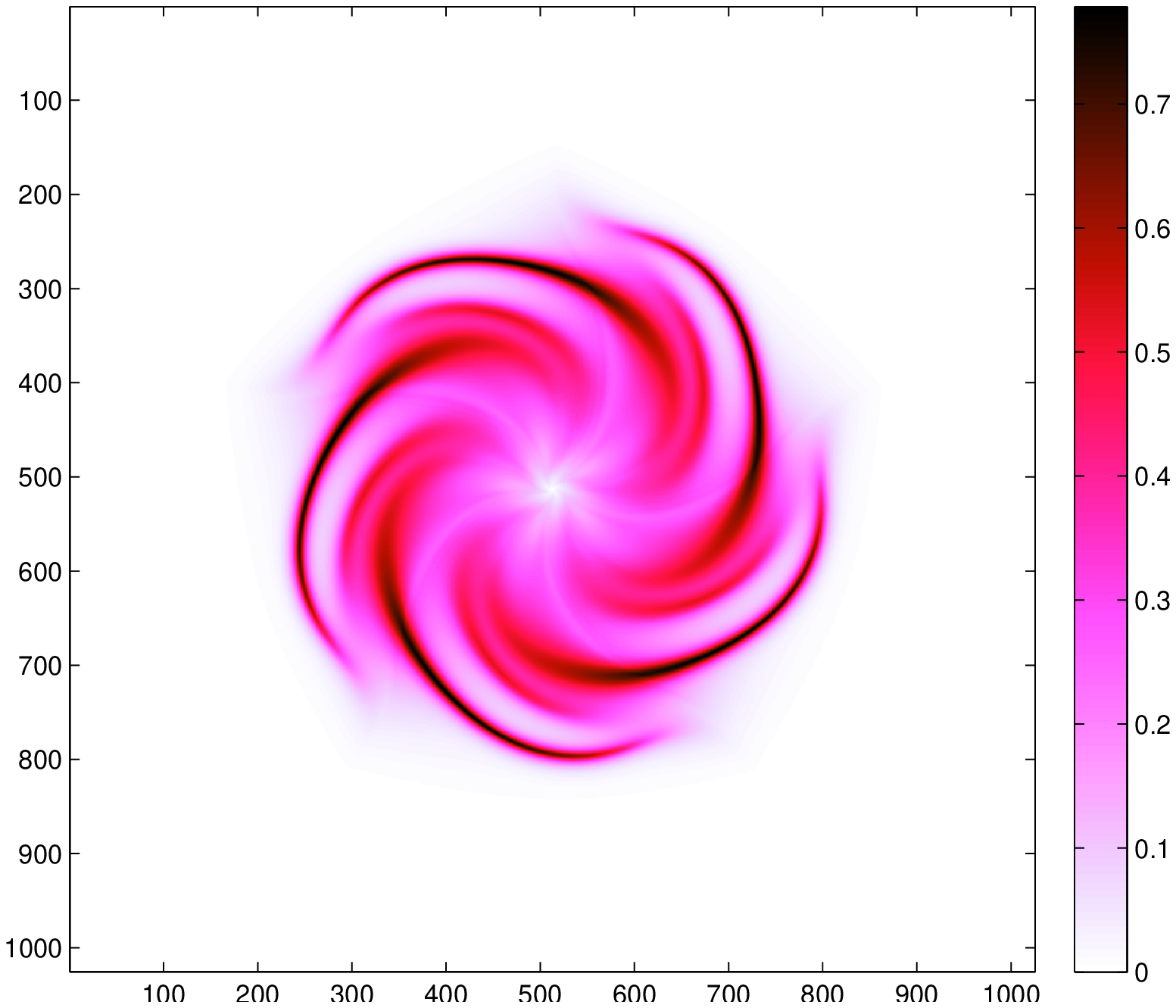}
                \caption{``Star'', $T=0.1583$, $\alpha = \frac{\sigma}{8}$}
                \label{fig:star4_22_surf}
        \end{subfigure}
        \begin{subfigure}[b]{0.49\textwidth}
                \includegraphics[width=\textwidth]{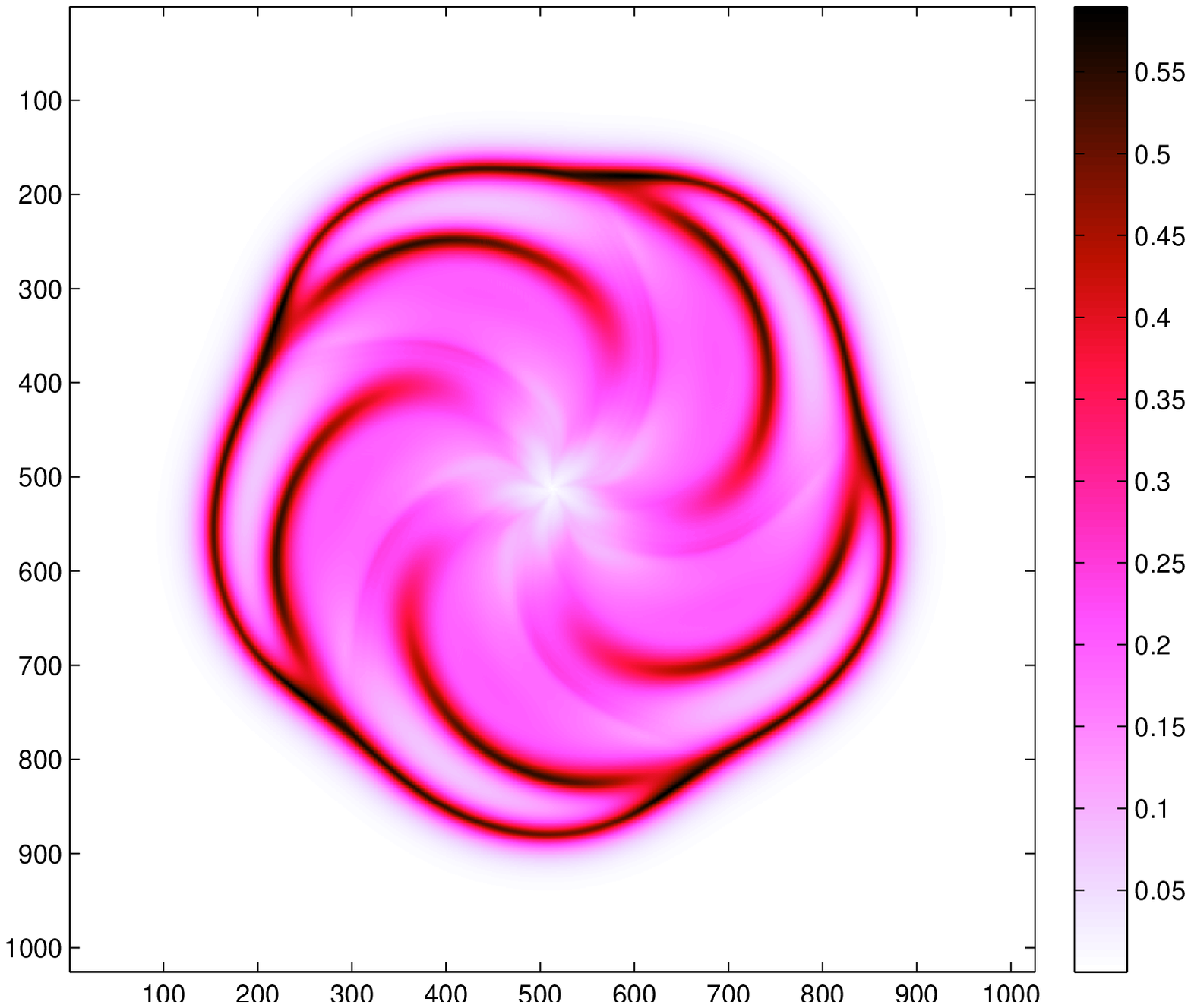}
                \caption{``Star'', $T=0.666$, $\alpha = \frac{\sigma}{4}$}
                \label{fig:star4_24}
        \end{subfigure}
        \begin{subfigure}[b]{0.49\textwidth}
                \includegraphics[width=\textwidth]{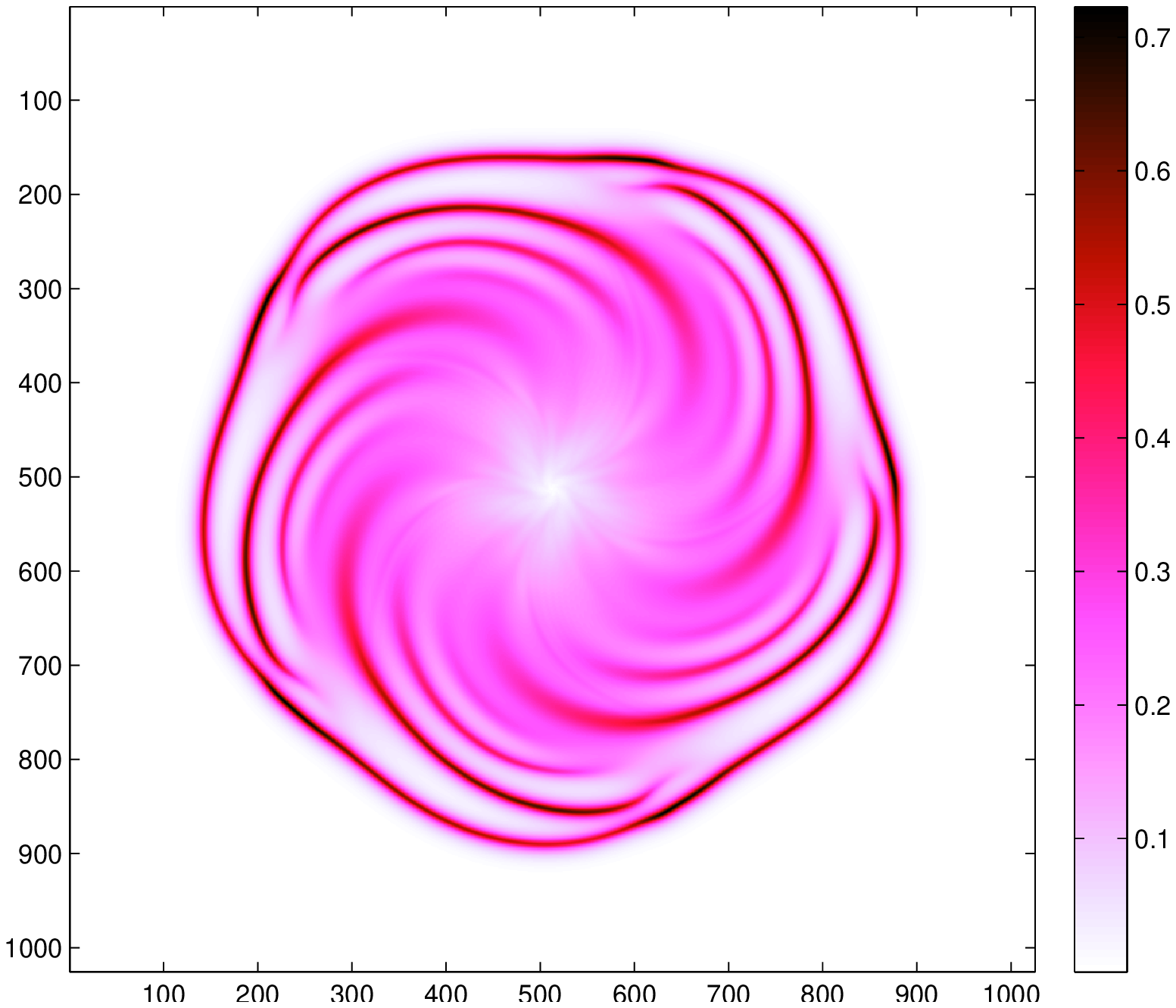}
                \caption{``Star'', $T=0.4750$, $\alpha = \frac{\sigma}{8}$}
                \label{fig:star4_24_surf}
        \end{subfigure}
        \begin{subfigure}[b]{0.49\textwidth}
                \includegraphics[width=\textwidth]{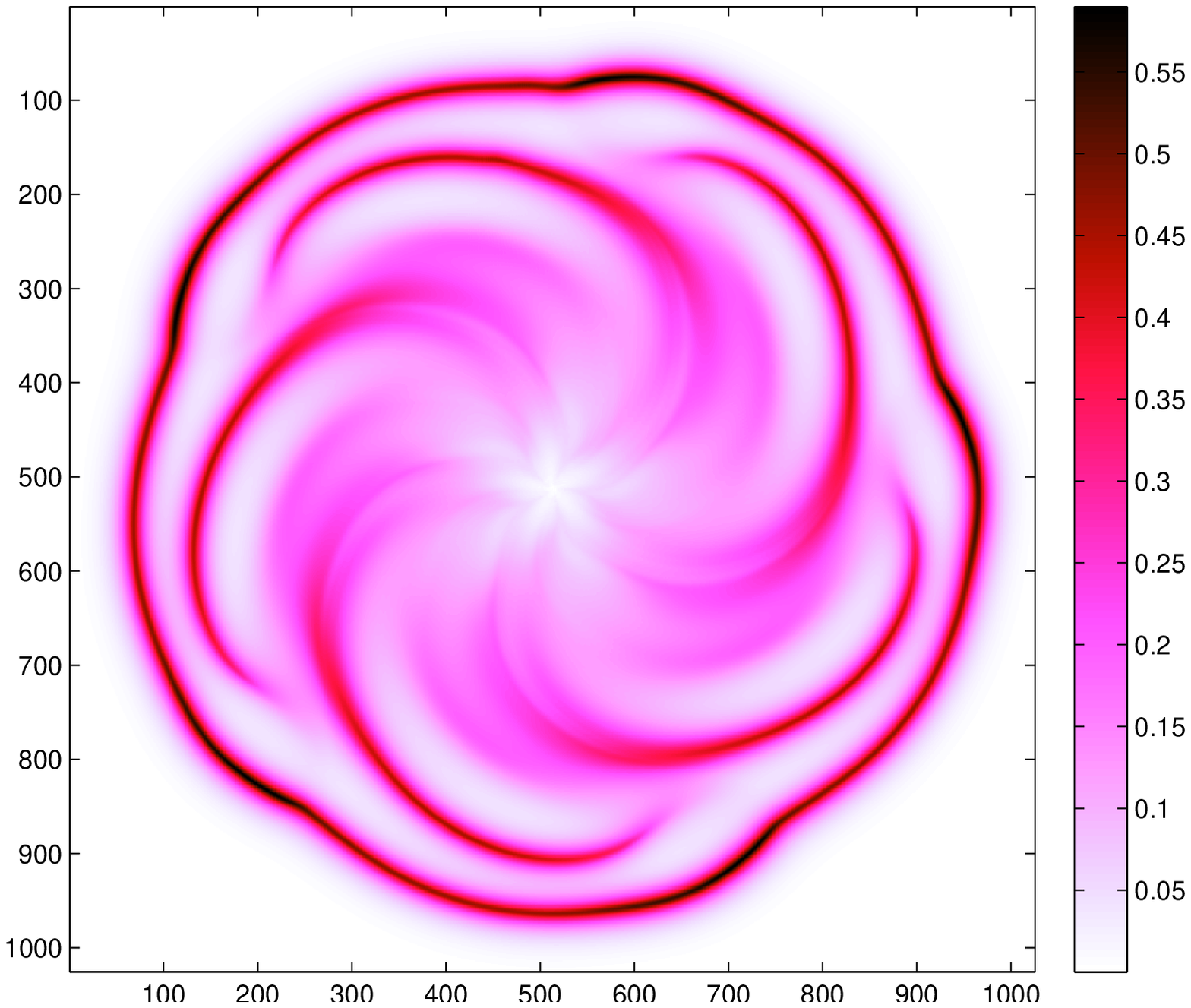}
                \caption{``Star'', $T=1$, $\alpha = \frac{\sigma}{4}$}
                \label{fig:star4_26}
        \end{subfigure}
        \begin{subfigure}[b]{0.49\textwidth}
                \includegraphics[width=\textwidth]{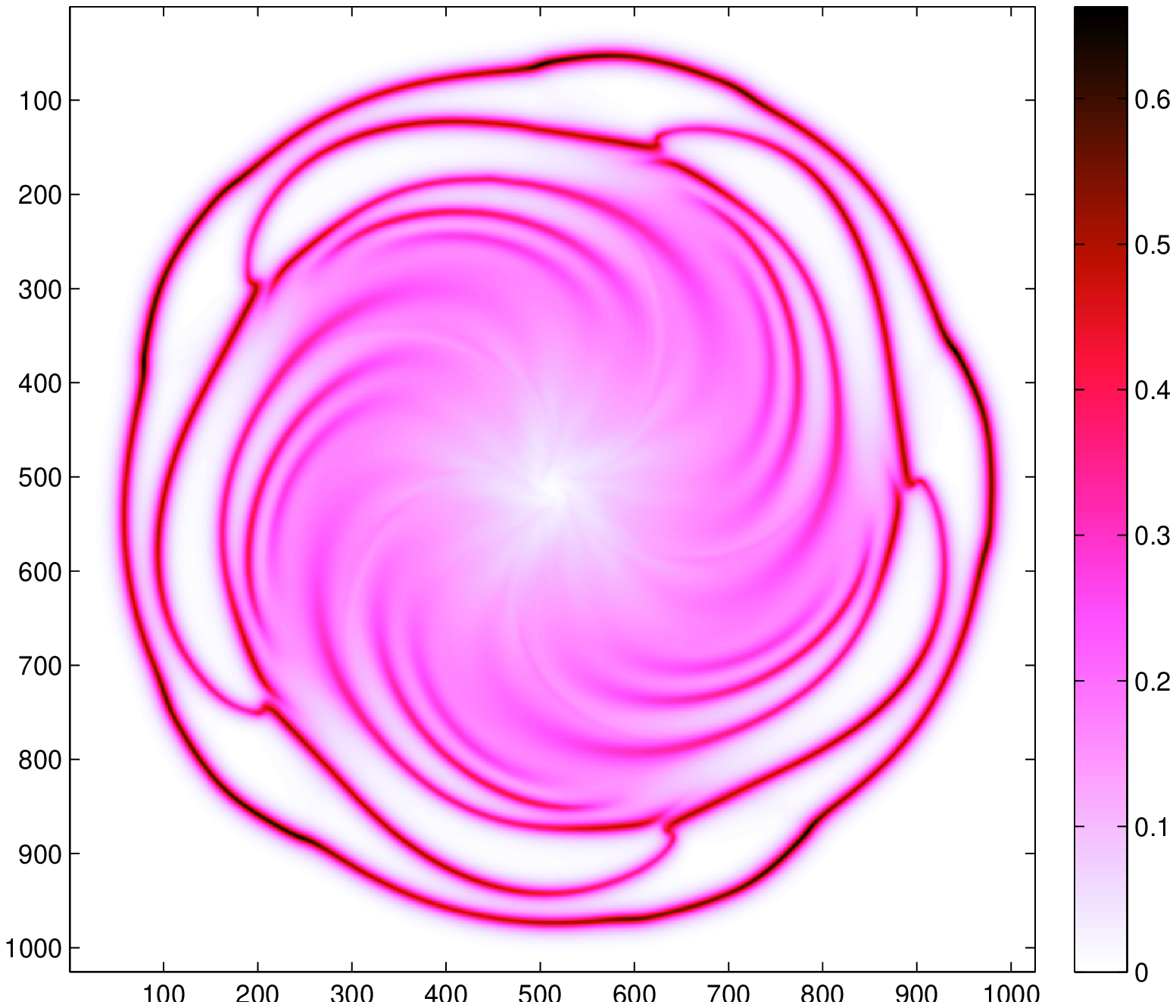}
                \caption{``Star'', $T=0.95$, $\alpha = \frac{\sigma}{8}$}
                \label{fig:star4_26_surf}
        \end{subfigure}
\caption{Evolution of ``Star'' on a grid $1025 \times 1025$, Scheme 
2}
\label{fig:ev_scheme2_alpha=sigma/4_star}
\end{figure}

\clearpage

\subsection{Empirical Convergence Analysis}

We test the rate of convergence for the schemes presented in this
manuscript. To simplify the study, we consider only the convergence rate for
the initial profile corresponding to the ``Plate'' case, when $\alpha=\sigma$.
The rate of convergence appears to be linear with respect to $\Delta t = 
\Delta x = \Delta y$ for all the schemes, up to some constant factor (see 
Figure~\ref{fig:convergence}). The apparent better rate of convergence for 
small values of $\Delta t$ is probably due to having a number of grid points 
too close to the ones used to compute the reference solution. This is 
seemingly suboptimal, if compared with \cite{Yuto}, where the rate of 
convergence is $\mathscr{O}(\Delta t ^2 + \Delta x ^ 2)$, but we have to keep 
in account the two following facts:
\begin{itemize}
 \item The equation investigated in \cite{Yuto} is the modified 
Camassa--Holm equation, that is, $m = (1-\partial_{xx}^2)^2$. This, as a 
geodesic equation, has a smoother metric, which could affect the rate of 
convergence of the method.
\item The assumption of regularity in \cite{Yuto} are $\mathscr{C}^7$ with 
respect to the spatial component and $\mathscr{C}^3$ with respect to 
the temporal component, while our convergence analysis is performed with almost 
singular initial profiles.  
\end{itemize}
In view of this, it appears reasonable having $\mathscr{O}(\Delta t + \Delta 
x )$ as rate of convergence for our test.
\begin{figure}
        \centering
                \includegraphics[width=\textwidth]{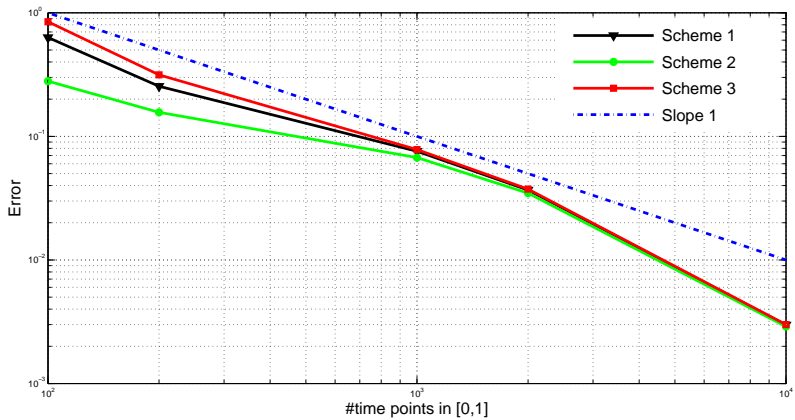}
                \caption{Empirical convergence analysis.}
                \label{fig:convergence}
\end{figure}

\subsection{Empirical Reversibility Analysis}

The results about reversibility that we present in this section are measured 
with 
respect to the $\norm{2}{\cdot}$-norm and with a discretization such that 
$\frac{\Delta t}{\Delta x} = \frac{1}{4} $.

In Table~\ref{tab:rev_scheme2} and~\ref{tab:rev_scheme3} we report the absolute
and the relative error, where the relative error is obtained by dividing the
absolute error by the norm of the initial profile.

We report also the results about reversibility for the first scheme implemented 
as a predictor-corrector with a fixed number of iterations (see 
Table~\ref{tab:rev_schemePC}). Although we
do not expect reversibility, we can notice that the results we obtain are not 
too far off from the ones in 
\ref{tab:rev_scheme2} and \ref{tab:rev_scheme3}, as long as the profiles we 
start with are ``simple'' and the ratio $\frac{\alpha}{\sigma}$ is not too 
small. However, when $\frac{\alpha}{\sigma}$ gets too small, as in the case 
$\frac{\alpha}{\sigma} = \frac18$, we can see that the scheme is no longer 
reversible. This problem can be avoided by simply implementing Scheme 1 with a 
variable number of iteration and a control over the relative error, as done in 
the previous section. We report in Table~\ref{tab:rev_scheme1_comparison} an 
example of comparison of performances for two different implementation of 
scheme number one. We want to stress that this takes in general much more time 
than what required by having the number of iteration fixed and equal to $5$. On 
average such an implementation required $23$ iterations of the corrector step, 
in order to reach the desired tolerance of $10^{-14}$. We refrain from reporting a 
complete table with the reversibility test for Scheme 1 with a variable number 
of corrector iterations. We limit ourselves to observe that in general, if 
$\Delta t$ is small enough (accordingly to the spatial discretization, the 
norm of the initial profile and the ratio $\frac{\alpha}{\sigma}$), the 
results about reversibility improve drastically. An example of how this happens 
is reported in Table~\ref{tab:rev_scheme2_comparison}, where the same 
experiment for different values of $\Delta t$.

\begin{table}[htbp]
\caption{Reversibility of Scheme 1, fixed number $(5)$ of corrections}
\label{tab:rev_schemePC} 
\begin{center}
  \begin{tabular}{ | l || c | c | c | c | }
    \hline
Grid:	$200 \times 200$  & $\alpha=\sigma$ & $\alpha=\frac{\sigma}{2}$ &
$\alpha=\frac{\sigma}{4}$ & $\alpha=\frac{\sigma}{8}$ \\ \hline
    Plate ($\%$)
	  & $0.0027$ & $0.0231$ & $0.3751$ & $6.3410$ \\
\hline
    Parallel ($\%$) 
	  & $0.0249$ & $0.1746$ & $3.0949$ & $44.5986$ \\
\hline
    Star  ($\%$) 
	  & $0.0032$ & $0.0035$ & $0.0242$ & $0.4076$ \\
\hline
  \end{tabular}
\end{center}
\end{table}

\begin{table}[htbp]
\caption{Reversibility of Scheme 2}
\label{tab:rev_scheme2} 
\begin{center}
  \begin{tabular}{ | l || c | c | c | c | }
    \hline
Grid:	$200 \times 200$  & $\alpha=\sigma$ & $\alpha=\frac{\sigma}{2}$ &
$\alpha=\frac{\sigma}{4}$ & $\alpha=\frac{\sigma}{8}$ \\ \hline
Plate ($\%$) 
    & $0.0080$ & $0.0062$ & $0.0233$ & $0.8971$ \\
\hline
Parallel ($\%$)
    & $0.0559$ & $0.0367$ & $0.1131$ & $0.7554$ \\
\hline
Star ($\%$)
    & $0.0066$ & $0.0090$ & $0.0164$ & $0.0956$ \\
\hline
  \end{tabular}
\end{center}
\end{table}

\begin{table}[htbp]
\caption{Reversibility of Scheme 3}
\label{tab:rev_scheme3} 
\begin{center}
  \begin{tabular}{ | l || c | c | c | c | }
    \hline
Grid:	$200 \times 200$  & $\alpha=\sigma$ & $\alpha=\frac{\sigma}{2}$ &
$\alpha=\frac{\sigma}{4}$ & $\alpha=\frac{\sigma}{8}$ \\ \hline
Plate ($\%$)
	  & $0.0058$ & $0.0063$ & $0.0185$ & $2.1017$ \\
\hline
Parallel ($\%$)
	  & $0.0604$ & $0.0320$ & $0.0795$ & $0.2771$ \\
\hline
Star ($\%$)
	  & $0.0096$ & $0.0111$ & $0.0209$ & $0.0623$ \\
\hline
  \end{tabular}
\end{center}
\end{table}

\begin{table}[htbp]
\caption{Reversibility of Scheme 1: different performances}
\label{tab:rev_scheme1_comparison} 
\begin{center}
  \begin{tabular}{ | l || c | c | }
    \hline
$\alpha = \frac{\sigma}{8}$ & Scheme 1 & Scheme 1\footnotemark[1] \\
\hline
Parallel ($\%$) & $0.0602$ & $44.5986$ \\
\hline
  \end{tabular}
\end{center}
\end{table}
\footnotetext[1]{With a fixed number of $5$ corrector iterations.}

\begin{table}[htbp]
\caption{Reversibility of Scheme 2: different values of $\frac{\Delta 
t}{\Delta x}$}
\label{tab:rev_scheme2_comparison} 
\begin{center}
  \begin{tabular}{ | l || c | c | }
    \hline
$\alpha = \frac{\sigma}{8}$ & $\frac{\Delta t}{\Delta x} = \frac14$ & 
$\frac{\Delta t}{\Delta x} = \frac{1}{16}$ \\
\hline
Plate ($\%$) &  $0.8971$  & $0.0021$   \\
\hline
  \end{tabular}
\end{center}
\end{table}

\subsection{Performance Analysis}
In Table~\ref{tab:cost_per_it} we report the average cost per iteration of each
of the schemes presented so far. It is clear that, although all the schemes 
have a cost per time step that grows linearly with the dimension of the 
system to solve, Scheme 2 is faster than the other schemes by approximately a factor~$10$. 
Scheme 3 and Scheme 1 appear similar in terms of performance for 
coarse grids, but for finer grids we see that Scheme 1 
(with 3 corrections at each step) is approximately twice as fast as Scheme 3.

\begin{table}[htbp]
\caption{Average cost per iteration (seconds) on different grids}
\label{tab:cost_per_it} 
\begin{center}
  \begin{tabular}{ | l || c | c | c | c | c | c | c | c | c | c | }
    \hline
    Grid-size & $100^2$ & $200^2$ & $300^2$ & $400^2$ & $500^2$ 
	  & $600^2$ & $700^2$ & $800^2$ & $900^2$ & $1000^2$  \\
\hline\hline
Scheme 1\footnotemark[2]     & $0.17$  & $0.80$ & $1.77$ & $3.45$ & $5.65$
	  & $8.50$  & $12.3$ & $16.2$ & $20.8$ & $25.1$    \\
\hline
Scheme 2  	  & $0.04$ & $0.18$ & $0.43$ & $0.85$ & $1.47$
	  & $2.11$ & $3.10$ & $4.18$ & $5.56$ & $6.48$ \\
\hline
Scheme 3  & $0.19$ & $0.95$ & $2.55$ & $5.15$ & $8.73$
	  & $14.5$ & $20.5$ & $30.5$ & $44$ & $60$   \\
\hline
  \end{tabular}
\end{center}
\end{table}
\footnotetext[2]{With a fixed number of $3$ corrector iterations.}

\section{Conclusions}
In this manuscript we have developed a multidimensional version of three
different integrators originally meant to solve the Camassa--Holm equation, and
now adapted to integrate the EPDiff equation in an arbitrary number of
dimension. We proved that our schemes admit a unique solution, preserve the 
numerical energy 
of the equation, and that two of them also preserve the momenta. The 
theoretical results, together with an analysis of the reversibility, are also 
verified empirically for a wide selection of benchmark problems. 

Our study reveals that Scheme~2 is a likely method-of-choice, since it produces results as 
accurate as the other two schemes at a cost per iteration which is a tenth of 
that of Scheme~3 and a fifth of that of Scheme~1 implemented in a predictor-corrector routine with fixed number 
of iterations, and since it possesses both the property of being revertible and
the property of conserving both the energy and the momenta. However, the better 
stability of Scheme~3 and the conservation of the ``real numerical energy'' of 
Scheme~1 suggest that these two schemes are not out of the game, and 
might be worth considering depending on the applications under consideration.

\appendix

\section{Omitted Proofs}\label{sec:om_proofs}
\begin{tiny}
\subsection{Proof of 
Lemma~\ref{lemma:first_equality_scheme1}} 
\label{sec:om_proofs_first_equality_scheme1}
\begin{proof}
\begin{align*}
&\frac{1}{\Delta t}\Big( 
\sum_{j=0}^{\cJ-1}\sum_{k=0}^{\cK-1}H_{k,j}^{(n+1)}\Delta x 
\Delta y - \sum_{j=0}^{\cJ-1}\sum_{k=0}^{\cK-1}H_{k,j}^{(n)}\Delta x 
\Delta y\Big) \\
&=\frac{1}{2\Delta t}\Big( \sum_{j=0}^{\cJ-1}\sum_{k=0}^{\cK-1}  
M_{1;k,j}^{(n+1)} 
U_{1;k,j}^{(n+1)} -  M_{1;k,j}^{(n)} U_{1;k,j}^{(n)} \\
&\qquad \qquad + M_{2;k,j}^{(n+1)} 
U_{2;k,j}^{(n+1)} - M_{2;k,j}^{(n)} 
U_{2;k,j}^{(n)} \Big)\Delta x 
\Delta y\\
&=\frac{1}{4\Delta t}\Big( \sum_{j=0}^{\cJ-1}\sum_{k=0}^{\cK-1}  
M_{1;k,j}^{(n+1)} 
U_{1;k,j}^{(n+1)} + M_{1;k,j}^{(n)}U_{1;k,j}^{(n+1)} 
- M_{1;k,j}^{(n+1)}U_{1;k,j}^{(n)} - M_{1;k,j}^{(n)}U_{1;k,j}^{(n)}\\
&\qquad \qquad   + M_{1;k,j}^{(n+1)} 
U_{1;k,j}^{(n+1)} - M_{1;k,j}^{(n)}U_{1;k,j}^{(n+1)} 
+ M_{1;k,j}^{(n+1)}U_{1;k,j}^{(n)} - M_{1;k,j}^{(n)}U_{1;k,j}^{(n)}\\
&\qquad \qquad  + M_{2;k,j}^{(n+1)} 
U_{2;k,j}^{(n+1)} + M_{2;k,j}^{(n)}U_{2;k,j}^{(n+1)} 
- M_{2;k,j}^{(n+1)}U_{2;k,j}^{(n)} - M_{2;k,j}^{(n)}U_{2;k,j}^{(n)}\\
&\qquad \qquad   + M_{2;k,j}^{(n+1)} 
U_{2;k,j}^{(n+1)} - M_{2;k,j}^{(n)}U_{2;k,j}^{(n+1)} 
+ M_{2;k,j}^{(n+1)}U_{2;k,j}^{(n)} - M_{2;k,j}^{(n)}U_{1;k,j}^{(n)}\Big)\Delta 
x 
\Delta y.
\end{align*}
We can now factorize and get
\begin{align*}
&\frac{1}{4\Delta t}\Big( \sum_{j=0}^{\cJ-1}\sum_{k=0}^{\cK-1}  
(M_{1;k,j}^{(n+1)} 
+ M_{1;k,j}^{(n)})U_{1;k,j}^{(n+1)} 
- (M_{1;k,j}^{(n+1)} + M_{1;k,j}^{(n)})U_{1;k,j}^{(n)}\\
&\qquad \qquad   + (M_{1;k,j}^{(n+1)} - M_{1;k,j}^{(n)})U_{1;k,j}^{(n+1)} 
+ (M_{1;k,j}^{(n+1)} - M_{1;k,j}^{(n)}) U_{1;k,j}^{(n)}\\
&\qquad \qquad  M_{2;k,j}^{(n+1)} 
+ M_{2;k,j}^{(n)})U_{2;k,j}^{(n+1)} 
- (M_{2;k,j}^{(n+1)} + M_{2;k,j}^{(n)})U_{2;k,j}^{(n)}\\
&\qquad \qquad   + (M_{2;k,j}^{(n+1)} - M_{2;k,j}^{(n)})U_{2;k,j}^{(n+1)} 
+ (M_{2;k,j}^{(n+1)} - M_{2;k,j}^{(n)}) U_{2;k,j}^{(n)} \Big)\Delta x 
\Delta y,
\end{align*}
that is,
\begin{align*}
&\frac{1}{2\Delta t}\Big( \sum_{j=0}^{\cJ-1}\sum_{k=0}^{\cK-1}  
\frac{(M_{1;k,j}^{(n+1)} 
+ M_{1;k,j}^{(n)})}{2} (U_{1;k,j}^{(n+1)} - U_{1;k,j}^{(n)})\\
&\qquad \qquad \qquad  + (M_{1;k,j}^{(n+1)} - M_{1;k,j}^{(n)}) 
\frac{( U_{1;k,j}^{(n+1)} 
+  U_{1;k,j}^{(n)})}{2}\\
&\qquad \qquad \qquad\frac{(M_{2;k,j}^{(n+1)} 
+ M_{2;k,j}^{(n)})}{2} (U_{2;k,j}^{(n+1)} - U_{2;k,j}^{(n)})\\
&\qquad \qquad \qquad  + (M_{2;k,j}^{(n+1)} - M_{2;k,j}^{(n)}) 
\frac{( U_{2;k,j}^{(n+1)} 
+  U_{2;k,j}^{(n)})}{2} \Big)\Delta x 
\Delta y,\\
&= \frac{1}{2\Delta t}\Big( \sum_{j=0}^{\cJ-1}\sum_{k=0}^{\cK-1}  
M_{1;k,j}^{(n+\frac12)}(U_{1;k,j}^{(n+1)} - U_{1;k,j}^{(n)}\Big)\Delta x 
\Delta y\\
&+ \frac{1}{2\Delta t}\Big( 
\sum_{j=0}^{\cJ-1}\sum_{k=0}^{\cK-1}  
U_{1;k,j}^{(n+\frac12)}(M_{1;k,j}^{(n+1)} - M_{1;k,j}^{(n)}\Big)\Delta x 
\Delta y\\
&+ \frac{1}{2\Delta t}\Big( \sum_{j=0}^{\cJ-1}\sum_{k=0}^{\cK-1}  
M_{2;k,j}^{(n+\frac12)}(U_{2;k,j}^{(n+1)} - U_{2;k,j}^{(n)}\Big)\Delta x 
\Delta y\\
&+ \frac{1}{2\Delta t}\Big( 
\sum_{j=0}^{\cJ-1}\sum_{k=0}^{\cK-1}  
U_{2;k,j}^{(n+\frac12)}(M_{2;k,j}^{(n+1)} - M_{2;k,j}^{(n)}\Big)\Delta x 
\Delta y.
\end{align*}
We use now the fact that, component-wise $ M_{\cdot,jk}^{(\cdot)} = 
(1 - \alpha^2\delta^{<2>}_{kk} -\alpha^2 \delta^{<2>}_{jj} 
)U_{\cdot,jk}^{(\cdot)}  $ and, 
analogously to what happens in the one dimensional case, it holds that
\begin{align*}
\sum_{k=0}^{\cK-1} f_k (\delta^{<2>}_{kk} g_k) \Delta x &= - 
\sum_{k=0}^{\cK-1}\frac{ (\delta^{+}_k f_k) (\delta^{+}_k g_k) + (\delta^{-}_k 
g_k)(\delta^{-}_k g_k)}{2}\Delta x \\ &= \sum_{k=0}^{\cK-1} g_k 
(\delta^{<2>}_{kk} 
f_k) \Delta x. 
\end{align*}
The last expression becomes therefore 
\begin{align*}
&\frac{1}{2\Delta t}\Big( \sum_{j=0}^{\cJ-1}\sum_{k=0}^{\cK-1} 
U_{1;k,j}^{(n+\frac12)}(M_{1;k,j}^{(n+1)} - M_{1;k,j}^{(n)})\Delta x 
\Delta y\\
&\quad \, \, \, +
\sum_{j=0}^{\cJ-1}\sum_{k=0}^{\cK-1}  
U_{1;k,j}^{(n+\frac12)}(M_{1;k,j}^{(n+1)} - M_{1;k,j}^{(n)})\Delta x 
\Delta y\\
&\quad \, \, \,+ \sum_{j=0}^{\cJ-1}\sum_{k=0}^{\cK-1}  
U_{2;k,j}^{(n+\frac12)}(M_{2;k,j}^{(n+1)} - M_{2;k,j}^{(n)})\Delta x 
\Delta y\\
&\quad \, \, \,+ 
\sum_{j=0}^{\cJ-1}\sum_{k=0}^{\cK-1}  
U_{2;k,j}^{(n+\frac12)}(M_{2;k,j}^{(n+1)} - M_{2;k,j}^{(n)})\Delta x 
\Delta y\Big),
\end{align*}
that reduces to
\begin{align*}
 \sum_{j=0}^{\cJ-1}\sum_{k=0}^{\cK-1} \Big(
U_{1;k,j}^{(n+\frac12)} \frac{M_{1;k,j}^{(n+1)} - M_{1;k,j}^{(n)}}{\Delta t} + 
U_{2;k,j}^{(n+\frac12)}\frac{M_{2;k,j}^{(n+1)} - M_{2;k,j}^{(n)}}{\Delta 
t}\Big)\Delta x 
\Delta y.
\end{align*}
\end{proof}

\subsection{Proof of 
Lemma~\ref{lemma:first_equality_scheme2}} 
\label{sec:om_proofs_first_equality_scheme2}
\begin{proof}
\begin{align*}
&\sum_{j=0}^{\cJ-1}\sum_{k=0}^{\cK-1}H_{k,j}^{(n +\frac12 )}\Delta x 
\Delta y - \sum_{j=0}^{\cJ-1}\sum_{k=0}^{\cK-1}H_{k,j}^{(n- \frac12 )}\Delta x 
\Delta y \\
&=\frac14 \sum_{j=0}^{\cJ-1}\sum_{k=0}^{\cK-1}\Big(  
M_{1;k,j}^{(n+1)} U_{1;k,j}^{(n)} 
+ M_{1;k,j}^{(n)} U_{1;k,j}^{(n+1)} 
+ M_{2;k,j}^{(n+1)} U_{2;k,j}^{(n)}
+ M_{2;k,j}^{(n)} U_{2;k,j}^{(n+1)}\\
&\qquad \qquad -  M_{1;k,j}^{(n)} U_{1;k,j}^{(n-1)}
- M_{1;k,j}^{(n-1)} U_{1;k,j}^{(n)}
- M_{2;k,j}^{(n)} U_{2;k,j}^{(n-1)}
- M_{2;k,j}^{(n-1)} U_{2;k,j}^{(n)} \Big)   \Delta x  
\Delta y.
\end{align*}
We factor out the terms involving $U_{2;k,j}^{(n)}$, thus getting
\begin{align*}
&\frac12 \sum_{j=0}^{\cJ-1}\sum_{k=0}^{\cK-1}\Big(  
\frac{M_{1;k,j}^{(n+1)} - M_{1;k,j}^{(n-1)}}{2} U_{1;k,j}^{(n)}
+ \frac{M_{2;k,j}^{(n+1)} - M_{2;k,j}^{(n-1)}}{2} U_{2;k,j}^{(n)} \\
&\qquad \qquad+ \frac{U_{1;k,j}^{(n+1)}-U_{1;k,j}^{(n-1)}}{2} M_{1;k,j}^{(n)}
+ \frac{U_{2;k,j}^{(n+1)} - U_{2;k,j}^{(n-1)}}{2} M_{2;k,j}^{(n)}
\Big)   \Delta x  
\Delta y.
\end{align*}
We can now use in the last line the self-adjointness of $ ( 1 - \alpha^2
\delta^{<2>}_{jj} 
-  \alpha^2\delta^{<2>}_{kk})$ in a similar way to what done for the previous 
scheme, thus 
getting:
\begin{align*}
& \sum_{j=0}^{\cJ-1}\sum_{k=0}^{\cK-1}\Big(  
\frac{M_{1;k,j}^{(n+1)} - M_{1;k,j}^{(n-1)}}{2} U_{1;k,j}^{(n)}
+ \frac{M_{2;k,j}^{(n+1)} - M_{2;k,j}^{(n-1)}}{2} U_{2;k,j}^{(n)}
\Big)   \Delta x  
\Delta y.
\end{align*}
\end{proof}
\subsection{Proof of 
Lemma~\ref{lemma:first_equality_scheme3}} 
\label{sec:om_proofs_first_equality_scheme3}
\begin{proof}
\begin{align*}
&\sum_{j=0}^{\cJ-1}\sum_{k=0}^{\cK-1}H_{k,j}^{(n +\frac12 )}\Delta x 
\Delta y - \sum_{j=0}^{\cJ-1}\sum_{k=0}^{\cK-1}H_{k,j}^{(n- \frac12 )}\Delta x 
\Delta y \\
&=\frac14 \sum_{j=0}^{\cJ-1}\sum_{k=0}^{\cK-1}\Big(  
M_{1;k,j}^{(n+1)} U_{1;k,j}^{(n+1)} 
+ M_{1;k,j}^{(n)} U_{1;k,j}^{(n)} 
+ M_{2;k,j}^{(n+1)} U_{2;k,j}^{(n+1)}
+ M_{2;k,j}^{(n)} U_{2;k,j}^{(n)}\\
&\qquad \qquad -  M_{1;k,j}^{(n)} U_{1;k,j}^{(n)}
- M_{1;k,j}^{(n-1)} U_{1;k,j}^{(n-1)}
- M_{2;k,j}^{(n)} U_{2;k,j}^{(n)}
- M_{2;k,j}^{(n-1)} U_{2;k,j}^{(n-1)} \Big)   \Delta x  
\Delta y,
\end{align*}
which becomes
\begin{align*}
&\sum_{j=0}^{\cJ-1}\sum_{k=0}^{\cK-1}H_{k,j}^{(n +\frac12 )}\Delta x 
\Delta y - \sum_{j=0}^{\cJ-1}\sum_{k=0}^{\cK-1}H_{k,j}^{(n- \frac12 )}\Delta x 
\Delta y \\
&=\frac14 \sum_{j=0}^{\cJ-1}\sum_{k=0}^{\cK-1}\Big(  
M_{1;k,j}^{(n+1)} U_{1;k,j}^{(n+1)} 
+ M_{2;k,j}^{(n+1)} U_{2;k,j}^{(n+1)} \\
&\qquad \qquad \qquad - M_{1;k,j}^{(n-1)} U_{1;k,j}^{(n-1)}
- M_{2;k,j}^{(n-1)} U_{2;k,j}^{(n-1)} \Big)   \Delta x  
\Delta y.
\end{align*}
We add and subtract $ M_{1;k,j}^{(n-1)} U_{1;k,j}^{(n+1)}  + M_{2;k,j}^{(n-1)} 
U_{2;k,j}^{(n+1)}  $ so that we factorize the expression, getting:
\begin{align*}
&\frac12 \sum_{j=0}^{\cJ-1}\sum_{k=0}^{\cK-1}\Big(  
\frac{M_{1;k,j}^{(n+1)} - M_{1;k,j}^{(n-1)}}{2} U_{1;k,j}^{(n+1)}
+ \frac{M_{2;k,j}^{(n+1)} - M_{2;k,j}^{(n-1)}}{2} U_{2;k,j}^{(n+1)} \\
&\qquad \qquad+ \frac{U_{1;k,j}^{(n+1)}-U_{1;k,j}^{(n-1)}}{2} M_{1;k,j}^{(n-1)}
+ \frac{U_{2;k,j}^{(n+1)} - U_{2;k,j}^{(n-1)}}{2} M_{2;k,j}^{(n-1)}
\Big)   \Delta x  
\Delta y.
\end{align*}
We can now use in the last line the self-adjointness of $ ( 1 - \alpha^2
\delta^{<2>}_{jj} 
- \alpha^2 \delta^{<2>}_{kk})$ in a similar way to what done for the previous 
scheme, thus 
getting:
\begin{align*}
& \sum_{j=0}^{\cJ-1}\sum_{k=0}^{\cK-1}\Big(  
\frac{M_{1;k,j}^{(n+1)} - M_{1;k,j}^{(n-1)}}{2} 
\frac{ U_{1;k,j}^{(n+1)} + U_{1;k,j}^{(n-1)} }{2}\\
&\qquad \qquad+ \frac{M_{2;k,j}^{(n+1)} - M_{2;k,j}^{(n-1)}}{2}
\frac{ U_{2;k,j}^{(n+1)} + U_{2;k,j}^{(n-1)} }{2}
\Big)   \Delta x  
\Delta y.
\end{align*}
\end{proof}
\subsection{Proof of 
Theorem~\ref{thm:solvability_scheme1}} 
\label{sec:om_proofs_solvability_scheme1}
\begin{proof}
We make use of the shorthand notation $(1 - \alpha^2\delta^{<2>}_{k,j})$ 
instead of $(1 - \alpha^2\delta^{<2>}_{kk} - \alpha^2\delta^{<2>}_{jj} )$.
The scheme, written only in terms of $M$, is given by
\begin{tiny}
\begin{align*}
M_{1;k,j}^{(n+1)}  = M^{(n)}_{1;k,j} - \Delta t &\Big( \frac{ M_{1;k,j}^{(n)} +
M_{1;k,j}^{(n+1)} }{2} \cdot \delta^{<1>}_k \frac{
(1 - \alpha^2\delta^{<2>}_{k,j})^{-1} M_{1;k,j}^{(n)} 
+ (1 - \alpha^2\delta^{<2>}_{k,j})^{-1} M_{1;k,j}^{(n+1)}}{2} 
\\ 
&+ \frac{ M_{2;k,j}^{(n)} +M_{2;k,j}^{(n+1)}}{2} \cdot \delta^{<1>}_k 
\frac{ (1 - \alpha^2\delta^{<2>}_{k,j})^{-1}M_{2;k,j}^{(n)} +
(1 -\alpha^2 \delta^{<2>}_{k,j})^{-1}M_{2;k,j}^{(n+1)}}{2} \\
&+ \delta^{<1>}_k ( \frac{ M_{1;k,j}^{(n)} +
M_{1;k,j}^{(n+1)}}{2}  \cdot \frac{ (1 - \alpha^2
\delta^{<2>}_{k,j})^{-1}M_{1;k,j}^{(n)} + 
(1 -\alpha^2 \delta^{<2>}_{k,j})^{-1}M_{1;k,j}^{(n+1)}}{2} ) \\
&+ \delta^{<1>}_j ( \frac{ M_{1;k,j}^{(n)} +
M_{1;k,j}^{(n+1)}}{2} \cdot \frac{ (1 - 
\alpha^2\delta^{<2>}_{k,j})^{-1}M_{2;k,j}^{(n)} 
+ 
(1 - \alpha^2\delta^{<2>}_{k,j})^{-1}M_{2;k,j}^{(n+1)}}{2} )
\Big),
\end{align*}
\end{tiny}
and
\begin{tiny}
\begin{align*}
M_{2;k,j}^{(n+1)} = M^{(n)}_{2;k,j} - \Delta t &\Big( \frac{ M_{1;k,j}^{(n)} +
M_{1;k,j}^{(n+1)} }{2} \cdot \delta^{<1>}_j \frac{
(1 - \alpha^2\delta^{<2>}_{k,j})^{-1} M_{1;k,j}^{(n)} 
+ (1 - \alpha^2\delta^{<2>}_{k,j})^{-1} M_{1;k,j}^{(n+1)}}{2} 
\\ 
&+ \frac{ M_{2;k,j}^{(n)} +M_{2;k,j}^{(n+1)}}{2} \cdot \delta^{<1>}_j 
\frac{ (1 - \alpha^2\delta^{<2>}_{k,j})^{-1}M_{2;k,j}^{(n)} +
(1 -\alpha^2 \delta^{<2>}_{k,j})^{-1}M_{2;k,j}^{(n+1)}}{2} \\
&+ \delta^{<1>}_k ( \frac{ M_{2;k,j}^{(n)} +
M_{2;k,j}^{(n+1)}}{2}  \cdot \frac{ (1 - \alpha^2
\delta^{<2>}_{k,j})^{-1}M_{1;k,j}^{(n)} + 
(1 - \alpha^2\delta^{<2>}_{k,j})^{-1}M_{1;k,j}^{(n+1)}}{2} ) \\
&+ \delta^{<1>}_j ( \frac{ M_{2;k,j}^{(n)} +
M_{2;k,j}^{(n+1)}}{2} \cdot \frac{ (1 -\alpha^2 
\delta^{<2>}_{k,j})^{-1}M_{2;k,j}^{(n)} 
+ 
(1 - \alpha^2\delta^{<2>}_{k,j})^{-1}M_{2;k,j}^{(n+1)}}{2} )
\Big).
\end{align*}
\end{tiny}
We introduce a function
\begin{align*}
\Phi_{\ba} : \bR^{2\times \cK \times \cJ} \rightarrow \bR^{2\times \cK 
\times \cJ},
\end{align*}
where $\ba$ and $\bv$ belong to $\bR^{2\times \cK \times \cJ}$.
The function is defined in terms of the operator defined in 
Section~\ref{section:tools} as follows:
\begin{tiny}
\begin{align*}
(\Phi_{\ba}(\bv))_{1}  &:= \ba_{1} -  \frac{\Delta t}{4} \Big\{ \Big(  
\ba_{1} + \bv_{1}\Big) \cdot 
\Big( D^{<1>}_{x}  Q^{-1} \ba_{1} +  D^{<1>}_{x} Q^{-1} \bv_{1} \Big) \\
&+ \Big( \ba_{2} + \bv_{2}\Big) \cdot 
\Big( D^{<1>}_{x}  Q^{-1} \ba_{2} + D^{<1>}_{x}  Q^{-1}  \bv_{2} \Big) \\
&+ D^{<1>}_{x} \Big[ \Big( 
\ba_{1} + \bv_{1} \Big)  \cdot 
\Big(  Q^{-1} \ba_{1} +  Q^{-1} \bv_{1} \Big) \Big] 
+ D^{<1>}_{y} \Big[ \Big( \ba_{1} + \bv_{1}\Big) \cdot 
\Big(  Q^{-1} \ba_{2} +  Q^{-1} \bv_{2} \Big) \Big] \Big\},
\end{align*}
\begin{align*}
(\Phi_{\ba}(\bv))_{2}  &:= \ba_{2} -  \frac{\Delta t}{4} \Big\{ \Big(  
\ba_{1} + \bv_{1}\Big) \cdot 
\Big( D^{<1>}_{y}  Q^{-1} \ba_{1} +  D^{<1>}_{y} Q^{-1}  \bv_{1} \Big) \\
&+ \Big( \ba_{2} + \bv_{2}\Big) \cdot 
\Big( D^{<1>}_{y}  Q^{-1} \ba_{2} + D^{<1>}_{y}  Q^{-1}  \bv_{2} \Big) \\
&+ D^{<1>}_{x} \Big[ \Big( 
\ba_{2} + \bv_{2} \Big)  \cdot 
\Big(  Q^{-1} \ba_{1} +  Q^{-1} \bv_{1} \Big) \Big] 
+ D^{<1>}_{y} \Big[ \Big( 
\ba_{2} + \bv_{2}\Big) \cdot 
\Big(  Q^{-1} \ba_{2} +  Q^{-1} \bv_{2} \Big) \Big] \Big\}.
\end{align*}
\end{tiny}

We want to show first that the map $(\Phi_{\ba}(\bv))$ goes from a certain set 
to itself, under suitable assumptions, and then show that the map is a 
contraction over that particular set. This would imply that there exists a 
unique fixed-point, that is to say, a unique solution to our scheme.

We define the set $\Omega_a := \{ \bv \in \bR^{2 \times \cJ \times \cK} \colon 
\norm{}{\bv} \leq \rho r_a  \}$, where $r_a := \norm{}{\ba}$ and where the 
norms are the graph-norms, whenever required from the context.

We take norms:
\begin{tiny}
\begin{align*}
\norm{}{  (\Phi_{\ba}(\bv))_{1}}  &:= \norm{}{ \ba_{1}} +  \frac{\Delta t}{4} 
\Big\{ \norm{}{ (  
\ba_{1} + \bv_{1}) \cdot 
( D^{<1>}_{x}  Q^{-1} \ba_{1} +  D^{<1>}_{x} Q^{-1} \bv_{1} ) } \\
&+ \norm{}{ ( \ba_{2} + \bv_{2}) \cdot 
( D^{<1>}_{x}  Q^{-1} \ba_{2} + D^{<1>}_{x}  Q^{-1}  \bv_{2} )} \\
&+ \norm{}{ D^{<1>}_{x}} \norm{}{  ( 
\ba_{1} + \bv_{1} )  \cdot 
(  Q^{-1} \ba_{1} +  Q^{-1} \bv_{1} ) }
+ \norm{}{ D^{<1>}_{y}} \norm{}{  ( \ba_{1} + \bv_{1} ) \cdot 
(  Q^{-1} \ba_{2} +  Q^{-1} \bv_{2} ) }\Big\},
\end{align*}
\begin{align*}
\norm{}{ (\Phi_{\ba}(\bv))_{2}}  &:= \norm{}{ \ba_{2}} +  \frac{\Delta t}{4} 
\Big\{ \norm{}{  \Big(  
\ba_{1} + \bv_{1}\Big) \cdot 
\Big( D^{<1>}_{y}  Q^{-1} \ba_{1} +  D^{<1>}_{y} Q^{-1}  \bv_{1} \Big) }\\
&+ \norm{}{ \Big( \ba_{2} + \bv_{2}\Big) \cdot 
\Big( D^{<1>}_{y}  Q^{-1} \ba_{2} + D^{<1>}_{y}  Q^{-1}  \bv_{2} \Big) } \\
&+ \norm{}{ D^{<1>}_{x} } \norm{}{  \Big( 
\ba_{2} + \bv_{2} \Big)  \cdot 
\Big(  Q^{-1} \ba_{1} +  Q^{-1} \bv_{1} \Big) }
+ \norm{}{ D^{<1>}_{y} } \norm{}{  \Big( 
\ba_{2} + \bv_{2}\Big) \cdot 
\Big(  Q^{-1} \ba_{2} +  Q^{-1} \bv_{2} \Big)} \Big\}.
\end{align*}
\end{tiny}
We use elementary inequalities:
\begin{tiny}
\begin{align*}
\norm{}{  (\Phi_{\ba}(\bv))_{1}}  &\leq \norm{}{ \ba_{1}} +  \frac{\Delta 
t}{4}\frac{1}{\Delta x \sqrt{\Delta x \Delta y}} 
\Big[ 2\norm{}{ \ba_{1} + \bv_{1} }^2
+ \norm{}{ \ba_{2} + \bv_{2}}^2 \Big] + \frac{\Delta 
t}{4}\frac{1}{\Delta y \sqrt{\Delta x \Delta y}} 
 \norm{}{ \ba_{1} + \bv_{1}} \norm{}{  \ba_{2} + \bv_{2} }, \\
\norm{}{  (\Phi_{\ba}(\bv))_{2}}  &\leq \norm{}{ \ba_{2}} +  \frac{\Delta 
t}{4}\frac{1}{\Delta y \sqrt{\Delta x \Delta y}} 
\Big[ \norm{}{ \ba_{1} + \bv_{1} }^2
+ 2\norm{}{ \ba_{2} + \bv_{2}}^2 \Big] + \frac{\Delta 
t}{4}\frac{1}{\Delta x \sqrt{\Delta x \Delta y}} 
 \norm{}{ \ba_{1} + \bv_{1}} \norm{}{  \ba_{2} + \bv_{2} }.
\end{align*}
\end{tiny}
We make us of the following auxiliary notation:
\begin{align*}
C_x := \frac{\Delta 
t}{4}\frac{1}{\Delta x \sqrt{\Delta x \Delta y}}, \quad C_y := \frac{\Delta 
t}{4}\frac{1}{\Delta y \sqrt{\Delta x \Delta y}},
\end{align*}
and of the following auxiliary inequalities:
\begin{align*}
\norm{}{\ba_{1} + \bv_{1}}\norm{}{\ba_{2} + \bv_{2}} &\leq \norm{}{\ba}^2 + 
\norm{}{\bv}^2,\\
\norm{}{ \ba_{1} + \bv_{1}}^2 + 2\norm{}{ \ba_{2} + \bv_{2}}^2 &\leq  
4\norm{}{\ba}^2 + 4\norm{}{\bv}^2,
\end{align*}
to obtain the following simplified estimate:
\begin{align*}
\norm{}{(\Phi_{\ba}(\bv))_{1}}  &\leq \norm{}{\ba_{1}} + 4 C_x
 \Big( \norm{}{\ba}^2 + \norm{}{\bv}^2 \Big) + 
C_y \Big( \norm{}{\ba}^2 + \norm{}{\bv}^2 \Big)  \\
\norm{}{(\Phi_{\ba}(\bv))_{2}}  &\leq \norm{}{\ba_{2}} + 4 C_y
 \Big( \norm{}{\ba}^2 + \norm{}{\bv}^2 \Big) + 
C_x \Big( \norm{}{\ba}^2 + \norm{}{\bv}^2 \Big),
\end{align*}
that is to say:
\begin{align*}
\norm{}{(\Phi_{\ba}(\bv))_{1}}  &\leq \norm{}{\ba_{1}} + (4 C_x + C_y)
 \Big( \norm{}{\ba}^2 + \norm{}{\bv}^2 \Big) \\
\norm{}{(\Phi_{\ba}(\bv))_{2}}  &\leq \norm{}{\ba_{2}} + (4 C_y +C_x)
 \Big( \norm{}{\ba}^2 + \norm{}{\bv}^2 \Big).
 \end{align*}
We evaluate now $\norm{}{(\Phi_{\ba}(\bv))} $:
\begin{tiny}
\begin{align*}
&\norm{}{(\Phi_{\ba}(\bv))}^2 = \norm{}{(\Phi_{\ba}(\bv))_{1}}^2 + 
\norm{}{(\Phi_{\ba}(\bv))_{2}}^2 \\
&= \Big[ \norm{}{\ba_{1}} + (4 C_x + C_y)
 \Big( \norm{}{\ba}^2 + \norm{}{\bv}^2 \Big)\Big]^2 + \Big[ 
\norm{}{\ba_{2}} 
+ (4 C_y +C_x)
 \Big( \norm{}{\ba}^2 + \norm{}{\bv}^2 \Big) \Big]^2\\
&\leq 2 \Big[  \norm{}{\ba_{1}}^2 + (4 C_x + C_y)^2
 \Big( \norm{}{\ba}^2 + \norm{}{\bv}^2 \Big) ^2 +  \norm{}{\ba_{2}}^2
+ (4 C_y +C_x)^2
 \Big( \norm{}{\ba}^2 + \norm{}{\bv}^2 \Big)^2 \Big] \\
&= 2 \Big\{ \norm{}{\ba}^2 + \Big[ (4 C_x + C_y)^2 + (C_x + 4C_y)^2\Big]
 \Big( \norm{}{\ba}^2 + \norm{}{\bv}^2 \Big) ^2\Big\}\\
 &= 2 \Big\{ \norm{}{\ba}^2 + \Big[ 17 C_x^2 + 17C_y^2 + 
16C_x C_y\Big]\Big( \norm{}{\ba}^2 + 
\norm{}{\bv}^2 \Big) ^2\Big\}\\
 &\leq 2 \Big\{ \norm{}{\ba}^2 + (25 C_x^2 + 25 C_y^2) \Big( \norm{}{\ba}^2 + 
\norm{}{\bv}^2 \Big) ^2\Big\}\\
 &\leq 2 \Big\{ \norm{}{\ba}^2 + 2(25 C_x^2 + 25 C_y^2) \Big( \norm{}{\ba}^4 
+ 
\norm{}{\bv}^4 \Big) \Big\}\\
&\leq 2 \Big\{  r_a^2 + 50 (C_x^2 + C_y^2)
 \Big( r_a^4 + \rho^4 r_a^4 \Big) \Big\}.
\end{align*}
\end{tiny}
We can now see that the condition to have $\Phi_{\ba} : \Omega_a \rightarrow 
\Omega_a$, is satisfied if the following holds:
\begin{align*}
2 r_a^2 + 100 (C_x^2 + C_y^2)
 \Big( r_a^4 + \rho^4 r_a^4 \Big) \leq \rho^2 r_a^2.
\end{align*}
This gives us a first condition to fulfil, namely:
\begin{align}\label{eq:2d_conv_condition_part1}
(C_x^2 + C_y^2) \leq \frac{ (\rho^2 -2) r_a^2}{100(r_a^4 + \rho^4 r_a^4)}
\end{align}
It follows immediately that whatever $\rho$ we choose, it has to be at least 
greater than $\sqrt{2}$.

We now have to investigate the difference 
$\norm{}{\Phi_{\ba}(v) - \Phi_{\ba}(w)}$ to find out what kind of 
condition it takes to have a contraction onto $\Omega_a$ when both $v$ and $w$ 
belongs to $\Omega_a$.
\begin{tiny}
\begin{align*}
(\Phi_{\ba}(\bv))_{1} - (\Phi_{\ba}(\bw))_{1}  &:= \frac{\Delta 
t}{4} \Big\{ 
\ba_{1} \cdot D^{<1>}_{x}  Q^{-1} ( \bw_1 -\bv_1) 
+  \ba_{2} \cdot  D^{<1>}_{x}  Q^{-1} ( \bw_2 -\bv_2)\\
&\qquad + ( \bw_1 -\bv_1)  \cdot  D^{<1>}_{x}  Q^{-1} \ba_{1}
+  ( \bw_2 -\bv_2) \cdot  D^{<1>}_{x}  Q^{-1}  \ba_{2} \\
&\qquad  +\frac12 \Big[  ( \bw_1 + \bv_1) \cdot  D^{<1>}_{x}  Q^{-1} ( \bw_1 - 
\bv_1) 
+( \bw_1 - \bv_1) \cdot  D^{<1>}_{x}  Q^{-1} ( \bw_1 + \bv_1) \Big]\\
&\qquad +\frac12 \Big[  ( \bw_2 + \bv_2) \cdot  D^{<1>}_{x}  Q^{-1} ( \bw_2 - 
\bv_2) 
+( \bw_2 - \bv_2) \cdot  D^{<1>}_{x}  Q^{-1} ( \bw_2 + \bv_2) \Big] \\
&\qquad+ D^{<1>}_{x} \Big[ \ba_1 \cdot Q^{-1}(\bw_1 - \bv_1)
+  (\bw_1 - \bv_1)\cdot Q^{-1}\ba_1\\
&\qquad + \frac12 \Big( (\bw_1 + \bv_1)\cdot Q^{-1} (\bw_1 - \bv_1)
+ (\bw_1 - \bv_1)\cdot Q^{-1} (\bw_1 + \bv_1) \Big) \Big] \\
&\qquad+ D^{<1>}_{y} \Big[ \ba_1 \cdot Q^{-1}(\bw_2 - \bv_2)
+  (\bw_1 - \bv_1)\cdot Q^{-1}\ba_2 \\
&\qquad+ \frac12 \Big( (\bw_1 + \bv_1)\cdot Q^{-1} (\bw_2 - \bv_2)
+ (\bw_1 - \bv_1)\cdot Q^{-1} (\bw_2 + \bv_2) \Big) \Big]
\Big\},\\
(\Phi_{\ba}(\bv))_{2} - (\Phi_{\ba}(\bw))_{2}  &:= \frac{\Delta 
t}{4} \Big\{ 
\ba_{1} \cdot D^{<1>}_{y}  Q^{-1} ( \bw_1 -\bv_1) 
+  \ba_{2} \cdot  D^{<1>}_{y}  Q^{-1} ( \bw_2 -\bv_2)\\
&\qquad + ( \bw_1 -\bv_1)  \cdot  D^{<1>}_{y}  Q^{-1} \ba_{1}
+  ( \bw_2 -\bv_2) \cdot  D^{<1>}_{y}  Q^{-1}  \ba_{2} \\
&\qquad  +\frac12 \Big[  ( \bw_1 + \bv_1) \cdot  D^{<1>}_{y}  Q^{-1} ( \bw_1 - 
\bv_1) 
+( \bw_1 - \bv_1) \cdot  D^{<1>}_{y}  Q^{-1} ( \bw_1 + \bv_1) \Big]\\
&\qquad +\frac12 \Big[  ( \bw_2 + \bv_2) \cdot  D^{<1>}_{y}  Q^{-1} ( \bw_2 - 
\bv_2) 
+( \bw_2 - \bv_2) \cdot  D^{<1>}_{y}  Q^{-1} ( \bw_2 + \bv_2) \Big] \\
&\qquad+ D^{<1>}_{x} \Big[ \ba_2 \cdot Q^{-1}(\bw_1 - \bv_1)
+  (\bw_2 - \bv_2)\cdot Q^{-1}\ba_1\\
&\qquad + \frac12 \Big( (\bw_2 + \bv_2)\cdot Q^{-1} (\bw_1 - \bv_1)
+ (\bw_2 - \bv_2)\cdot Q^{-1} (\bw_1 + \bv_1) \Big) \Big] \\
&\qquad+ D^{<1>}_{y} \Big[ \ba_2 \cdot Q^{-1}(\bw_2 - \bv_2)
+  (\bw_2 - \bv_2)\cdot Q^{-1}\ba_2 \\
&\qquad+ \frac12 \Big( (\bw_2 + \bv_2)\cdot Q^{-1} (\bw_2 - \bv_2)
+ (\bw_2 - \bv_2)\cdot Q^{-1} (\bw_2 + \bv_2) \Big) \Big]
\Big\}.
\end{align*}
\end{tiny}
We take norms and start estimating
\begin{tiny}
\begin{align*}
\norm{}{(\Phi_{\ba}(\bv))_{1} - (\Phi_{\ba}(\bw))_{1}}  &\leq C_x \Big[
\norm{}{  \ba_{1}}\norm{}{\bw_1 -\bv_1}
+  \norm{}{\ba_{2}} \norm{}{ \bw_2 -\bv_2}\\
&\qquad + \norm{}{ \bw_1 -\bv_1} \norm{}{\ba_{1}}
+  \norm{}{ \bw_2 -\bv_2} \norm{}{\ba_{2}} \\
&\qquad  +\frac12 \Big(  \norm{}{ \bw_1 + \bv_1} \norm{}{\bw_1 - 
\bv_1}
+ \norm{}{\bw_1 - \bv_1} \norm{}{ \bw_1 + \bv_1} \Big)\\
&\qquad +\frac12 \Big(  \norm{}{ \bw_2 + \bv_2 } \norm{}{ \bw_2 - 
\bv_2} + \norm{}{ \bw_2 - \bv_2}\norm{}{ \bw_2 + \bv_2}
\Big) \\
&\qquad+ \norm{}{\ba_1} \norm{}{ \bw_1 - \bv_1}
+  \norm{}{\bw_1 - \bv_1} \norm{}{\ba_1} \\
&\qquad + \frac12 \Big( \norm{}{\bw_1 + \bv_1} \norm{}{\bw_1 - \bv_1}
+ \norm{}{\bw_1 - \bv_1} \norm{}{\bw_1 + \bv_1} \Big) \Big]\\
&\qquad+ C_y \Big[ \norm{}{\ba_1}\norm{}{\bw_2 - \bv_2}
+  \norm{}{\bw_1 - \bv_1} \norm{}{\ba_2} \\
&\qquad+ \frac12 \Big( \norm{}{\bw_1 + \bv_1}\norm{}{\bw_2 - 
\bv_2} + \norm{}{\bw_1 - \bv_1} \norm{}{\bw_2 + \bv_2} \Big) \Big], \\
\norm{}{(\Phi_{\ba}(\bv))_{2} - (\Phi_{\ba}(\bw))_{2}}  &\leq C_y \Big[ 
\norm{}{ \ba_{1}}\norm{}{ \bw_1 -\bv_1}
+ \norm{}{ \ba_{2}}\norm{}{ \bw_2 -\bv_2}\\
&\qquad + \norm{}{ \bw_1 -\bv_1}\norm{}{\ba_{1}}
+  \norm{}{ \bw_2 -\bv_2}\norm{}{\ba_{2}} \\
&\qquad  +\frac12 \Big(  \norm{}{ \bw_1 + \bv_1}\norm{}{
\bw_1 - \bv_1}
+ \norm{}{ \bw_1 - \bv_1}\norm{}{ \bw_1 + \bv_1}  \Big)\\
&\qquad +\frac12 \Big(  \norm{}{ \bw_2 + \bv_2}\norm{}{ \bw_2 - 
\bv_2}
+ \norm{}{ \bw_2 - \bv_2}\norm{}{ \bw_2 + \bv_2} \Big) \\
&\qquad + \norm{}{\ba_2} \norm{}{\bw_2 - \bv_2}
+  \norm{}{\bw_2 - \bv_2}\norm{}{\ba_2} \\
&\qquad+ \frac12 \Big( \norm{}{\bw_2 + \bv_2}\norm{}{\bw_2 - \bv_2}
+ \norm{}{\bw_2 - \bv_2} \norm{}{\bw_2 + \bv_2} \Big) \Big]\\
&\qquad+ C_x \Big[ \norm{}{\ba_2}\norm{}{\bw_1 - \bv_1}
+ \norm{}{\bw_2 - \bv_2}\norm{}{\ba_1}\\
&\qquad + \frac12 \Big( \norm{}{\bw_2 + \bv_2}\norm{}{\bw_1 - \bv_1}
+ \norm{}{\bw_2 - \bv_2}\norm{}{\bw_1 + \bv_1} \Big) \Big] .
\end{align*}
\end{tiny}
This leads to
\begin{tiny}
\begin{align*}
\norm{}{(\Phi_{\ba}(\bv))_{1} - (\Phi_{\ba}(\bw))_{1}}  &\leq
\Big[ 2 C_x \Big( 2\norm{}{  \ba_{1}} + \norm{}{\bw_1  + \bv_1} \Big)
+ \frac{C_y}{2}\Big( 2\norm{}{  \ba_{2}} + \norm{}{\bw_2  + 
\bv_2}\Big)\Big]\norm{}{ \bw_1 -\bv_1}\\
&\quad + \Big[ C_x \Big( 2\norm{}{  \ba_{2}} + \norm{}{\bw_2  + \bv_2}\Big)
+ \frac{C_y}{2}\Big( 2\norm{}{  \ba_{1}} + \norm{}{\bw_1  + 
\bv_1}\Big)\Big]\norm{}{ \bw_2 -\bv_2},\\
\norm{}{(\Phi_{\ba}(\bv))_{2} - (\Phi_{\ba}(\bw))_{2}}  &\leq
\Big[ 2 C_y \Big( 2\norm{}{  \ba_{2}} + \norm{}{\bw_2  + \bv_2}\Big)
+ \frac{C_x}{2}\Big( 2\norm{}{  \ba_{1}} + \norm{}{\bw_1  + 
\bv_1}\Big)\Big]\norm{}{ \bw_2 -\bv_2}\\
&\quad + \Big[ C_y \Big( 2\norm{}{  \ba_{1}} + \norm{}{\bw_1  + \bv_1}\Big)
+ \frac{C_x}{2}\Big( 2\norm{}{  \ba_{2}} + \norm{}{\bw_2  + 
\bv_2}\Big)\Big]\norm{}{ \bw_1 -\bv_1}.
\end{align*}
\end{tiny}
We square the quantities above and use $(a+b)^2 \leq 2(a^2+b^2)$ once:
\begin{tiny}
\begin{align*}
\norm{}{(\Phi_{\ba}(\bv))_{1} - (\Phi_{\ba}(\bw))_{1}}^2  &\leq
2\Big[ 2 C_x \Big( 2\norm{}{  \ba_{1}} + \norm{}{\bw_1  + \bv_1} \Big)
+ \frac{C_y}{2}\Big( 2\norm{}{  \ba_{2}} + \norm{}{\bw_2  + 
\bv_2}\Big)\Big]^2\norm{}{ \bw_1 -\bv_1}^2\\
&\quad + 2\Big[ C_x \Big( 2\norm{}{  \ba_{2}} + \norm{}{\bw_2  + \bv_2}\Big)
+ \frac{C_y}{2}\Big( 2\norm{}{  \ba_{1}} + \norm{}{\bw_1  + 
\bv_1}\Big)\Big]^2\norm{}{ \bw_2 -\bv_2}^2,\\
\norm{}{(\Phi_{\ba}(\bv))_{2} - (\Phi_{\ba}(\bw))_{2}}^2  &\leq
2\Big[ 2 C_y \Big( 2\norm{}{  \ba_{2}} + \norm{}{\bw_2  + \bv_2}\Big)
+ \frac{C_x}{2}\Big( 2\norm{}{  \ba_{1}} + \norm{}{\bw_1  + 
\bv_1}\Big)\Big]^2\norm{}{ \bw_2 -\bv_2}^2\\
&\quad + 2\Big[ C_y \Big( 2\norm{}{  \ba_{1}} + \norm{}{\bw_1  + \bv_1}\Big)
+ \frac{C_x}{2}\Big( 2\norm{}{  \ba_{2}} + \norm{}{\bw_2  + 
\bv_2}\Big)\Big]^2\norm{}{ \bw_1 -\bv_1}^2,
\end{align*}
\end{tiny}
and once more:
\begin{tiny}
\begin{align*}
\norm{}{(\Phi_{\ba}(\bv))_{1} - (\Phi_{\ba}(\bw))_{1}}^2  &\leq
4\Big[ 4 C_x^2 \Big( 2\norm{}{  \ba_{1}} + \norm{}{\bw_1  + \bv_1} \Big)^2
+ \frac{C_y^2}{4}\Big( 2\norm{}{  \ba_{2}} + \norm{}{\bw_2  + 
\bv_2}\Big)^2\Big]\norm{}{ \bw_1 -\bv_1}^2\\
&\quad + 4\Big[ C_x^2 \Big( 2\norm{}{  \ba_{2}} + \norm{}{\bw_2  + \bv_2}\Big)^2
+ \frac{C_y^2}{4}\Big( 2\norm{}{  \ba_{1}} + \norm{}{\bw_1  + 
\bv_1}\Big)^2\Big]\norm{}{ \bw_2 -\bv_2}^2,\\
\norm{}{(\Phi_{\ba}(\bv))_{2} - (\Phi_{\ba}(\bw))_{2}}^2  &\leq
4\Big[ 4 C_y^2 \Big( 2\norm{}{  \ba_{2}} + \norm{}{\bw_2  + \bv_2}\Big)^2
+ \frac{C_x^2}{4}\Big( 2\norm{}{  \ba_{1}} + \norm{}{\bw_1  + 
\bv_1}\Big)^2\Big]\norm{}{ \bw_2 -\bv_2}^2\\
&\quad + 4\Big[ C_y^2 \Big( 2\norm{}{  \ba_{1}} + \norm{}{\bw_1  + \bv_1}\Big)^2
+ \frac{C_x^2}{4}\Big( 2\norm{}{  \ba_{2}} + \norm{}{\bw_2  + 
\bv_2}\Big)^2\Big]\norm{}{ \bw_1 -\bv_1}^2,
\end{align*}
\end{tiny}
and one last time
\begin{tiny}
\begin{align*}
\norm{}{(\Phi_{\ba}(\bv))_{1} - (\Phi_{\ba}(\bw))_{1}}^2  &\leq
4\Big[ 8 C_x^2 \Big( 4\norm{}{  \ba_{1}}^2 + \norm{}{\bw_1  + \bv_1}^2 \Big)
+ \frac{C_y^2}{2}\Big( 4\norm{}{  \ba_{2}}^2 + \norm{}{\bw_2  + 
\bv_2}^2 \Big)\Big]\norm{}{ \bw_1 -\bv_1}^2\\
&\quad + 4\Big[ 2C_x^2 \Big( 4\norm{}{  \ba_{2}}^2 + \norm{}{\bw_2  + 
\bv_2}^2\Big)
+ \frac{C_y^2}{2}\Big( 4\norm{}{  \ba_{1}}^2 + \norm{}{\bw_1  + 
\bv_1}^2\Big)\Big]\norm{}{ \bw_2 -\bv_2}^2,\\
\norm{}{(\Phi_{\ba}(\bv))_{2} - (\Phi_{\ba}(\bw))_{2}}^2  &\leq
4\Big[ 8 C_y^2 \Big( 4\norm{}{  \ba_{2}}^2 + \norm{}{\bw_2  + \bv_2}^2\Big)
+ \frac{C_x^2}{2}\Big( 4\norm{}{  \ba_{1}}^2 + \norm{}{\bw_1  + 
\bv_1}^2\Big)\Big]\norm{}{ \bw_2 -\bv_2}^2\\
&\quad + 4\Big[ 2C_y^2 \Big( 4\norm{}{  \ba_{1}}^2 + \norm{}{\bw_1  + 
\bv_1}^2\Big)
+ \frac{C_x^2}{2}\Big( 4\norm{}{  \ba_{2}}^2 + \norm{}{\bw_2  + 
\bv_2}^2\Big)\Big]\norm{}{ \bw_1 -\bv_1}^2.
\end{align*}
\end{tiny}
We can now sum up the two last inequalities that we have obtained. It follows 
that:
\begin{tiny}
\begin{align*}
\norm{}{\Phi_{\ba}(\bv) - \Phi_{\ba}(\bw)}^2  &= 
\norm{}{(\Phi_{\ba}(\bv))_{1} - (\Phi_{\ba}(\bw))_{1}}^2  + 
\norm{}{(\Phi_{\ba}(\bv))_{2} - (\Phi_{\ba}(\bw))_{2}}^2 \\
&\leq 4\Big[ (32 C_x^2 + 8 C_y^2) \norm{}{  \ba_{1}}^2 
+ (2 C_x^2 + 2 C_y^2) \norm{}{  \ba_{2}}^2 \\
&+ (8 C_x^2 + 2 C_y^2) \norm{}{ \bw_1 +\bv_1}^2
+  (\frac12 C_x^2 +  \frac12 C_y^2) \norm{}{ \bw_2 +\bv_2}^2 \Big] \norm{}{ 
\bw_1 -\bv_1}^2\\
&+ 4\Big[ (32 C_y^2 + 8 C_x^2) \norm{}{  \ba_{2}}^2 
+ (2 C_x^2 + 2 C_y^2) \norm{}{  \ba_{1}}^2 \\
&+ (8 C_y^2 + 2 C_x^2) \norm{}{ \bw_2 +\bv_2}^2
+  (\frac12 C_x^2 +  \frac12 C_y^2) \norm{}{ \bw_1 +\bv_1}^2 \Big] \norm{}{ 
\bw_2 -\bv_2}^2.
\end{align*}
\end{tiny}
By means of a gross factorization, we can further estimate the expression as 
\begin{tiny}
\begin{align*}
\norm{}{\Phi_{\ba}(\bv) - \Phi_{\ba}(\bw)}^2 &\leq 4 \Big[ ( 32C_x^2 + 
8C_y^2)\norm{}{  \ba }^2  +  (8 C_x^2 + 2 C_y^2) \norm{}{ \bw +\bv}^2 \Big] 
\norm{}{ \bw_1 -\bv_1}^2\\
&+ 4\Big[ ( 32C_y^2 + 8C_x^2)\norm{}{  \ba }^2  +  (8 C_y^2 + 2 C_x^2) 
\norm{}{ \bw +\bv}^2 \Big] \norm{}{ \bw_2 -\bv_2}^2.
\end{align*}
\end{tiny}
We now use the estimates we have on $\bw$, $\bv$, $\ba$, to get that
\begin{tiny}
\begin{align*}
\norm{}{ \Phi_{\ba}(\bv) - \Phi_{\ba}(\bw) }^2 &\leq 4 \Big[ ( 32C_x^2 + 
8C_y^2)r_a^2  +  (8 C_x^2 + 2 C_y^2) 4\rho^2 r_a^2 \Big] 
\norm{}{ \bw_1 -\bv_1}^2\\
&+ 4\Big[ ( 32C_y^2 + 8C_x^2)r_a^2  +  (8 C_y^2 + 2 C_x^2) 
4\rho^2 r_a^2 \Big] \norm{}{ \bw_2 -\bv_2}^2.
\end{align*}
\end{tiny}
By means of another gross factorization we finally achieve that
\begin{align}\label{eq:2d_conv_condition_part2}
\norm{}{ \Phi_{\ba}(\bv) - \Phi_{\ba}(\bw) }^2 &\leq 128(C_x^2 + 
C_y^2)r_a^2(1 + \rho^2) \norm{}{ \bw -\bv}^2.
\end{align}
If we insert \eqref{eq:2d_conv_condition_part1} in 
\eqref{eq:2d_conv_condition_part2}  
\begin{align*}
\norm{}{\Phi_{\ba}(\bv) - \Phi_{\ba}(\bw)}^2 & \leq \frac{128}{100}\frac{ 
(\rho^2 -2)(1 + \rho^2)}{(1 + \rho^4 )}\norm{}{ \bw -\bv}^2.
\end{align*}
In order to have a contraction we have to require that
\begin{align*}
 \frac{128}{100}\frac{ 
(\rho^2 -2)(1 + \rho^2)}{(1 + \rho^4 )}\leq 1,
\end{align*}
which is satisfied for any $\rho \leq \sqrt{\frac{32+ \sqrt{3516}}{14}} 
\approx 2.55$. 

We therefore choose $\rho$ to be equal to $\sqrt{2+\sqrt{5}} \approx 2.058$, so 
that we obtain the largest admissible right-hand side in condition 
\eqref{eq:2d_conv_condition_part1}, which now reads
\begin{align}\label{eq:2d_conv_condition_final}
(C_x^2 + C_y^2) \leq \frac{\sqrt5 - 2}{200 r_a^2}.
\end{align}
The initial claim follows by writing explicitly the sum $C_x^2 + C_y^2$.
\end{proof}

\end{tiny}

\bibliographystyle{alpha}
\bibliography{biblionew}
\end{document}